\documentclass{jmd}
\usepackage{amscd,amssymb}

\issueinfo[A]{2}{1}{139}

\newcommand{\noz}{m}
\newcommand{\noi}{n}

\newcommand{\ind}{\mathit{ind}}

\newcommand\ZZ{{\mathbb Z}/2{\mathbb Z}}
\newcommand\Zz{{\mathbb Z}_2}
\newcommand\Z[1]{{\mathbb Z}^{#1}}

\newcommand\R[1]{{\mathbb R}^{#1}}
\newcommand\C[1]{{\mathbb C}^{#1}}

\newcommand{\cH}{{\mathcal H}}
\newcommand{\cM}{{\mathcal M}}

\newcommand{\cQ}{{\mathcalplain Q}}

\newlength{\halfbls}

\newtheorem*{NNTheorem}{Theorem}

\newtheorem{Proposition}{Proposition}
\newtheorem{Lemma}{Lemma}

\newtheorem{Conjecture}{Conjecture}

\theoremstyle{definition}
\newtheorem*{Definition}{Definition}
\newtheorem{Convention}{Convention}

\theoremstyle{remark}

\newtheorem*{NNRemark}{Remark}
\newtheorem{Example}{Example}
\newtheorem*{NNExample}{Example}


\date{October 1, 2007}
\begin{document}

\title[Representatives of Abelian and quadratic differentials]
{Explicit Jenkins--Strebel representatives of all strata of Abelian and quadratic differentials}

\author{Anton Zorich}
\address{IRMAR, Universit\'e Rennes 1, Campus de Beaulieu, 35042 Rennes, France}
\email{Anton.Zorich@univ-rennes1.fr}
\urladdr{http://perso.univ-rennes1.fr/anton.zorich}
\thanks{The work was partially supported by ANR grant BLAN06-3\_138280, ``Dynamics in Teichm\"uller space''}

\dedicatory{To G. A. Margulis on his 60th birthday}

\begin{abstract}
Moduli spaces of Abelian and quadratic differentials are stratified
by multiplicities of zeroes; connected components of the strata correspond to
ergodic components of the Teichm\"uller geodesic flow.
It is known that the strata are not necessarily
connected; the connected components were recently classified by
{M. Kontsevich} and the author and by {E. Lanneau}. The strata can be also viewed as families
of flat metrics with conical singularities and with $\ZZ$-holonomy.

For every connected component of each stratum of Abelian and quadratic differentials we construct an explicit representative which is a Jenkins--Strebel differential with a single cylinder. By an elementary variation of this construction we represent almost every Abelian (quadratic) differential in the corresponding connected component of the stratum as a polygon with identified pairs of edges, where combinatorics of identifications is explicitly described.

Specifically, the combinatorics is expressed in terms of a generalized permutation. For any component of any stratum of Abelian and quadratic differentials we  construct a generalized permutation in the corresponding extended Rauzy class.
\end{abstract}
\maketitle

\setcounter{tocdepth}{1}
\tableofcontents

\section{Introduction}

\subsection{Flat surfaces versus Abelian and quadratic differentials}

Consider a collection of vectors $ \vec v_1,  \dots, \vec v_\noi$
in $\R{2}$ and construct from  these  vectors a broken line in  a
natural way:  the $j$th edge of the broken  line is represented by
the vector $\vec{v}_j$. Construct another broken line starting at
the same point as  the initial one by taking the same  vectors in
the order  $\vec{v}_{\pi^{-1}(1)}, \dots, \vec{v}_{\pi^{-1}(\noi)}$,  where $\pi^{-1}$ is a permutation of\/ $\noi$ elements. By construction, the two  broken  lines share the same endpoints;  suppose  that  they bound   a   polygon   like in   Figure \ref{zorich:fig:suspension}. Identifying the pairs  of sides corresponding to the same vectors $\vec{v}_j$, $j=1,  \dots,  \noi$,  by  parallel  translations we obtain a surface endowed with  a  flat  metric.

The  flat metric is nonsingular outside  of  a finite number of cone-type singularities  corresponding  to the vertices of the polygon.  By construction, the flat metric has  trivial  holonomy:  a  parallel transport of a  vector  along a closed path  does  not change the direction   (and   length)  of  the  vector.  This  implies,   in particular, that all cone angles at the singularities are    integer  multiples  of\/ $2\pi$.

\begin{figure}[htb]
\includegraphics{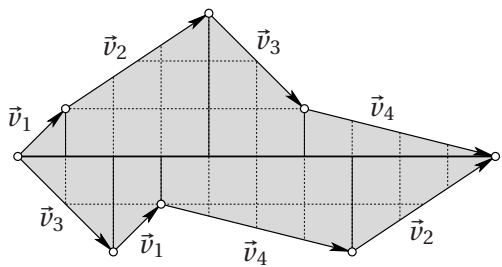}
\begin{picture}(20,2)(190,17)
\put(103,-32){$\vec v_1$}
\put(139,-5){$\vec v_2$}
\put(197,-3){$\vec v_3$}
\put(240,-28){$\vec v_4$}
\put(115,-70){$\vec v_3$}
\put(153,-80){$\vec v_1$}
\put(192,-83){$\vec v_4$}
\put(255,-75){$\vec v_2$}
\end{picture}
\vspace{105bp}
\caption{
\label{zorich:fig:suspension}
Identifying  corresponding  pairs  of  sides of this  polygon  by
parallel translations  we obtain a surface  of genus two. The  flat
metric has trivial holonomy; it has
a single  conical  singularity  with  cone angle $6\pi$.
}
\end{figure}

The  polygon  in  our  construction depends continuously  on  the
vectors $\vec{v}_j$.  This  means that the combinatorial geometry
of the  resulting flat surface  (the genus $g$, the number $\noz$
and the types of the resulting conical singularities) does not change
under small deformations of the vectors  $\vec{v}_j$. This allows us
to consider  a flat surface  as an  element of a  family of  flat
surfaces  sharing  common combinatorial geometry; here we do  not
distinguish isometric flat surfaces.

Choosing a tangent vector at some point of a surface we can transport this
vector to any other point. When the surface has trivial holonomy the result
does not depend on the path, so any direction is globally defined on the
surface. It is convenient to include the choice of direction in the
definition of a flat structure. In particular, we want to distinguish the
flat structure represented by the polygon in Figure
\ref{zorich:fig:suspension} and the one represented by the same polygon
rotated by an angle which is not a multiple of\/ $2\pi$.

Consider a natural coordinate $z$ in the complex plane. In this
coordinate parallel translations which we use to identify the
sides of the polygon in Figure \ref{zorich:fig:suspension} are represented as $z'=z+\mathit{const}$.
Since this correspondence  is holomorphic, it means that our flat
surface  $S$  with punctured conical points inherits a complex
structure. It is easy to check that the complex structure extends
to the punctured points.  Consider  now a holomorphic 1-form $d\!z$
in  the complex plane.  When  we  pass  to the  surface  $S$  the
coordinate  $z$  is not globally defined anymore. However,  since
the changes of  local coordinates are defined as $z'=z+\mathit{const}$, we
see that $d\!z=d\!z'$. Thus, the  holomorphic  1-form  $d\!z$ on $\C{}$
defines  a  holomorphic  1-form  $\omega$ on $S$ which  in  local
coordinates has  the form $\omega=d\!z$. It  is easy to  check that
the form $\omega$ has zeroes exactly at those points of\/ $S$ where
the flat structure has conical singularities.

Conversely, one can
show  that  a pair (Riemann surface, holomorphic  1-form)
uniquely defines a flat structure of the type described above.

In an appropriate local coordinate $w$ a holomorphic 1-form can be
represented in a neighborhood of a zero as $w^d d\!w$, where the power $d$ is
called the \emph{degree} of the zero. The form $\omega$ has a zero of degree
$d$ at a conical point with a cone angle $2\pi(d+1)$. The sum of degrees
$d_1+\dots+d_\noz$ of zeroes of a holomorphic 1-form on a Riemann surface
of genus $g$ equals $2g-2$. The moduli space $\cH_g$ of pairs (complex
structure, holomorphic 1-form) is a $\C{g}$-vector bundle over the moduli
space $\cM_g$ of complex structures. The space $\cH_g$ can be naturally
decomposed into strata $\cH(d_1,\dots,d_\noz)$ enumerated by
unordered\footnote{Not quite unordered: in this paper, the elements $d_1$
and $d_\noz$ (the first one and the last one) play a distinguished role,
see Convention \ref{conv:zeroes:at:the:end} in the next section} partitions
of the number $2g-2$ in a collection of positive integers
$2g-2=d_1+\dots+d_\noz$. Any holomorphic 1-form corresponding to a fixed
stratum $\cH(d_1,\dots,d_\noz)$ has exactly $\noz$ zeroes, and
$d_1,\dots,d_\noz$ are the degrees of zeroes. Note that an individual
stratum $\cH(d_1,\dots,d_\noz)$ in general does not form a fiber bundle
over $\cM_g$.

More details on the geometry and topology of the strata (in particular
studies of a natural Lebesgue measure and ergodicity of the Teichm\"uller
geodesic flow) can be found in the fundamental papers of H. Masur
\cite{M82} and W. Veech \cite{V82,V86,V90}. The bibliography on this
subject currently contains hundreds of papers.

\begin{figure}[ht]
\includegraphics{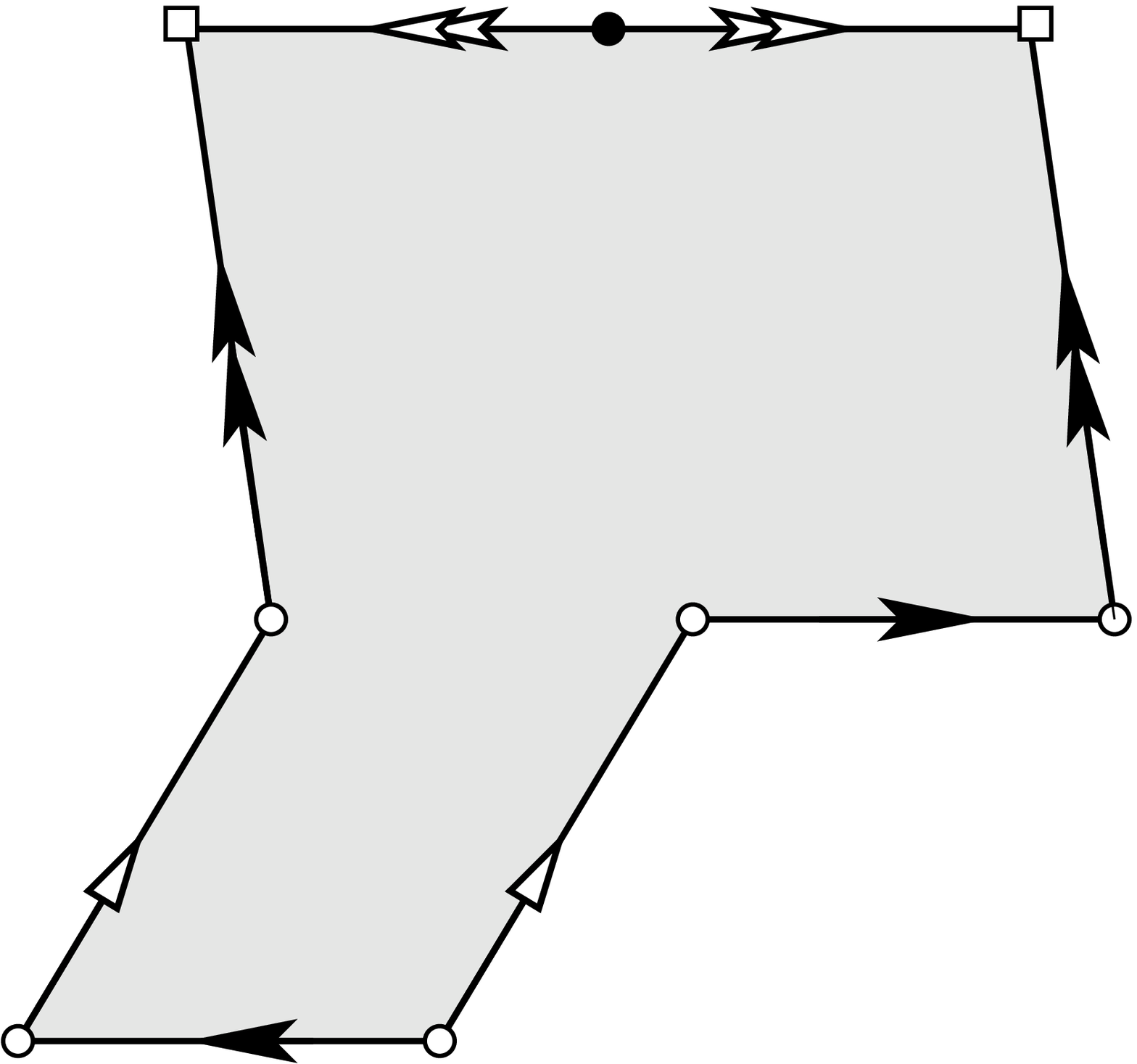}
\vspace{120bp}
\caption{
\label{zorich:half:translation}
Identifying  corresponding  pairs  of  sides by isometries
we obtain a flat surface of genus one with holonomy group $\ZZ$.
The associated quadratic differential belongs to the stratum $\cQ(2,-1,-1)$.
}
\end{figure}

\hspace*{-2.8pt}Similarly, closed flat surfaces with conical singularities, holonomy group $\ZZ$, and a choice of a line field at some point correspond to meromorphic quadratic differentials with at most simple poles. Flat surfaces of this type can be also glued from polygons: the sides of the polygon are again distributed into pairs of parallel sides of equal lengths, but this time the sides might be identified either by a parallel translation or by a central symmetry, see Figure \ref{zorich:half:translation}.

\subsection{Interval-exchange transformations and Rauzy classes}
\label{ss:Interval:exchange:transformations:and:Rauzy:classes}

Consider a flat surface $S$ having trivial linear holonomy. Consider a
family of parallel geodesics  emitted from a transverse segment $X$ and
their first return to $X$. The resulting first-return map is called an
\emph{interval-exchange transformation} $T\colon X\to X$. This map is a
piecewise isometry and it preserves the orientation. The example
below illustrates how interval-exchange transformations can be defined in an intrinsic combinatorial way.

\begin{figure}[hb]
\includegraphics{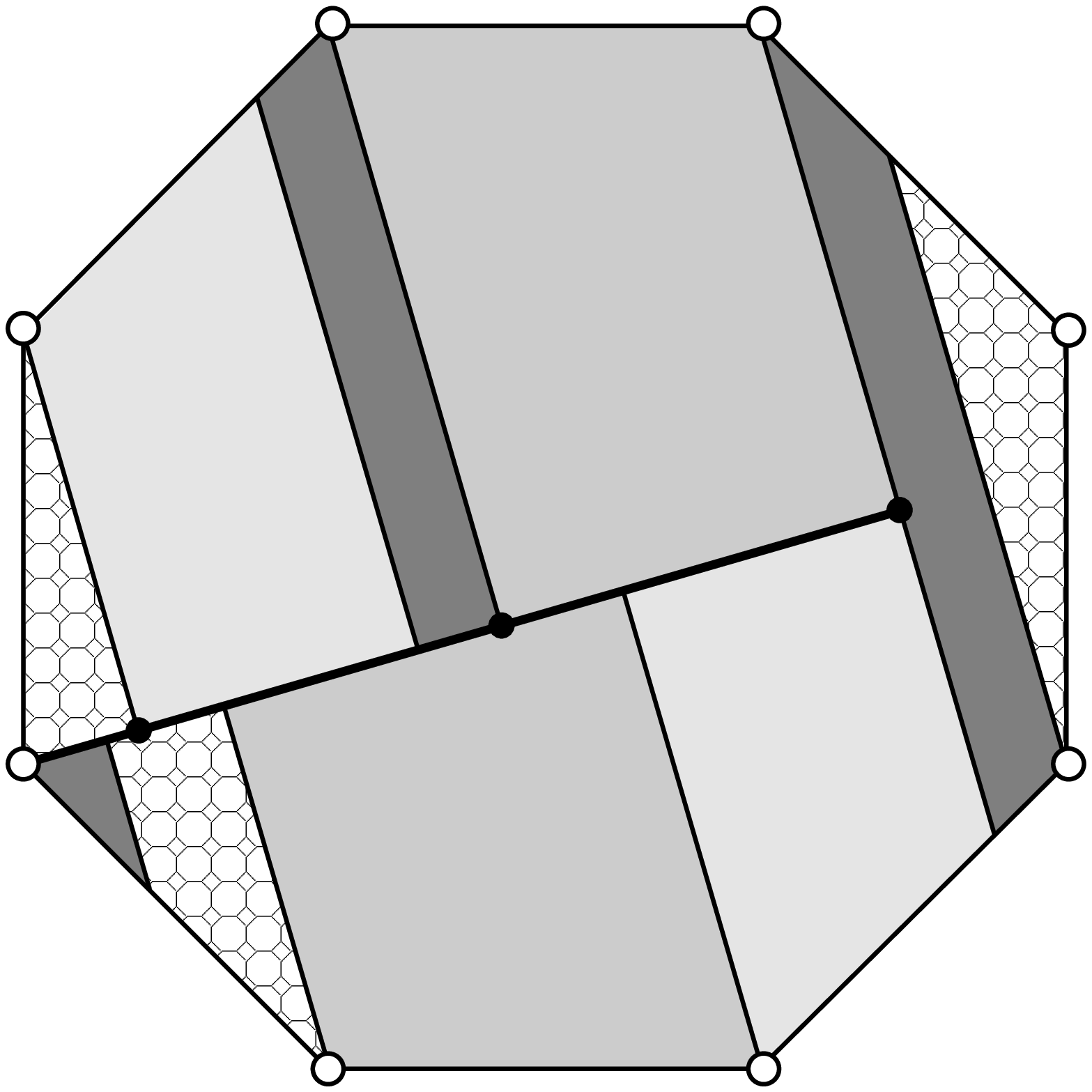}
\begin{picture}(0,0)(65,65)
\put(0,0){$1$}
\put(21,0){$2$}
\put(41,0){$3$}
\put(67,0){$4$}
\put(0,-13){$3$}
\put(10,-13){$1$}
\put(37,-13){$4$}
\put(74,-13){$2$}
\end{picture}
\begin{picture}(0,0)(191,1)
\put(111,-48){$\vec v_1$}
\put(144,-9){$\vec v_2$}
\put(196,-5){$\vec v_3$}
\put(232,-37){$\vec v_4$}
\put(116,-95){$\vec v_4$}
\put(157,-126){$\vec v_3$}
\put(204,-119){$\vec v_2$}
\put(236,-84){$\vec v_1$}
\end{picture}
\vspace{130bp}
\caption{
\label{fig:iet:octagon}
The first-return map $T\colon X\to X$ induced by
the vertical flow on a transverse segment $X$ is an
interval-exchange transformation.
}
\end{figure}

\begin{Example}
\label{ex:iet:2413}
Consider the flat surface from Figure \ref{fig:iet:octagon} and the first-return map of a geodesic flow in the vertical direction to a horizontal interval $X$. We see that this \emph{vertical flow} splits at a cone point and thus the first-return map chops the horizontal interval $X$  into several subintervals placing them back to $X$ in a different order (without overlaps and preserving the orientation). Since in our particular case the cone angle at the single cone point is $6\pi=3\cdot 2\pi$ there  are \emph{three} vertical trajectories  which  hit  the cone point.  The corresponding points at which $X$ is chopped are marked with bold dots. The remaining discontinuity  point  of\/ $X$ corresponds to a trajectory which hits the endpoint of\/ $X$. Subintervals $X_1, \dots, X_4$ (counted from left to right) appear after the first-return map in the order $X_3, X_1, X_4, X_2$. Thus, we can naturally associate a permutation
$$
\pi=\begin{pmatrix}1&2&3&4\\3&1&4&2\end{pmatrix}=(3,1,4,2)^{-1}=(2,4,1,3)
$$
to the corresponding interval-exchange transformation.
\end{Example}

A permutation $\pi$ of\/ $\noi$ elements $\{1, 2, \dots, \noi\}$ is called
\emph{irreducible} if it does not have any invariant proper subsets of the form  $\{1, 2, \dots, k\}$, where $k<\noi$. Having an interval-exchange transformation $T\colon X\to X$ corresponding to an irreducible permutation one
can always construct a \emph{suspension} over the interval-exchange transformation $T$: a flat surface $S$ and  a horizontal segment
$X\subset S$ inside it such that the first return of the vertical
flow to $X$ gives the initial  interval-exchange transformation,
see \cite{M82} or \cite{V82}.

Figure \ref{zorich:fig:suspension} illustrates a construction of a
suspension suggested in \cite{M82}. Namely, considering
vectors $\vec{v}_1, \dots, \vec{v}_\noi$ as complex numbers we
define the vector $\vec{v}_k$ as
$$
\vec{v}_k\dfn |X_k|+\sqrt{-1}\ \cdot\ \big(\pi(k)-k\big), \qquad k=1, \dots, \noi,
$$
where $|X_k|$ is the length of the $k$th subinterval,
and $\pi$ is a permutation defining the interval-exchange transformation. Irreducibility of the permutation $\pi$ implies that two broken lines $v_1, \dots, v_\noi$ and $v_{\pi^{-1}(1)}, \dots, v_{\pi^{-1}(n)}$  define a polygon, and, moreover, that the first broken line is located above the horizontal diagonal, and the second broken line is located below the horizontal diagonal as in Figure \ref{zorich:fig:suspension}. By construction, the first-return map induced by the vertical flow on the horizontal diagonal coincides with the initial interval-exchange transformation.

It is easy to check that any two closed surfaces obtained as suspensions over two interval-exchange transformations sharing the same permutation belong to the same connected component of the same stratum $\cH(d_1, \dots, d_\noz)$, see  \cite{V82}.

In our construction of a suspension $S$ over an interval-exchange transformation $T\colon X\to X$, the endpoints of the segment $X$ are
located at cone points of the flat surface $S$. By construction, the cone
angles at these cone points depend only on the permutation $\pi$ (and not
on lengths $|X_i|$, $i=1,\dots,n$, of the subintervals being exchanged).

\begin{Convention}
\label{conv:zeroes:at:the:end}
Whenever we say in this article that \emph{a permutation $\pi$ represents
a stratum $\cH(d_1,\dots,d_\noz)$}, we always assume that the corresponding
suspension has a singularity of degree $d_1$ at the left endpoint of\/ $X$
and one of degree $d_\noz$ at the right endpoint of\/ $X$.
\end{Convention}

A \emph{saddle connection} is a geodesic segment joining a pair of
cone singularities or a cone singularity  to itself without
any  singularities  in  its  interior.  For  the flat metrics  as
described  above,  regular  closed  geodesics  always  appear  in
families; any such family fills a  maximal  cylinder  bounded  on
each side by a closed saddle connection or by a chain of parallel
saddle connections.

Consider a flat surface $S$ in some stratum $\cH(d_1, \dots, d_\noz)$; by convention it is endowed with a distinguished vertical direction. Assume that the vertical direction is \emph{minimal}, \ie it does not admit any vertical saddle connection. Almost any surface in any stratum satisfies this condition, see \cite{M82,V82}.
A minimal vertical flow endows any horizontal interval $X$ embedded into
$S$ and having no singular points in its interior with an interval exchange
of\/ $\noi, \noi+1$ or $\noi+2$ subintervals, where $\noi=2g+\noz-1$, see,
say, \cite{V82}. By convention, let us always choose the horizontal
interval $X$ so that the induced interval-exchange transformation has the
minimal possible number $\noi$ of subintervals being exchanged.

Taking a union of permutations realized on all  horizontal segments satisfying the above condition we obtain a subset  $\mathfrak{R}_{ex}(S)\in\mathfrak{S}_n$ of the set $\mathfrak{S}_n$ of permutations of\/ $\noi$ elements. The following theorem is a slight reformulation of the results in \cite{V82}.

\begin{NNTheorem}[W. A. Veech]
Consider a flat surface $S$ with trivial holonomy. Suppose that the flow in the vertical direction on $S$ is minimal.
The set $\mathfrak{R}_{ex}(S)$ of permutations of\/ $\noi=2g+\noz-1$ elements realized by interval-exchange transformations induced by the vertical flow on a flat surface $S$ is the same for almost all flat surfaces $S$ in any connected component of any stratum $\cH(d_1, \dots, d_\noz)$. Denoting this set by
$\mathfrak{R}_{ex}$ we get an inclusion $\mathfrak{R}_{ex}(S_0)\subseteq\mathfrak{R}_{ex}$ for any surface $S_0$ in the same connected component of the stratum (provided minimality of the vertical flow on $S_0$).

The sets $\mathfrak{R}_{ex}(S_1)$ and $\mathfrak{R}_{ex}(S_2)$ corresponding to surfaces $S_1, S_2$ (with minimal vertical flows) from different connected components or from different strata do not intersect.
\end{NNTheorem}

A set $\mathfrak{R}_{ex}$ as in the above theorem is called an \emph{extended Rauzy class}. It contains only irreducible permutations. Conversely, let $\pi$ be an irreducible permutation.
We have seen, following constructions of H. Masur \cite{M82} and of W. Veech \cite{V82}, that one can construct a suspension $S(\pi)$ over any interval-exchange transformation corresponding to an irreducible permutation $\pi$, and clearly $\pi\in\mathfrak{R}_{ex}(S(\pi))$. Hence, any irreducible permutation $\pi$ belongs to an extended Rauzy class representing a connected component of a stratum embodying $S(\pi)$.

Thus, the set of all irreducible permutations decomposes into a disjoint union of extended Rauzy classes, and the theorem of Veech
establishes a one-to-one correspondence between connected components of strata of Abelian differentials and extended Rauzy classes.

Actually, an extended Rauzy class has an alternative, purely combinatorial (and much more constructive) definition as a minimal collection of irreducible permutations invariant under three explicit combinatorial operations. Two operations were introduced by G. Rauzy in \cite{Rauzy} and an additional one was introduced by W. A. Veech; see also Appendix \ref{a:generalized:rauzy:operations}. Applying this combinatorial approach W. Veech and P. Arnoux have decomposed irreducible permutations of a small number of elements in extended Rauzy classes and have found the first examples $\cH(4)$ and $\cH(6)$ of strata having several connected components.

\begin{NNRemark}
Some irreducible permutations give rise to strata with marked points. For example, if an interval-exchange transformation maps two consecutive intervals under exchange to two consecutive (in the same order) intervals, the corresponding suspension gets a ``fake singularity''. We tacitly avoid this type of permutation in the current paper.
\end{NNRemark}

\subsection{Analogs of interval-exchange transformations for quadratic differentials. Generalized permutations}
\label{ss:analogs:of:iet}
In many aspects the constructions of the previous section can be generalized to quadratic differentials. Consider a flat surface with holonomy group $\ZZ$ and a an oriented segment $X$ transverse to the vertical foliation. Since the vertical foliation is nonorientable, we emit trajectories from $X$ both in upward and downward directions. Making a slit along $X$ we get two shores $X^+$ and $X^-$ of the slit; we emit trajectories ``downward'' from the ``bottom'' shore $X^-$ and ``upward'' from the ``top'' shore $X^+$. We get a well-defined first-return map $T$ which is a piecewise isometry of\/ $X^+\sqcup X^-$ to itself. Each of the two copies $X^+$ and $X^-$ of\/ $X$ inherit an orientation of\/ $X$. When the image of a subinterval gets to the opposite shore
(as in the previous section) it preserves the orientation; when it gets to the same shore it changes the orientation.

Consider the natural partitions $X^+=X_1\sqcup\dots\sqcup X_r$ and
$X^-=X_{r+1}\sqcup\dots\sqcup X_s$, where each $X_i$ is a maximal subinterval of continuity of the map $T$. By construction, the map $T\colon X^+\sqcup X^-\to X^+\sqcup X^-$ is an involution which does not map any interval to itself, in particular, the total number $s$  of subintervals is even, $s=2\noi$. Denoting subintervals in each pair in involution by identical symbols we encode  combinatorics of the map $T$ by two lines of symbols in such way that every symbol appears exactly twice. We call such combinatorial data a \emph{generalized permutation}.

\begin{figure}[htb]
\includegraphics{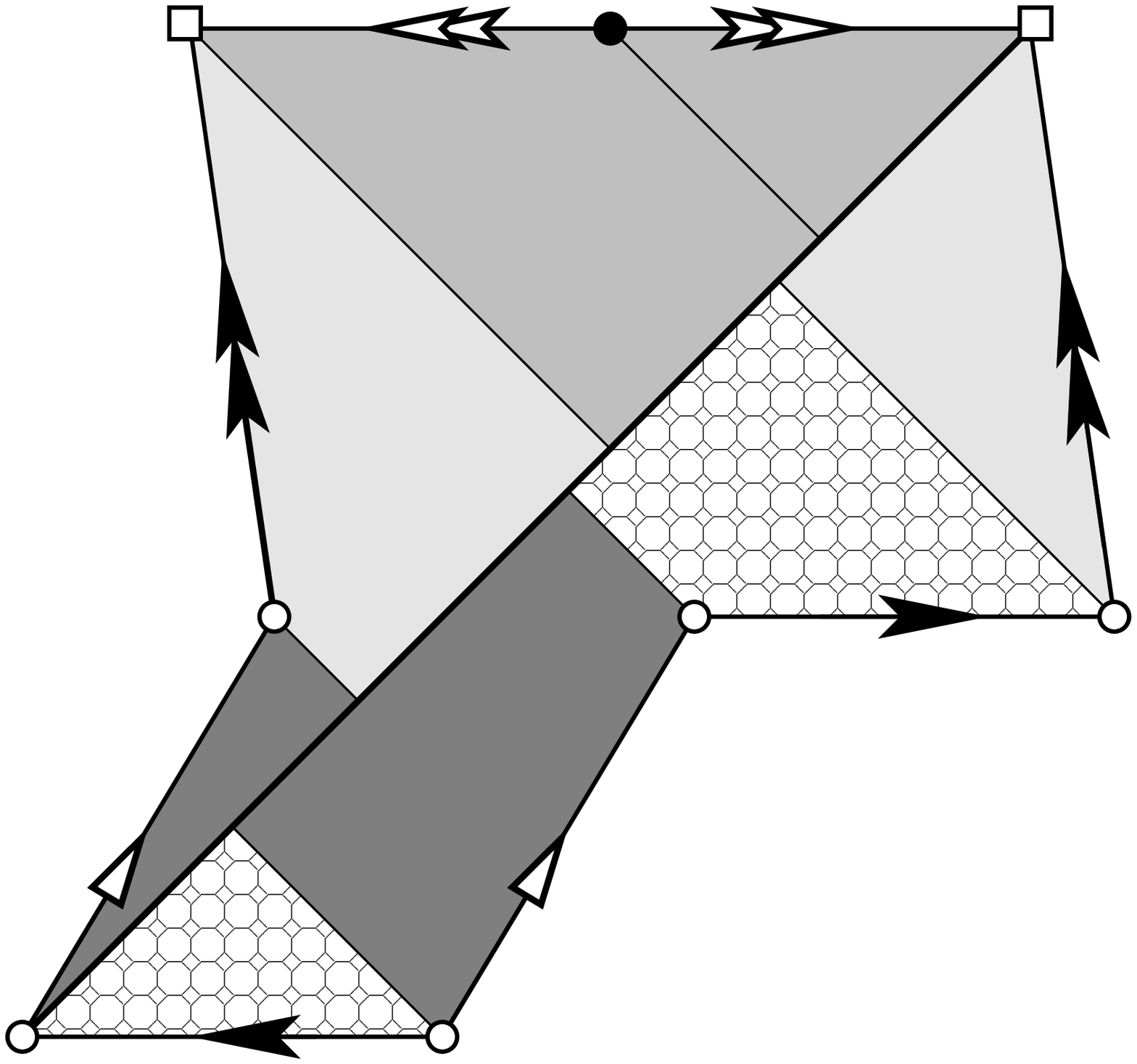}
\begin{picture}(0,0)(65,57)
\put(30,0){$1$}
\put(55,0){$2$}
\put(90,0){$3$}
\put(121,0){$3$}
\put(4,-13){$4$}
\put(42,-13){$1$}
\put(85,-13){$4$}
\put(118,-13){$2$}
\end{picture}
\vspace{110bp}
\caption{
\label{fig:generalized:iet}
Generalized
interval-exchange transformation.
}
\end{figure}

\begin{NNExample}
For a surface $S$ and a horizontal segment $X$ as in Figure \ref{fig:generalized:iet} the vertical foliation defines on two shores of\/ $X$  a generalized interval-exchange transformation with a generalized permutation
$$
\pi=\begin{pmatrix}1&2&3&3\\4&1&4&2\end{pmatrix}.
$$
Note, however, that the cardinalities of the upper and lower lines of a generalized permutation are in general different.
\end{NNExample}

One can define an \emph{irreducible} generalized permutation. This notion
is not quite elementary: an adequate combinatorial definition was
elaborated only recently by C. Boissy and E. Lanneau in \cite{BL}. We do
not reproduce this definition since in our paper we basically consider only
generalized permutations of a special form \eqref{eq:cylindric:permutation}
below which are always irreducible.

Any irreducible generalized permutation defines a family of generalized
inter\-val-exchange transformations; every interval-exchange transformation
in this family admits a suspension. Assume that our irreducible generalized
permutation is not a \emph{true} permutation. Then the flat surface $S$
obtained as a suspension has nontrivial holonomy group $\ZZ$. The connected
component of the stratum $\cQ(d_1, \dots, d_\noz)$, embodying $S$ is
uniquely determined by the generalized permutation. The collection of
generalized permutations of\/ $\noi=2g+\noz-1$ symbols corresponding to a
given connected component of a given stratum is again called an
\emph{extended Rauzy class}; it can be defined either implicitly by a
theorem analogous to the theorem of W. A. Veech cited above, or by an
effective combinatorial construction (minimal nonempty collection of
irreducible generalized permutations invariant under Rauzy operations) due
to C. Boissy and E. Lanneau. We refer the reader to their paper \cite{BL}
for a comprehensive study of relations between the combinatorics, geometry
and dynamics of generalized permutations and their Rauzy classes (see also
the outline in Appendix \ref{a:generalized:rauzy:operations}).

Generalizations of interval-exchange transformations corresponding to measured foliations on \emph{nonorientable} surfaces were studied by {C. Danthony} and
A.~Nogueira in \cite{DN}.

\subsection{Jenkins--Strebel differentials with a single cylinder}
An Abelian or a quadratic differential is called a \emph{Jenkins--Strebel
differential} if the union of critical leav\-es and critical points of its
horizontal foliation is compact. Equivalently, a differential is
Jenkins--Strebel if and only if any nonsingular horizontal leaf is closed.
In other words, the corresponding flat surface is glued from a finite
number of maximal flat cylinders filled with closed horizontal leaves. In
this paper we consider the special case when the Jenkins--Strebel
differential is represented by a \emph{single} flat cylinder $C$ filled
by closed horizontal leaves. Note that all zeroes and poles (critical
points of the horizontal foliation) of such differential are located on the
boundary of this cylinder.

Each of the two boundary components $\partial C^+$ and $\partial C^-$ of the cylinder is subdivided into a collection of horizontal saddle connections $\partial C^+=X_{\alpha_1}\sqcup\dots\sqcup X_{\alpha_r}$ and $\partial C^-=X_{\alpha_{r+1}}\sqcup\dots\sqcup X_{\alpha_s}$. The subintervals are naturally organized in pairs of subintervals of equal length; subintervals in every pair are identified by a natural isometry which preserves the orientation of the surface. Denoting both subintervals in the pair representing the same saddle connection by the same symbol, we can naturally encode the combinatorics of  identification of the boundaries of the cylinder by two lines of symbols,

\begin{equation}
\label{eq:js:prepermutation}
\begin{picture}(0,0)(-2,0)
\put(-3,10){\vector(1,0){0}}
\put(-5,15){\oval(10,10)[bl]}
\put(30,15){\oval(80,10)[t]}
\put(65,15){\oval(10,10)[br]}
\put(-3,-4){\vector(1,0){0}}
\put(-5,-9){\oval(10,10)[tl]}
\put(40,-9){\oval(100,10)[b]}
\put(85,-9){\oval(10,10)[tr]}
\end{picture}
\begin{matrix}\alpha_1&\dots&\alpha_r&\\ \alpha_{r+1}&\dots&\dots&\alpha_s\end{matrix}
\vspace{8pt}
\end{equation}

\noindent
where symbols in each line are organized in a cyclic order. By construction, every symbol appears exactly twice. If all the symbols in each line are distinct, the resulting flat surface has trivial linear holonomy and corresponds to an Abelian differential. Otherwise a flat metric of the resulting closed surface has holonomy group $\ZZ$; in the latter case it corresponds to a meromorphic quadratic differential with at most simple poles.

\begin{figure}[htb]
\includegraphics{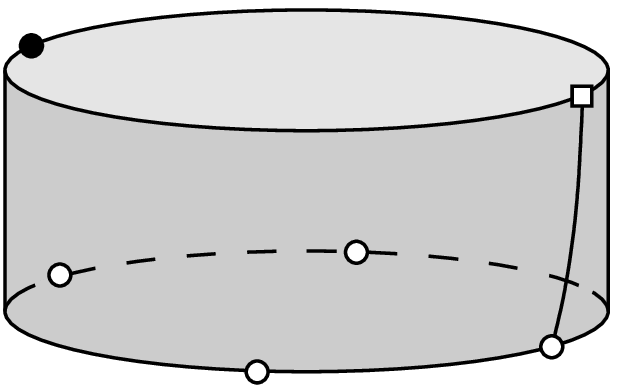}
\includegraphics{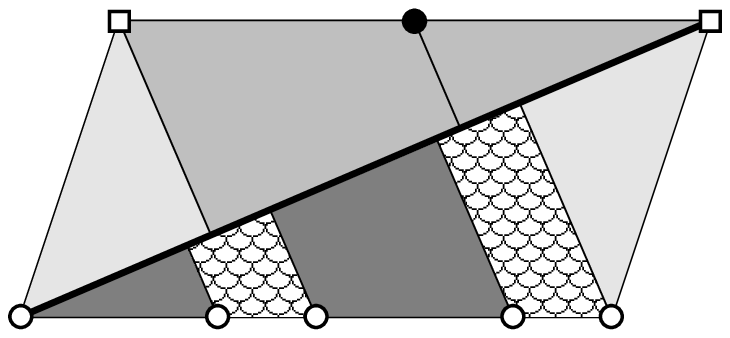}
\includegraphics{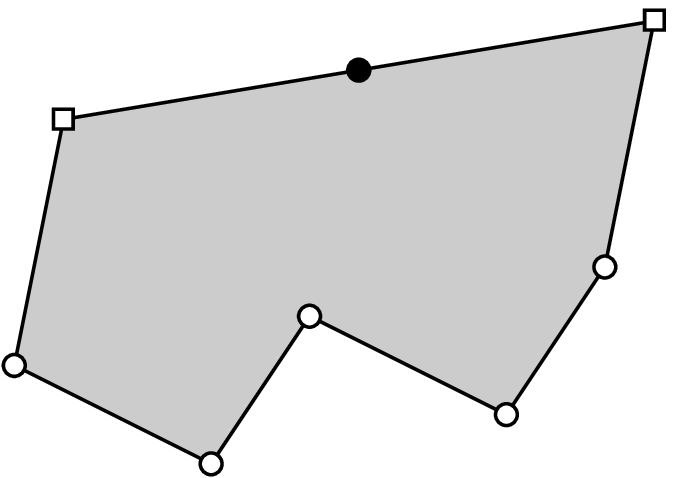}
\begin{picture}(0,0)(160,10)
\put(40,7){$X_1$}
\put(25,-10){$X_1$}
\put(17,-29){$X_3$}
\put(64,-27){$X_0$}
\put(0,-54){$X_2$}
\put(48,-56){$X_3$}
\put(76,-51){$X_2$}
\end{picture}
\begin{picture}(0,0)(50,-100)
\put(25,-112){1}
\put(69,-112){1}
\put(2,-168){2}
\put(24,-168){3}
\put(46,-168){2}
\put(67,-168){3}
\put(-9,-140){0}
\put(87,-140){0}
\end{picture}
\begin{picture}(0,0)(-73,-100)
\put(25,-112){1}
\put(67,-104){1}
\put(8,-169){2}
\put(35,-164){3}
\put(49,-162){2}
\put(77,-157){3}
\put(-7,-140){0}
\put(90,-128){0}
\end{picture}
\vspace{75bp}
\caption{
\label{fig:Jenkins:Strebel}
A Jenkins--Strebel differential with a single cylinder, one of its parallelogram  patterns, and its deformation inside the embodying stratum.
}
\end{figure}

Combinatorial data encoding identifications of the boundary components of the cylinder resembles a generalized permutation defined in the previous section. Choose a singular point on each of the two boundary components of the cylinder and join the two points by a geodesic segment $X_{\alpha_0}$, as in Figure \ref{fig:Jenkins:Strebel}. For example, choose the left endpoint of\/ $X_{\alpha_1}$ on the upper boundary component $\partial C^+$ and the left endpoint of\/ $X_{\alpha_{r+1}}$  on the lower boundary component $\partial C^-$. Cutting our metric cylinder by $X_{\alpha_0}$ we unfold our flat surface into a parallelogram. Consider a diagonal of this parallelogram and a direction transverse to this diagonal. The induced generalized interval-exchange transformation corresponds to one of the following two generalized permutations
\begin{equation}
\label{eq:cylindric:permutation}
\begin{pmatrix}\alpha_0&\alpha_1&\dots&\alpha_r&&\\ &\alpha_{r+1}&\dots&\dots&\alpha_s&\alpha_0\end{pmatrix}
\qquad\text{or}\qquad
\begin{pmatrix}&\alpha_1&\dots&\alpha_r&&\alpha_0\\ \alpha_0&\alpha_{r+1}&\dots&\dots&\alpha_s\end{pmatrix}
\end{equation}
representing two possible choices of a diagonal of the parallelogram.

\begin{NNExample}

Consider a Jenkins--Strebel differential with a single cylinder, a cutting segment $X_0$ and a diagonal of the resulting parallelogram as in Figure \ref{fig:Jenkins:Strebel}. The foliation orthogonal to the diagonal defines a generalized interval-exchange transformation on the two sides of the diagonal with a generalized permutation
\begin{equation}
\label{eq:011:23230}
\pi=\begin{pmatrix}
0&1&1&&\\
2&3&2&3&0
\end{pmatrix}.
\end{equation}
\end{NNExample}

\begin{NNRemark}
Note that neither of the two lines of structure \eqref{eq:js:prepermutation} has a distinguished element. Thus
there is no distinguished generalized permutation associated to a Jenkins--Strebel differential with a single cylinder: before adding an extra symbol $\alpha_0$ as in \eqref{eq:cylindric:permutation} we can cyclically move the elements in any of the two lines.

Moreover, if our flat surface is represented by a quadratic (and not Abelian) differential, there is no canonical way to assign the notions of ``top'' and ``bottom'' boundary components $\partial C^+$ and $\partial C^-$ of the cylinder. Thus, there is no canonical choice between the two structures
\vskip2pt$$
\begin{picture}(0,0)(-2,0)
\put(-3,10){\vector(1,0){0}}
\put(-5,15){\oval(10,10)[bl]}
\put(30,15){\oval(80,10)[t]}
\put(65,15){\oval(10,10)[br]}
\put(-3,-4){\vector(1,0){0}}
\put(-5,-9){\oval(10,10)[tl]}
\put(40,-9){\oval(100,10)[b]}
\put(85,-9){\oval(10,10)[tr]}
\end{picture}
\begin{matrix}\alpha_1&\dots&\alpha_r&\\ \alpha_{r+1}&\dots&\dots&\alpha_s\end{matrix}
\qquad\qquad\qquad
\begin{picture}(0,0)(-2,0)
\put(-3,10){\vector(1,0){0}}
\put(-5,15){\oval(10,10)[bl]}
\put(40,15){\oval(100,10)[t]}
\put(85,15){\oval(10,10)[br]}
\put(-3,-4){\vector(1,0){0}}
\put(-5,-9){\oval(10,10)[tl]}
\put(30,-9){\oval(80,10)[b]}
\put(65,-9){\oval(10,10)[tr]}
\end{picture}
\begin{matrix}\alpha_{s}&\dots&\dots&\alpha_{r+1}\\
\alpha_r&\dots&\alpha_1&
\end{matrix}
\vspace{8pt}
$$
\end{NNRemark}

\begin{NNExample}

All generalized permutations below correspond to the same Jenkins--Strebel differential with a single cylinder:

\begin{equation*}
\begin{pmatrix}0,1,2,1,2,3\\3,4,4,0\end{pmatrix},\
\begin{pmatrix}0,3,1,2,1,2\\4,3,4,0\end{pmatrix},\
\begin{pmatrix}0,4,4,3\\3,2,1,2,1,0\end{pmatrix},\
\begin{pmatrix}0,3,4,4\\2,1,2,1,3,0\end{pmatrix}
\end{equation*}
\end{NNExample}

\subsection{From cylindrical generalized permutations to Jenkins--Strebel differentials and polygonal patterns of flat surfaces}
\label{ss:generailzed:permutation}

It is convenient to formalize the combinatorial data above in the
following definitions.

\begin{Definition}
Consider an alphabet $\{\alpha_0, \alpha_1, \dots\}$. A \emph{generalized permutation} is an ordered pair
$$
\begin{pmatrix}\alpha_{i_0}&\alpha_{i_1}&\dots&\alpha_{i_r}&\\ \alpha_{i_{r+1}}&\dots&\dots&\dots&\alpha_{i_s}\end{pmatrix}
$$
of nonempty finite words (usually called ``lines'') $\{\alpha_{i_0},\alpha_1,\dots,\alpha_{i_r}\}$,
$\{\alpha_{i_{r+1}},\dots,\alpha_{i_s}\}$ satisfying the following condition: every symbol present in at least one of the words
appears exactly one more time either in the same word or in the other one.
\end{Definition}

A bijection between alphabets $\{\alpha_0, \alpha_1, \dots\}$ and
$\{\beta_0, \beta_1, \dots\}$ (including a bijection of an alphabet to itself) induces natural bijection between generalized permutations in letters $\alpha_i$ and generalized permutations in letters $\beta_i$. Unless stated explicitly, we will identify the corresponding permutations.

\begin{Definition}
A generalized permutation is called \emph{cylindrical} if the first symbol of one of the lines coincides with the last symbol of the complementary line ({see \eqref{eq:cylindric:permutation}})
and if, moreover, the set of all symbols in either of two lines does not form a proper subset of the set of symbols in a complementary line.
\end{Definition}
Note that we allow the unions of symbols in the lines of a cylindrical generalized permutation to coincide: in this case we get a true permutation.

From now on we shall consider only cylindrical generalized permutations.
More general generalized permutations will reappear only in Section
\ref{ss:exceptional:strata} and in the appendices. In particular, any
cylindrical generalized permutation is necessarily \emph{irreducible},
see \cite{BL}. For practical purposes, this means that we can always
construct a ``suspension'' over a cylindrical generalized permutation; this
construction is described in the next several paragraphs.

Consider a cylindrical generalized permutation $\pi$. Let $\{\beta_1,
\dots, \beta_p\}$ be the symbols which are present only in the bottom line
(as symbols ``$2,3$'' in generalized permutation \eqref{eq:011:23230}
above), and let $\{\gamma_1, \dots, \gamma_q\}$ be the symbols which are
present only in the top line (as symbol ``1'' in the generalized permutation
\eqref{eq:011:23230} above). Our definition of a cylindrical generalized
permutation implies that either these sets are both empty, and so $\pi$ is
a true permutation, or they are both nonempty. We are especially interested
in the latter case. Consider a generalized interval-exchange transformation
$T$ corresponding to $\pi$. We can choose any lengths for the subintervals
being exchanged provided they satisfy the linear relation
\begin{equation}
\label{eq:relation:top:bottom}
|X_{\gamma_1}|+\dots+|X_{\gamma_p}|=|X_{\beta_1}|+\dots+|X_{\beta_q}|.
\end{equation}
Clearly for any such generalized interval-exchange transformation we can perform a construction inverse to the one described in the previous section and realize a ``suspension'' by a Jenkins--Strebel differential with a single cylinder, such as in the middle and the left pictures in Figure \ref{fig:Jenkins:Strebel}. (This observation is due to E.~Lanneau, \cite{L}.)

Finally, note that we can consider a parallelogram constructed in our suspension as a polygon similar to the ones in Figures \ref{zorich:fig:suspension} through \ref{fig:generalized:iet}. The polygon is obtained from two broken lines
$\vec v_{\alpha_0},\vec v_{\alpha_1}, \dots, \vec v_{\alpha_r}$ and
$\vec v_{\alpha_{r+1}},\vec v_{\alpha_1}, \dots, \vec v_{\alpha_s}, \vec v_{\alpha_0}$, where all the vectors different from $\vec v_{\alpha_0}$ are horizontal. These vectors satisfy a relation analogous to relation \eqref{eq:relation:top:bottom}, namely:
$$
\vec v_{\gamma_1}+\dots+\vec v_{\gamma_p}=\vec v_{\beta_1}+\dots+\vec v_{\beta_q}.
$$

Deforming all the vectors slightly by a deformation respecting the above relation (see the right picture in Figure \ref{fig:Jenkins:Strebel}) we obtain a small open neighborhood of the initial Jenkins--Strebel differential in the embodying stratum $\cQ(d_1, \dots, d_\noz)$. In other words, an open set of flat surfaces in $\cQ(d_1, \dots, d_\noz)$ can be obtained by identification of pairs of corresponding sides of a polygon
as on the right picture in Figure \ref{fig:Jenkins:Strebel}. The combinatorics of the polygon (and of the identifications of pairs of sides) is fixed by the initial generalized permutation $\pi$. Note that the property ``a flat surface $S$ can be glued from a polygon of fixed combinatorics $\pi$ '' is $GL(2,\R{})$-invariant. Thus, by ergodicity of the $SL(2;\R{})$-action, we get the following simple observation.

\begin{Proposition}
If a flat surface $S$ can be unfolded to a polygon having combinatorial structure represented by a cylindrical generalized permutation, then almost any flat surface in the same connected component of the embodying stratum can be unfolded to a polygon sharing the same combinatorial structure.
\end{Proposition}

\subsection{Goal of the paper}

The main goal of the current paper is to explicitly construct a cylindrical generalized permutation representing any given connected component of any given stratum of Abelian or quadratic differentials. As was shown in the previous section, such a permutation immediately provides us with a Jenkins--Strebel differential with a single cylinder in the corresponding connected component, a polygonal representation of almost any flat surface in the corresponding connected component, and a representative of the corresponding extended Rauzy class.
Our main tool is the geometry of ribbon graph representations of Jenkins--Strebel differentials (see \cite{Kontsevich:Airy:Function} for more details) combined with elementary combinatorics of cylindrical permutations.

\begin{NNRemark}
It was proved by
A. Douady and J. Hubbard that
Jenkins--Strebel differentials are dense in the principal stratum $\cQ(1,\dots,1)$ of quadratic differentials, see \cite{DH}. This result was strengthened in
by H. Masur in \cite{M79} who proved that Jenkins--Strebel differentials \emph{with a single cylinder} are also dense in $\cQ(1,\dots,1)$. The statement on the density of
Jenkins--Strebel differentials with a \emph{single} cylinder
was extended by M. Kontsevich and the author to any stratum of Abelian differentials, see \cite{KZ03}. The latter proof was generalized by E. Lanneau in \cite{L} for any stratum of meromorphic quadratic differentials with at most simple poles.

\begin{figure}[htb]
\includegraphics{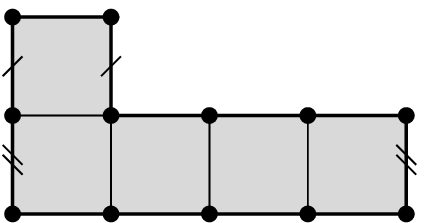}
\includegraphics{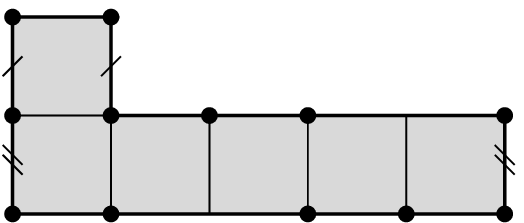}
\includegraphics{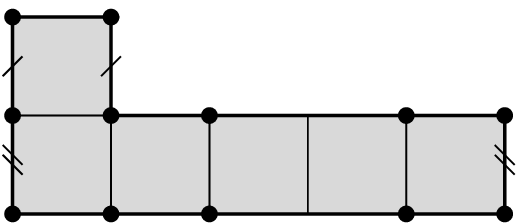}
\includegraphics{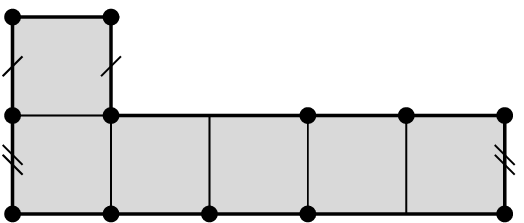}
\begin{picture}(0,0)
\begin{picture}(0,0)
\put(-168,-7){\scriptsize $1$}
\put(-153,-22){\scriptsize $2$}
\put(-139,-22){\scriptsize $3$}
\put(-125,-22){\scriptsize $4$}
\put(-168,-47){\scriptsize $1$}
\put(-153,-47){\scriptsize $4$}
\put(-139,-47){\scriptsize $3$}
\put(-125,-47){\scriptsize $2$}
\end{picture}
\begin{picture}(0,0)(-83,0)
\put(-168,-7){\scriptsize $1$}
\put(-153,-22){\scriptsize $2$}
\put(-139,-22){\scriptsize $3$}
\put(-119,-22){\scriptsize $4$}
\put(-168,-47){\scriptsize $1$}
\put(-147,-47){\scriptsize $4$}
\put(-125,-47){\scriptsize $3$}
\put(-111,-47){\scriptsize $2$}
\end{picture}
\begin{picture}(0,0)(-176,0)
\put(-168,-7){\scriptsize $1$}
\put(-153,-22){\scriptsize $2$}
\put(-133,-22){\scriptsize $3$}
\put(-111,-22){\scriptsize $4$}
\put(-168,-47){\scriptsize $1$}
\put(-153,-47){\scriptsize $4$}
\put(-133,-47){\scriptsize $3$}
\put(-111,-47){\scriptsize $2$}
\end{picture}
\begin{picture}(0,0)(-268,0)
\put(-168,-7){\scriptsize $1$}
\put(-146,-22){\scriptsize $2$}
\put(-125,-22){\scriptsize $3$}
\put(-111,-22){\scriptsize $4$}
\put(-168,-47){\scriptsize $1$}
\put(-153,-47){\scriptsize $4$}
\put(-139,-47){\scriptsize $3$}
\put(-118,-47){\scriptsize $2$}
\end{picture}
\end{picture}
\vspace{55bp}
\caption{
\label{fig:6tiled:surfaces}
$GL(2,\R{})$-orbits of these surfaces are closed and do not contain Jenkins--Strebel differentials with a single cylinder.
}
\end{figure}

Nevertheless, a closed $SL(2,\R{})$-invariant suborbifold of a stratum might contain no Jenkins--Strebel differentials with a single cylinder. As usual, counterexamples might be found among orbits of  arithmetic Veech surfaces. For example, the four surfaces presented in Figure \ref{fig:6tiled:surfaces} belong to four distinct $SL(2,\R{})$-orbits  in the connected component $\cH^{\mathit{hyp}}(4)$. These orbits contain correspondingly $15$, $15$, $10$ and $10$ square-tiled surfaces. None of them are composed of a single cylinder, which implies that none of the corresponding  $SL(2,\R{})$-orbits contain a Jenkins--Strebel differential with a single cylinder.
\end{NNRemark}

\subsection{Idea of construction}
\label{ss:idea:of:construction}

We complete the introduction by an illustration of the main idea of our construction in a simple particular case. Consider a Jenkins--Strebel differential with a single cylinder. Suppose for simplicity that the resulting flat surface has trivial linear holonomy. In the previous section we have represented a Jenkins--Strebel differential as a cylinder with some identifications of the boundary.

\begin{figure}[htb]
\includegraphics{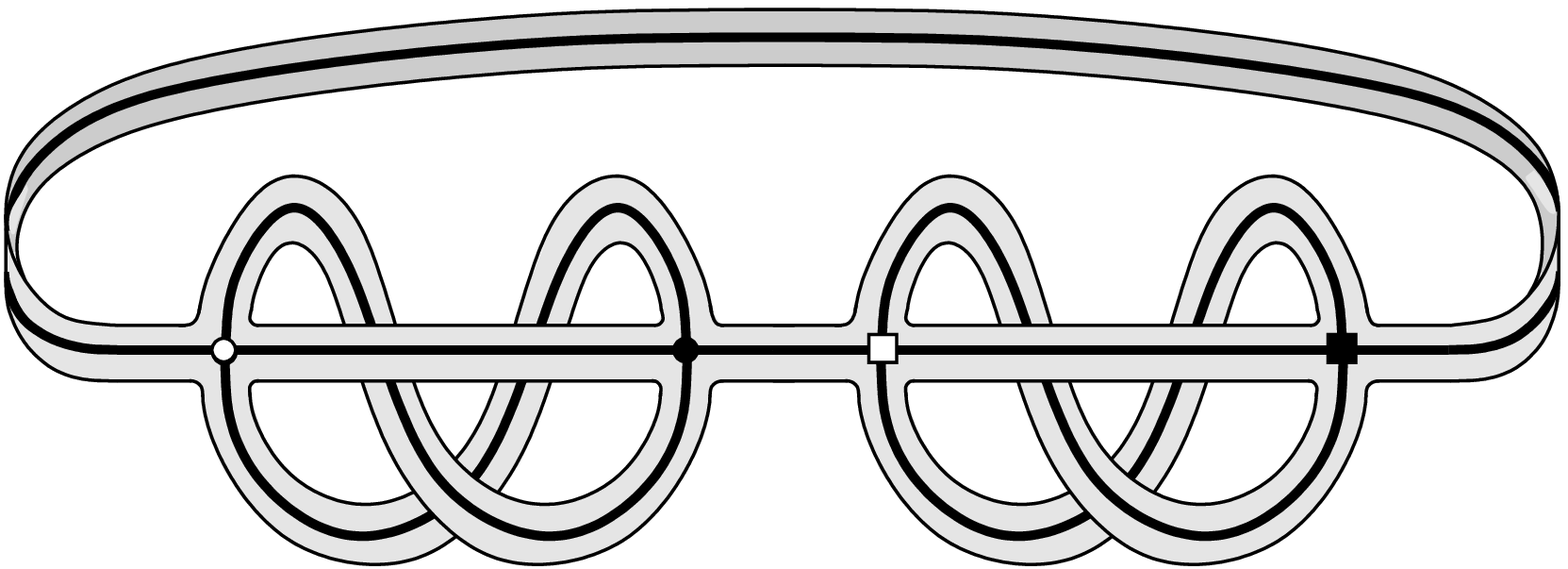}
\includegraphics{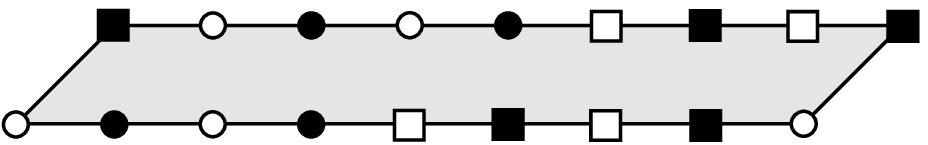}
\includegraphics{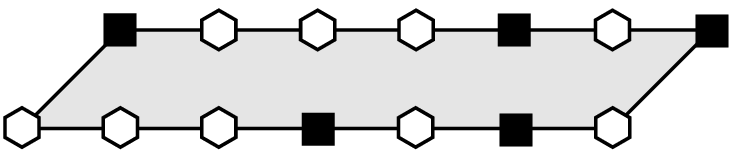}
\begin{picture}(0,0)(35,-89)
\put(-88,-138){$X_1$}
\put(-59,-112){$X_2$}
\put(-32,-138){$X_3$}
\put(-1,-112){$X_4$}
\put(30,-138){$X_5$}
\put(61,-112){$X_6$}
\put(90,-138){$X_7$}
\put(120,-112){$X_8$}
\put(146,-138){$X_1$}
\end{picture}

\begin{picture}(0,0)(3,-13)
\begin{picture}(0,0)(0,5)
\put(-88,-112){1}
\put(-60,-112){2}
\put(-32,-112){3}
\put(-29,-109){\circle{10}}
\put(-4,-112){4}
\put(24,-112){5}
\put(27,-109){\circle{10}}
\put(53,-112){6}
\put(82,-112){7}

\put(110,-112){8}
\end{picture}
\begin{picture}(0,0)(28,5)
\put(-89,-155){4}
\put(-60,-155){3}
\put(-57,-152){\circle{10}}
\put(-32,-155){2}
\put(-4,-155){5}
\put(-1,-152){\circle{10}}
\put(24,-155){8}
\put(53,-155){7}

\put(82,-155){6}
\put(110,-155){1}
\end{picture}
\begin{picture}(0,0)(3,5)
\put(-122,-131){0}
\put(115,-136){0}
\end{picture}
\end{picture}

\begin{picture}(0,0)(-3,-26)
\begin{picture}(0,0)(-27,75)
\put(-88,-112){1}
\put(-60,-112){2}
\put(-32,-112){4}
\put(-4,-112){6}
\put(24,-112){7}
\put(53,-112){8}
\end{picture}
\begin{picture}(0,0)(3,75)
\put(-89,-155){4}
\put(-61,-155){2}
\put(-32,-155){8}
\put(-4,-155){7}
\put(24,-155){6}
\put(53,-155){1}
\end{picture}
\begin{picture}(0,0)(-27,75)
\put(-126,-130){0}
\put(51,-137){0}
\end{picture}
\end{picture}
\vspace{235bp}
\caption{
\label{fig:merging:zeroes}
The ribbon graph representation of a Jenkins--Strebel differential with a single cylinder (top picture) versus the cylinder representation (middle picture). A Jenkins--Strebel differential with a single cylinder obtained from the initial one by contracting saddle connections $X_3$ and $X_5$ (bottom picture).
}
\end{figure}

Consider now another representation of the same Jenkins--Strebel differential, obtained by cutting the surface along a \emph{regular} horizontal leaf passing along the ``equator'' of the cylinder. We get an oriented \emph{flat ribbon graph} (see \cite{Kontsevich:Airy:Function} for details) with two boundary components. Its skeleton is realized by an oriented graph of horizontal saddle connections of our Jenkins--Strebel differential.

Following a boundary component in the positive direction of the horizontal foliation we can trace the cyclic order in which we follow the saddle connections. For example, for the concrete ribbon graph in the top picture in Figure \ref{fig:merging:zeroes}, we get
\begin{equation}
\label{eq:H1111:cyclic}
\begin{picture}(0,0)
\put(-3,10){\vector(1,0){0}}
\put(-5,15){\oval(10,10)[bl]}
\put(70,15){\oval(160,10)[t]}
\put(145,15){\oval(10,10)[br]}
\put(-3,-4){\vector(1,0){0}}
\put(-5,-9){\oval(10,10)[tl]}
\put(70,-9){\oval(160,10)[b]}
\put(145,-9){\oval(10,10)[tr]}
\end{picture}
\begin{matrix}
1\to 2\to 3\to 4\to 5\to 6\to 7\to 8\\
4\to 3\to 2\to 5\to 8\to 7\to 6\to 1
\end{matrix}
\vspace{8pt}
\end{equation}
\noindent
corresponding to the upper and lower boundary components $\partial C^+$ and $\partial C^-$, respectively. It is easy to pass from a ribbon graph representation to a cylinder representation and vice versa.

Note that the lengths of the saddle connections $|X_1|, \dots, |X_8|$ are independent variables, and we may deform them arbitrarily. We can even shrink one of the intervals $|X_1|, \dots, |X_8|$ completely
and we still get a legitimate flat surface represented by a new Jenkins--Strebel differential with a single cylinder.

It is obvious from the ribbon graph representation (see the top picture in Figure \ref{fig:merging:zeroes}) that in our example  the Jenkins--Strebel differential belongs to the stratum $\cH(1,1,1,1)$. Let us denote the zeroes (vertices of the skeleton graph) as indicated in Figure \ref{fig:merging:zeroes}. Consider the following embedded chain of saddle connections joining these four zeroes:
$$
\begin{CD}
\quad@>X_3>> \quad @>X_5>> \quad @>X_7>> \quad
\end{CD}
\begin{picture}(0,0)(150,0)

\thicklines
\put(0,0){\circle*{7}}
\put(46,0){\circle{6}}
\put(93,-3){\framebox(5,5){\quad}}
\put(140,-4){$\blacksquare$}
\end{picture}.
$$

Shrinking the saddle connection $X_3$ we merge two simple zeroes and get a Jenkins--Strebel differential in the stratum $\cH(2,1,1)$.
Shrinking the saddle connections $X_3$ and $X_7$ we merge two pairs of simple zeroes and get a Jenkins--Strebel differential in the stratum $\cH(2,2)$. Shrinking the saddle connections $X_3$ and $X_5$ we merge three simple zeroes and get a Jenkins--Strebel differential in the stratum $\cH(3,1)$. Shrinking all three saddle connections
$X_3, X_5, X_7$ we get a Jenkins--Strebel differential in $\cH(4)$.

In terms of a cylinder representation these operations are  simple: we just erase corresponding symbols in each of the two lines in \eqref{eq:H1111:cyclic}. For example, the cylinder representation of the Jenkins--Strebel differential in $\cH(3,1)$ obtained by shrinking the saddle connections $X_3$ and $X_5$ is presented in the bottom picture of Figure \ref{fig:merging:zeroes}; it is encoded as
$$
\begin{picture}(0,0)
\put(-3,10){\vector(1,0){0}}
\put(-5,15){\oval(10,10)[bl]}
\put(50,15){\oval(120,10)[t]}
\put(105,15){\oval(10,10)[br]}
\put(-3,-4){\vector(1,0){0}}
\put(-5,-9){\oval(10,10)[tl]}
\put(50,-9){\oval(120,10)[b]}
\put(105,-9){\oval(10,10)[tr]}
\end{picture}
\begin{matrix}
1\to 2\to 4\to 6\to 7\to 8\\
4\to 2\to 8\to 7\to 6\to 1
\end{matrix}\quad.
\vspace{8pt}
$$

Cutting the cylinder of the resulting surface by a segment $X_0$ we obtain a suspension over an interval-exchange transformation with permutation
$$
\pi=\begin{pmatrix}
0&1&2&4&6&7&8&\\
&4&2&8&7&6&1&0
\end{pmatrix},
$$
or, after alignment and reenumeration,
$$
\pi=\begin{pmatrix}
1&2&3&4&5&6&7\\
4&3&7&6&5&2&1
\end{pmatrix}.
$$
Thus, the latter permutation represents the stratum $\cH(3,1)$.

\begin{NNRemark}
In general, shrinking a saddle connection joining a pair of distinct zeroes may result in a degenerate surface: our saddle connection might have homologous ones (see \cite{Eskin:Masur:Zorich} and \cite{MZ} for details). Moreover, shrinking a saddle connection joining a zero to a simple pole for surfaces in the component $\cQ^{\mathit{irr}}(9,-1)$ \emph{always} yields a degenerate surface (see \cite{L}). Our situation is  special: horizontal saddle connections of an Abelian Jenkins--Strebel differential with a single cylinder are never homologous; for quadratic Jenkins--Strebel differentials it is  easy to identify the homologous ones (see \cite{MZ} for more details).
\end{NNRemark}

\section{Representatives of strata of Abelian differentials}
\label{s:Abelian:differentials}

We will separately consider Abelian and quadratic differentials. In each
case we start by recalling a classification of connected components of the strata. Then for each connected component we construct a cylindrical generalized permutation of the form \eqref{eq:cylindric:permutation} representing the chosen connected component.

\subsection{Classification of connected components:
strata of Abelian differentials}

Connected components of strata of Abelian differentials are classified by the following two parameters, see \cite{KZ03}.

For any $g\ge 2$
two special strata, namely $\cH(2g-2)$ and $\cH(g-1,g-1)$ contain \emph{hyperelliptic} connected components $\cH^{\mathit{hyp}}(2g-2)$ and $\cH^{\mathit{hyp}}(g-1,g-1)$. A flat surface in the stratum $\cH(2g-2)$ belongs to $\cH^{\mathit{hyp}}(2g-2)$ if and only if the underlying Riemann surface is hyperelliptic. A flat surface in the stratum $\cH(g-1,g-1)$ belongs to the component $\cH^{\mathit{hyp}}(g-1,g-1)$ if and only if the underlying Riemann surface is hyperelliptic and the hyperelliptic involution interchanges the zeroes. In particular, according to this definition the locus of hyperelliptic flat surfaces in $\cH(g-1,g-1)$ for which the hyperelliptic involution fixes the zeroes is located in one of the \text{nonhyperelliptic} connected components.

When all degrees of zeroes of an Abelian differential are even, \ie when the flat surface belongs to $\cH(2d_1, \dots, 2d_\noz)$, one can associate to the flat surface a \emph{parity of the spin-structure} which takes value zero or one depending only on the embodying connected component (see Appendix \ref{a:spin:structure}).

For flat surfaces of genus four and higher, these two invariants take all possible values and classify the connected components. In small genera some values of these invariants are not realizable, so there
are fewer connected components than in general.

\begin{NNTheorem}[M. Kontsevich and A. Zorich]

All connected components of any stratum  of Abelian differentials
on a complex curve of genus $g\ge 4$ are described by the following list:

The stratum $\mathcal{H}(2g-2)$ has  three  connected components:
the hyperelliptic one, $\mathcal{H}^{\mathit{hyp}}(2g-2)$,  and  two other
components:            $\mathcal{H}^{\mathit{even}}(2g-2)$             and
$\mathcal{H}^{\mathit{odd}}(2g-2)$  corresponding  to even  and  odd  spin
structures.

The stratum  $\mathcal{H}(2l,2l)$,  $l\ge  2$ has three connected
components:  the  hyperelliptic one,  $\mathcal{H}^{\mathit{hyp}}(2l,2l)$,
and   two   other  components:   $\mathcal{H}^{\mathit{even}}(2l,2l)$  and
$\mathcal{H}^{\mathit{odd}}(2l,2l)$.

All the  other strata of the form $\mathcal{H}(2l_1,\dots,2l_n)$,
where  all   $l_i\ge   1$,   have   two   connected   components:
$\mathcal{H}^{\mathit{even}}(2l_1,\dots,2l_n)$                         and
$\mathcal{H}^{\mathit{odd}}(2l_1,\dots,2l_n)$, corresponding to  even  and
odd spin structures.

The strata $\mathcal{H}(2l-1,2l-1)$, $l\ge 2$, have two connected
components;  one   of  them,  $\mathcal{H}^{\mathit{hyp}}(2l-1,2l-1)$,   is
hyperelliptic;  the  other one, $\mathcal{H}^{\mathit{nonhyp}}(2l-1,2l-1)$,  is not.

All  other strata of Abelian  differentials on complex  curves of
genera $g\ge 4$ are nonempty and connected.
\end{NNTheorem}

The theorem below shows that in genera $g=2,3$ some components
are missing with respect to the general case. Connected components
in small genera were classified by W. A. Veech and by P. Arnoux using \emph{Rauzy classes} (see Appendix \ref{a:generalized:rauzy:operations}); the corresponding invariants (hyperellipticity and parity of the spin structure) were evaluated by M. Kontsevich and the author. In full generality the theorem below is proved in \cite{KZ03}.

\begin{NNTheorem}

The moduli space  of Abelian differentials  on a complex curve  of  genus $g=2$    contains    two     strata:    $\mathcal{H}(1,1)$    and $\mathcal{H}(2)$.  Each  of them is connected and coincides  with its hyperelliptic component.

Each of  the  strata  $\mathcal{H}(2,2)$, $\mathcal{H}(4)$ of the
moduli space  of Abelian differentials  on a complex curve of genus $g=3$ has two  connected  components:  the  hyperelliptic  one, and one having odd spin structure. The other  strata  are  connected  for genus $g=3$.
\end{NNTheorem}

\subsection{Representatives of connected strata}

We use the following natural convention in the statements of Propositions \ref{pr:abelian:general} -- \ref{pr:qd:g:ge3}. Consider an ordered set $\{j_1, j_2, \dots\}$. For $n=1$ the subcollection
$$
\underbrace{j_1,\dots,j_{n-1}}_{n-1}
$$
is defined to be empty.

Let $d_1,\dots,d_\noz$ be an arbitrary collection of strictly positive integers satisfying the relation $d_1+\dots+d_\noz=2g-2$.

\begin{Proposition} A permutation obtained by erasing symbols
\label{pr:abelian:general}
\begin{multline*}
\underbrace{3,5,\dots,2d_1-1}_{d_1-1},\quad \underbrace{2d_1+3,\dots,2(d_1+d_2)-1}_{d_2-1},\quad  \underbrace{\dots \dots}_{\dots}\ ,
\\
\underbrace{\dots \dots}_{\dots}\ ,\quad
\underbrace{2(d_1+\dots+d_{\noz-1})+3,\dots,2(d_1+\dots+d_\noz)-1}_{d_\noz-1}
\end{multline*}
in the permutation
\begin{equation*}
\begin{pmatrix}
0&1&2&3&4&&5&6&7&8&\dots&4g-7&4g-6&4g-5&4g-4\\
& &4&3&2&&5&8&7&6&\dots&4g-7&4g-4&4g-5&4g-6&1&0
\end{pmatrix},
\begin{picture}(0,0)(228,-10)
\put(-71,0){\circle{10}}
\put(-71,-13){\circle{10}}
\put(-30,0){\circle{10}}
\put(-30,-13){\circle{10}}
\put(0,0){\circle{10}}
\put(0,-13){\circle{10}}
\put(64,-1){\oval(30,12)}
\put(64,-14){\oval(30,12)}
\put(138,-1){\oval(30,12)}
\put(138,-14){\oval(30,12)}
\end{picture}
\end{equation*}
represents the stratum $\cH(d_1,\dots,d_\noz)$.
\end{Proposition}
Note that for $g=2$ the above permutation should be read as
$$
\begin{pmatrix}
0&1&2&3&4&\\
& &4&3&2&1&0
\end{pmatrix}
$$
\begin{proof}
The proof is completely analogous to the one presented in Section \ref{ss:idea:of:construction} for the case $g=3$. Namely, consider a Jenkins--Strebel differential with a single cylinder as in Figure \ref{fig:Jenkins:Strebel:H1111}.
\begin{figure}[htb]
\includegraphics{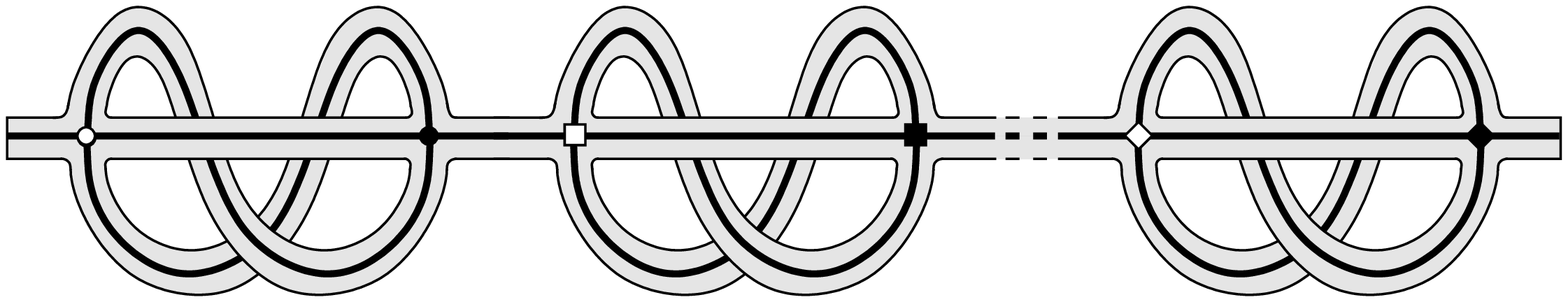}
\includegraphics{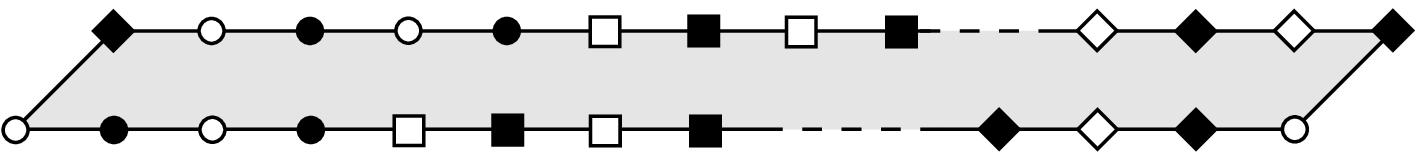}
\begin{picture}(0,0)(0,5)
\begin{picture}(0,0)(72,-120)
\put(-88,-140){\scriptsize $1$}
\put(-65,-118){\scriptsize $2$}
\put(-40,-140){\scriptsize $3$}
\put(-17,-118){\scriptsize $4$}
\put(9,-140){\scriptsize $5$}
\put(35,-118){\scriptsize $6$}
\put(61,-140){\scriptsize $7$}
\put(84,-118){\scriptsize $8$}
\put(144,-118){\tiny $4g-6$}
\put(171,-140){\tiny $4g\!-\!5$}
\put(194,-118){\tiny $4g-4$}
\put(225,-140){\scriptsize $1$}
\end{picture}

\begin{picture}(0,0)(34,-20)
\put(-94,-116){\scriptsize $1$}
\put(-70,-116){\scriptsize $2$}
\put(-48,-116){\scriptsize $3$}
\put(-46,-113){\circle{10}}
\put(-26,-116){\scriptsize $4$}
\put(-4,-116){\scriptsize $5$}
\put(20,-116){\scriptsize $6$}
\put(-2,-113){\circle{10}}
\put(43,-116){\scriptsize $7$}
\put(64,-116){\scriptsize $8$}
\put(45,-113){\circle{10}}
\put(94,-116){\scriptsize $...$}
\put(127,-116){\tiny $4g\!-\!6$}
\put(150,-116){\tiny $4g\!-\!5$}
\put(158,-114){\oval(20,9)}
\put(173,-116){\tiny $4g\!-\!4$}
\end{picture}
\begin{picture}(0,0)(60,-21)
\put(-96,-155){\scriptsize $4$}
\put(-73,-155){\scriptsize $3$}
\put(-71,-152){\circle{10}}
\put(-51,-155){\scriptsize $2$}
\put(-27,-155){\scriptsize $5$}
\put(-25,-152){\circle{10}}
\put(-5,-155){\scriptsize $8$}
\put(19,-152){\circle{10}}
\put(17,-155){\scriptsize $7$}
\put(41,-155){\scriptsize $6$}
\put(64,-152){\circle{10}}
\put(62,-155){\scriptsize $9$}
\put(77,-155){\scriptsize $...$}
\put(134,-153){\oval(20,9)}
\put(102,-155){\tiny $4g\!-\!4$}
\put(126,-155){\tiny $4g\!-\!5$}
\put(148,-155){\tiny $4g\!-\!6$}
\put(178,-155){\scriptsize $1$}
\end{picture}
\begin{picture}(0,0)(38,-20)
\put(-122,-132){\scriptsize $0$}
\put(178,-138){\scriptsize $0$}
\end{picture}
\end{picture}
\vspace{155bp} \caption{
\label{fig:Jenkins:Strebel:H1111}
A Jenkins--Strebel differential with a single cylinder in the stratum
$\cH(1,\dots,1)$ and its cylinder representation. The distinguished symbols
$3,5\dots,4g-5$ enumerate the saddle connections chosen for contraction.
}
\end{figure}
By construction, it belongs to the principal stratum $\cH(1,\dots,1)$ of Abelian differentials in genus $g$. The cylinder representation of this Jenkins--Strebel differential provides us with the above permutation representing the extended Rauzy class of the principal stratum. Note that an embedded chain of\/ $2g-1$ saddle connections
$$
\begin{CD}
P_1@>X_3>> P_2 @>X_5>> \dots @>X_{4g-5}>>P_{2g-2} \quad
\end{CD}
$$
joins all $2g-2$ zeroes of the corresponding Abelian differential.
Hence by contracting the saddle connections indexed by symbols from the groups indicated below,
\begin{setlength}{\multlinegap}{0pt}
\begin{multline*}
\Big\{\
\underbrace{3,5,\dots,2d_1-1}_{d_1-1},\quad 2d_1+1,\quad \underbrace{2d_1+3,\dots,2(d_1+d_2)-1}_{d_2-1},\quad 2(d_1+d_2)+1,\ \underbrace{\dots \dots}_{\dots}\ ,
\\
\underbrace{\dots \dots}_{\dots}\ ,\
2(d_1+\dots+d_{\noz-1})+3,\quad\underbrace{2(d_1+\dots+d_{\noz-1})+3,\dots,2(d_1+\dots+d_\noz)-1}_{d_\noz-1}
\ \Big\},
\end{multline*}
\end{setlength}
we merge zeroes $P_1, \dots P_{d_1}$ into a single zero of degree $d_1$, zeroes $P_{d_1+1}, \dots P_{d_1+d_2}$ into a single zero of degree $d_2$,..., zeroes $P_{d_1+\dots+d_{\noz-1}+1}, \dots, P_{d_1+\dots+d_\noz}$ into a single zero of degree $d_\noz$. A cylinder representation of the resulting Jenkins--Strebel differential provides us with a permutation obtained from the initial one by erasing the symbols enumerating the contracted saddle connections.
\end{proof}

\subsection{Representatives of components $\cH^{\mathit{odd}}(2d_1, \dots, 2d_\noz)$}
\label{ss:H2222:odd}

Let $d_1,\dots,d_\noz$ be an arbitrary collection of strictly positive integers such that $d_1+\dots+d_\noz=g-1$, where $g\ge 3$.

\begin{Proposition} A permutation obtained by erasing symbols
\label{pr:Jenkins:Strebel:H2222:odd}
\begin{multline*}
\underbrace{4,7,\dots,3d_1-2}_{d_1-1},\quad \underbrace{3d_1+4,\dots,3(d_1+d_2)-2}_{d_2-1},\quad  \underbrace{\dots \dots}_{\dots}\ ,
\\
\underbrace{\dots \dots}_{\dots}\ ,\quad
\underbrace{3(d_1+\dots+d_{\noz-1})+4,\dots,3(d_1+\dots+d_\noz)-2}_{d_\noz-1}
\end{multline*}
in the permutation
\begin{equation}
\label{eq:permutation:H222:odd}
\begin{pmatrix}
0\!&1&2&3&\!&4&5&6&\!&7&8&9&\dots&3g-5&3g-4&3g-3&\\
& &3&2&\!&4&6&5&\!&7&9&8&\dots&3g-5&3g-3&3g-4&1\!&0
\end{pmatrix},
\begin{picture}(0,0)(228,-10)
\put(-32,0){\circle{10}}
\put(-32,-13){\circle{10}}
\put(23,0){\circle{10}}
\put(23,-13){\circle{10}}
\put(101,-1){\oval(32,12)}
\put(101,-14){\oval(32,12)}
\end{picture}
\end{equation}
represents the component $\cH^{\mathit{odd}}(2d_1,\dots,2d_\noz)$.
\end{Proposition}
The proof is based on the following simple Lemma.
\begin{Lemma}
\label{lm:Jenkins:Strebel:H2222:odd}
A Jenkins--Strebel differential with a single cylinder as in Figure \ref{fig:Jenkins:Strebel:H2222:odd} belongs to the component $\cH^{\mathit{odd}}(\underbrace{2,\dots,2}_{g-1})$.
\end{Lemma}

\begin{figure}[htb]
\includegraphics{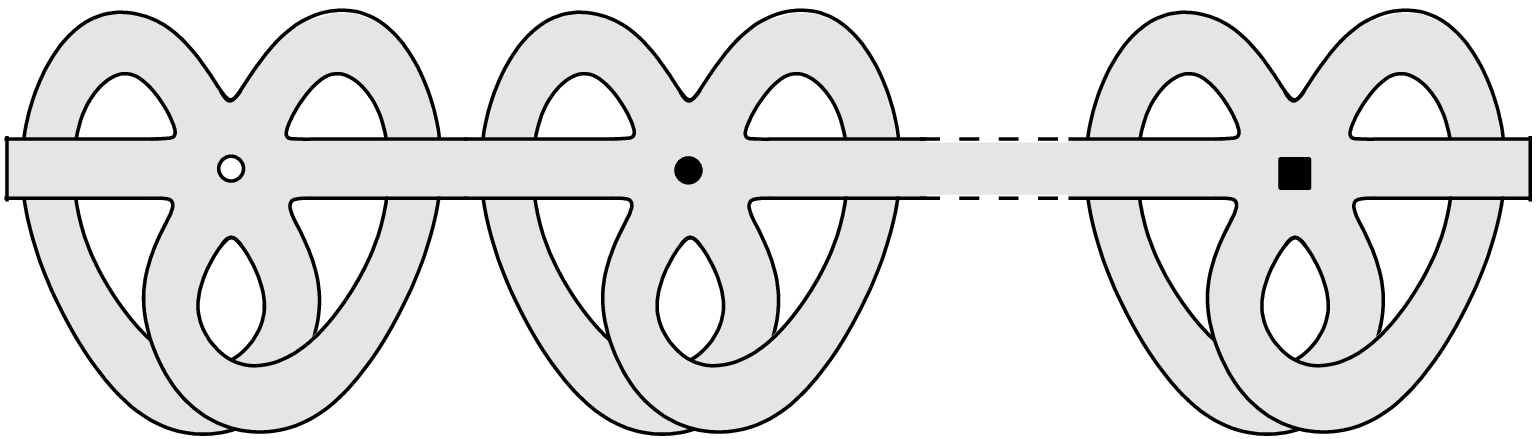}
\includegraphics{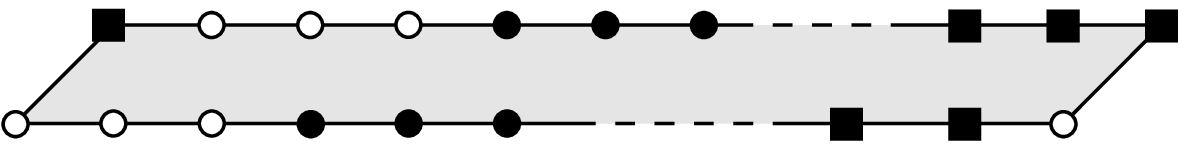}
\begin{picture}(0,0)(2,-10)
\begin{picture}(0,0)(41,-113)
\put(-68,-140){\scriptsize $1$}
\put(-68,-118){\scriptsize $2$}
\put(-29,-118){\scriptsize $3$}
\put(-9,-140){\scriptsize $4$}
\put(12,-118){\scriptsize $5$}
\put(50,-118){\scriptsize $6$}
\put(71,-140){\scriptsize $7$}
\put(81,-139){$\dots$}
\put(105,-118){\scriptsize $3g-4$}
\put(143,-118){\scriptsize $3g-3$}
\put(154,-140){\scriptsize $1$}

\end{picture}

\begin{picture}(0,0)(-1,-8)\put(-94,-116){\scriptsize $1$}
\put(-70,-116){\scriptsize $2$}
\put(-48,-116){\scriptsize $3$}
\put(-24,-113){\circle{10}}
\put(-26,-116){\scriptsize $4$}
\put(-4,-116){\scriptsize $5$}
\put(20,-116){\scriptsize $6$}
\put(43,-116){\scriptsize $7$}
\put(45,-113){\circle{10}}
\end{picture}
\begin{picture}(0,0)(22,-10)
\put(-94,-155){\scriptsize $3$}
\put(-70,-155){\scriptsize $2$}
\put(-46,-152){\circle{10}}
\put(-48,-155){\scriptsize $4$}
\put(-26,-155){\scriptsize $6$}
\put(-4,-155){\scriptsize $5$}
\put(22,-152){\circle{10}}
\put(20,-155){\scriptsize $7$}
\end{picture}
\begin{picture}(0,0)(2,-10)
\put(-124,-135){\scriptsize $0$}
\put(126,-140){\scriptsize $0$}
\end{picture}
\begin{picture}(0,0)(69,-8)
\put(121,-116){\scriptsize $...$}
\put(134,-116){\tiny $3g\!-\!5$}
\put(160,-116){\tiny $3g\!-\!4$}
\put(142,-114){\oval(20,9)}
\put(183,-116){\tiny $3g\!-\!3$}
\end{picture}
\begin{picture}(0,0)(93,-10)
\put(106,-153){\scriptsize $...$}
\put(131,-154){\tiny $3g\!-\!3$}
\put(157,-154){\tiny $3g\!-\!4$}
\put(188,-155){\scriptsize $1$}
\end{picture}
\end{picture}
\vspace{145bp}
\caption{
\label{fig:Jenkins:Strebel:H2222:odd}
A Jenkins--Strebel differential with a single cylinder in the component $\cH^{\mathit{odd}}(2,\dots,2)$ and its cylinder representation. The distinguished symbols $4,7\dots,3g-5$ enumerate saddle connections chosen for contraction.
}
\end{figure}
\begin{proof}[Proof of Lemma \ref{lm:Jenkins:Strebel:H2222:odd}]
The fact that our Abelian differential belongs to the stratum $\cH(\underbrace{2,\dots,2}_{g-1})$ is obvious; it is only necessary to compute the parity of the spin structure of this differential, see \cite{KZ03} and Appendix \ref{a:spin:structure}.

Consider the following collection of closed paths on our flat surface. On the $k$th repetitives pattern of the surface ($k=1,\dots,g-1$) we choose closed paths $\alpha_k, \beta_k$ as indicated in Figure \ref{fig:Jenkins:Strebel:H2222:odd:basis:of:cycles}.
We complete this collection of paths with two paths $\alpha_g$ and $\beta_g$, where $\alpha_g$ is a closed geodesic as in
Figure \ref{fig:Jenkins:Strebel:H2222:odd:basis:of:cycles}
and $\beta_g$ follows the chain of saddle connections
$$
X_1\to X_4\to X_7\to \dots\to X_{3g-5}\to
$$
avoiding the zeroes as indicated in Figure \ref{fig:Jenkins:Strebel:H2222:odd:basis:of:cycles}.
By construction, the paths $\alpha_k,\beta_k$, $k=1,\dots,g$, in each pair have simple transverse intersections, and paths from distinct pairs do not intersect. Hence, we have constructed a canonical basis of cycles.
\begin{figure}[htb]
\includegraphics{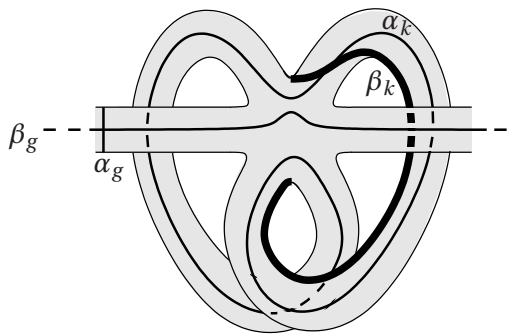}
\begin{picture}(0,0)(38,-106)
\put(-63,-154){$\beta_g$}
\put(-31,-165){$\alpha_g$}
\put(78,-111){$\alpha_k$}
\put(72,-134){$\beta_k$}
\end{picture}
\vspace{120bp} \caption{
\label{fig:Jenkins:Strebel:H2222:odd:basis:of:cycles}
Construction of a convenient canonical basis of cycles.
}
\end{figure}

A smooth path $\gamma$ on a flat surface with trivial linear holonomy defines a natural Gauss map $\gamma\to S^1$ to a circle. Define $\ind(\gamma)$ to be the index  of the Gauss map modulo $2$. It was proved in \cite{KZ03} that having realized a canonical basis of cycles by a collection of smooth connected closed curves avoiding singularities of a flat surface $S\in\cH(2d_1,\dots,2d_\noz)$ one can compute the parity of the spin structure as
\begin{equation}
\label{eq:parity:of:spin:structure}
\varphi(S)\dfn
\sum_{i=1}^g\big(\ind(\alpha_i) +1\big)\big(\ind(\beta_i) +1\big)\mod 2.
\end{equation}
With the exception of the path $\beta_g$ all paths in our family are
everywhere transverse to the horizontal foliation. This implies that the
corresponding indices $\ind(\alpha_1)=\ind(\beta_1)=\dots=\ind(\alpha_g)$
are equal to zero. It is easy to see that $\ind(\beta_g)=g-1$, since every
time $\beta_g$ passes near a zero, the image of the Gauss map makes a complete turn around the circle. Applying now formula \eqref{eq:parity:of:spin:structure} to our collection of paths we see that the parity of the spin structure is odd:
$$
\varphi(S)=
\sum_{i=1}^{g-1}(0 +1)(0 +1)\ +\ (0+1)\big((g-1)+1\big)=
g-1+g\equiv 1 \mod 2
$$
Lemma \ref{lm:Jenkins:Strebel:H2222:odd} is proved.
\end{proof}

\begin{proof}[Proof of Proposition \ref{pr:Jenkins:Strebel:H2222:odd}]
A cylinder representation of our Jenkins--Strebel differential provides us with permutation \eqref{eq:permutation:H222:odd} which by Lemma  \ref{lm:Jenkins:Strebel:H2222:odd}
represents the component $\cH^{\mathit{odd}}(\underbrace{2,\dots,2}_{g-1})$. Note that an embedded chain of\/ $g-1$ saddle connections
$$
\begin{CD}
P_1@>X_4>> P_2 @>X_7>> \dots @>X_{4g-5}>>P_{g-1} \quad
\end{CD}
$$
joins all $g-1$ zeroes of the corresponding Abelian differential.
Hence, contracting the saddle connections indexed by the symbols from the groups indicated in Proposition \ref{pr:Jenkins:Strebel:H2222:odd} we merge a group of zeroes $P_1, \dots P_{d_1}$ of degree $2$ into a single zero of degree $2d_1$, a group of zeroes $P_{d_1+1}, \dots, P_{d_1+d_2}$ of degree $2$ into a single zero of degree $2d_2$, \dots, and  a group of zeroes $P_{d_1+\dots+d_{\noz-1}+1}, \dots, P_{d_1+\dots+d_{\noz-1}+d_\noz}$ of degree $2$ into a single zero of degree $2d_\noz$.
Hence the resulting flat surface belongs to the stratum $\cH(2d_1,\dots,2d_\noz)$.

Recall that merging the zeroes we do not change the parity of the spin
structure, see \cite{KZ03}. (It also follows from formula
\eqref{eq:parity:of:spin:structure} since deforming our flat surface we can
deform the family of the paths in such a way that their indices do
not change.) This implies that the spin structure of the resulting flat
surface has odd parity.

A cylinder representation of the resulting Jenkins--Strebel differential provides us with a permutation obtained from permutation \eqref{eq:permutation:H222:odd} by erasing the symbols enumerating the contracted saddle connections.

It remains to prove that in the two particular cases when we get a surface in  the stratum $\cH(2d,2d)$ or  $\cH(2g-2)$ the resulting flat surface $S$ \emph{does not} belong to the hyperelliptic connected component as soon as $g\ge 3$. Consider the cylinder representation of the surface $S$:
$$
\begin{picture}(0,0)(-2,0)
\put(-3,10){\vector(1,0){0}}
\put(-5,15){\oval(10,10)[bl]}
\put(65,15){\oval(150,10)[t]}
\put(135,15){\oval(10,10)[br]}
\put(-3,-4){\vector(1,0){0}}
\put(-5,-9){\oval(10,10)[tl]}
\put(65,-9){\oval(150,10)[b]}
\put(135,-9){\oval(10,10)[tr]}
\end{picture}
\begin{matrix}
1&2&3&\dots&3g-4&3g-3\\
1&3&2&\dots&3g-3&3g-4
\end{matrix}
\vspace{8pt}
$$
By Proposition \ref{pr:hyperelliptic:Jenkins:Strebel} in Section \ref{ss:H:hyp}, this surface is not hyperelliptic. Proposition \ref{pr:Jenkins:Strebel:H2222:odd} is proved.
\end{proof}

\subsection{Representatives of components $\cH^{\mathit{even}}(2d_1, \dots, 2d_\noz)$}
\label{ss:H2222:even}

Consider a collection $d_1,\dots,d_\noz$ of strictly positive integers satisfying a relation $d_1+\dots+d_\noz=g-1$, where $g\ge 4$.

\begin{Proposition} A permutation obtained by erasing symbols
\label{pr:Jenkins:Strebel:H2222:even}
\begin{multline*}
\underbrace{4,7,\dots,3d_1-2}_{d_1-1},\quad \underbrace{3d_1+4,\dots,3(d_1+d_2)-2}_{d_2-1},\quad  \underbrace{\dots \dots}_{\dots}\ ,
\\
\underbrace{\dots \dots}_{\dots}\ ,\quad
\underbrace{3(d_1+\dots+d_{\noz-1})+4,\dots,3(d_1+\dots+d_\noz)-2}_{d_\noz-1}
\end{multline*}
in the permutation
\begin{equation}
\label{eq:permutation:H222:even}
\begin{pmatrix}
0&1&2&3&4&5&6&&7&8&9&\dots&3g-5&3g-4&3g-3&\\
& &6&5&4&3&2&&7&9&8&\dots&3g-5&3g-3&3g-4&1&0
\end{pmatrix},
\begin{picture}(0,0)(228,-10)
\put(-34,0){\circle{10}}
\put(-34,-13){\circle{10}}
\put(22,0){\circle{10}}
\put(22,-13){\circle{10}}
\put(101,-1){\oval(32,12)}
\put(101,-14){\oval(32,12)}
\end{picture}
\end{equation}
represents the component $\cH^{\mathit{even}}(2d_1,\dots,2d_\noz)$.
\end{Proposition}
The proof is based on the following simple Lemma.
\begin{Lemma}
\label{lm:Jenkins:Strebel:H2222:even}
A Jenkins--Strebel differential with a single cylinder as in Figure \ref{fig:Jenkins:Strebel:H2222:even} belongs to the component $\cH^{\mathit{even}}(\underbrace{2,\dots,2}_{g-1})$, where $g\ge 4$.
\end{Lemma}

\begin{figure}[htb]
\includegraphics{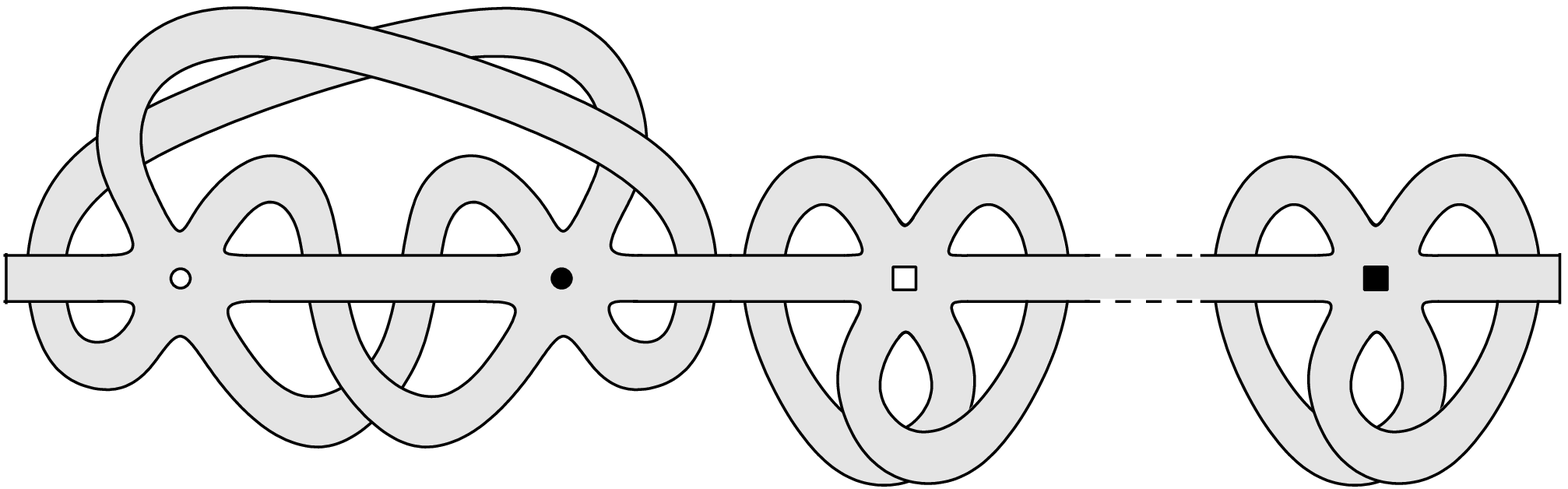}
\includegraphics{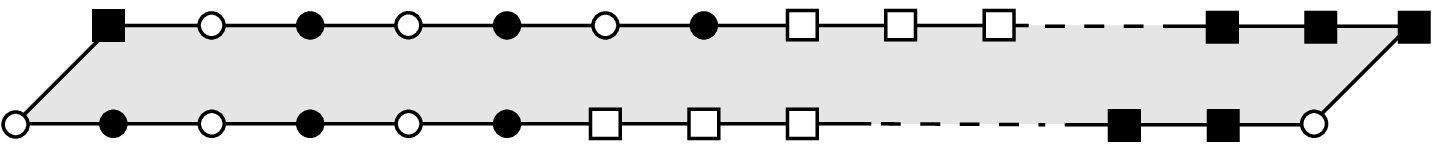}
\begin{picture}(0,0)(88,-93)
\put(-67,-140){\scriptsize $1$}
\put(-67,-100){\scriptsize $2$}
\put(25,-180){\scriptsize $3$}
\put(-7,-140){\scriptsize $4$}
\put(-39,-180){\scriptsize $5$}
\put(54,-100){\scriptsize $6$}
\put(73,-140){\scriptsize $7$}
\put(92,-118){\scriptsize $8$}
\put(130,-118){\scriptsize $9$}
\put(165,-140){$\dots$}
\put(187,-118){\scriptsize $3g-4$}
\put(224,-118){\scriptsize $3g-3$}
\put(236,-140){\scriptsize $1$}
\end{picture}

\begin{picture}(0,0)(0,-13)
\begin{picture}(0,0)(36,2)
\put(-93,-116){\scriptsize $1$}
\put(-70,-116){\scriptsize $2$}
\put(-48,-116){\scriptsize $3$}
\put(-23,-113){\circle{10}}
\put(-25,-116){\scriptsize $4$}
\put(-2,-116){\scriptsize $5$}
\put(20,-116){\scriptsize $6$}
\put(43,-116){\scriptsize $7$}
\put(45,-113){\circle{10}}
\put(66,-116){\scriptsize $8$}
\put(89,-116){\scriptsize $9$}
\put(111,-116){\scriptsize $...$}
\put(133,-116){\tiny $3g\!-\!5$}
\put(157,-116){\tiny $3g\!-\!4$}
\put(141,-114){\oval(20,9)}
\put(180,-116){\tiny $3g\!-\!3$}
\put(-120,-134){\scriptsize $0$}
\put(187,-138){\scriptsize $0$}
\end{picture}
\begin{picture}(0,0)(59,0)
\put(-94,-155){\scriptsize $6$}
\put(-70,-155){\scriptsize $5$}
\put(-46,-152){\circle{10}}
\put(-48,-155){\scriptsize $4$}
\put(-26,-155){\scriptsize $3$}
\put(-4,-155){\scriptsize $2$}
\put(21,-152){\circle{10}}
\put(19,-155){\scriptsize $7$}
\put(42,-155){\scriptsize $9$}
\put(65,-155){\scriptsize $8$}
\put(98,-153){\scriptsize $...$}

\put(133,-155){\tiny $3g\!-\!3$}
\put(156,-155){\tiny $3g\!-\!4$}
\put(185,-155){\scriptsize $1$}
\end{picture}
\end{picture}
\vspace{155bp} \caption{
\label{fig:Jenkins:Strebel:H2222:even}
A Jenkins--Strebel differential with a single cylinder in the component
$\cH^{\mathit{even}}(2,\dots,2)$ and its cylinder representation. The
distinguished symbols $4,7\dots,3g-5$ enumerate the saddle connections
chosen for contraction.
}
\end{figure}
\begin{proof}
The proofs of Lemma \ref{lm:Jenkins:Strebel:H2222:even} and of Proposition
\ref{pr:Jenkins:Strebel:H2222:even} are completely analogous to the ones of
Lemma \ref{lm:Jenkins:Strebel:H2222:odd} and of Proposition
\ref{pr:Jenkins:Strebel:H2222:odd} and we leave them to the reader as an
exercise.
\end{proof}

\begin{NNRemark}
Note that by assumption we have $g\ge 4$, so Figure
\ref{fig:Jenkins:Strebel:H2222:even} contains at least one repetitive
pattern as in Figure \ref{fig:Jenkins:Strebel:H2222:odd:basis:of:cycles}.
Note that for $g=2$ Figure \ref{fig:Jenkins:Strebel:H2222:even} does not make sense. To define the Figure for $g=3$ one has to consider it without repetitive patterns. By Proposition \ref{pr:hyperelliptic:Jenkins:Strebel} the
resulting Jenkins--Strebel differential belongs to
$\cH^{\mathit{hyp}}(2,2)$. This is not a coincidence: any flat surface in
$\cH(2,2)$ or in $\cH(4)$ having even parity of the spin-structure
necessarily belongs to the corresponding hyperelliptic component.
\end{NNRemark}

\subsection{Representatives of components $\cH^{\mathit{hyp}}(g-1,g-1)$ and $\cH^{\mathit{hyp}}(2g-2)$}
\label{ss:H:hyp}

\begin{Proposition}
\label{pr:hyperelliptic:Jenkins:Strebel}
Let $\omega$ be an Abelian Jenkins--Strebel differential with a single cylinder. If it belongs to a hyperelliptic connected component, then a natural cyclic structure \eqref{eq:js:prepermutation} on the set of horizontal saddle connections of $\omega$ has the following form:
\begin{equation}
\label{eq:hyperelliptic:Jenkins:Strebel}
\begin{picture}(0,0)(-3,0)
\put(-3,10){\vector(1,0){0}}
\put(-5,15){\oval(10,10)[bl]}
\put(50,15){\oval(120,10)[t]}
\put(105,15){\oval(10,10)[br]}
\put(-3,-4){\vector(1,0){0}}
\put(-5,-9){\oval(10,10)[tl]}
\put(50,-9){\oval(120,10)[b]}
\put(105,-9){\oval(10,10)[tr]}
\end{picture}
\begin{matrix}
1&2&\dots&k-1&k\\
k&k-1&\dots&2&1
\end{matrix}
\vspace{8pt}\quad.
\end{equation}
The Abelian differential $\omega$
belongs to $\cH^{\mathit{hyp}}(2g-2)$ when $k=2g-1$ is odd and to $\cH^{\mathit{hyp}}(g-1,g-1)$ when $k=2g$ is even.
\end{Proposition}
\begin{proof}
Deforming if necessary the lengths of the horizontal saddle connections (which does not change the connected component of $\omega$) we may assume that they are all distinct (and strictly positive).

A hyperelliptic involution $\tau$ acts on any Abelian differential $\omega$ as $\tau^\ast\omega=-\omega$. Hence the induced isometry of the corresponding flat surface preserves the horizontal foliation and changes the orientation of the foliation. This implies that the isometry $\tau$ acts on the horizontal cylinder by an orientation-preserv\-ing involution interchanging the boundaries of the cylinder. Since zeroes are mapped to zeroes, horizontal saddle connections are isometrically mapped to horizontal saddle connections. Since their lengths are different, every saddle connection $X_i$ is mapped to $X_i$ for $i=1,\dots, k$. This proves that the cyclic orders of the horizontal saddle connections on two components of the cylinder are inverse to each other, or, in other words, that the combinatorics of the cylinder representation of our Jenkins--Strebel differential are encoded by relation \eqref{eq:hyperelliptic:Jenkins:Strebel}. A simple exercise shows that the corresponding flat surface belongs to $\cH^{\mathit{hyp}}(2g-2)$ when $k=2g-1$ is odd and to $\cH^{\mathit{hyp}}(g-1,g-1)$ when $k=2g$ is even, and that in the latter case the symmetry of the cylinder (hyperelliptic involution) interchanges the two zeroes.
\end{proof}

As a corollary we obtain the following proposition which was basically proved by W. A. Veech in \cite{Veech:hyperelliptic:I} in slightly different terms.

\begin{Proposition}[W. A. Veech]
\label{pr:54321}
The permutation
$$
\begin{pmatrix}
1&2&\dots&2g-1&2g\\
2g&2g-1&\dots&2&1
\end{pmatrix}
$$
represents the component $\cH^{\mathit{hyp}}(2g-2)$.

The permutation
$$
\begin{pmatrix}
1&2&\dots&2g&2g+1\\
2g+1&2g&\dots&2&1
\end{pmatrix}
$$
represents the component $\cH^{\mathit{hyp}}(g-1,g-1)$.
\end{Proposition}

\begin{NNRemark}
The Rauzy class  generated by the permutation
$$
\begin{pmatrix}
\noi&\noi-1&\dots&2&1\\
1&2&\dots&\noi-1&\noi
\end{pmatrix}
$$
was studied  in   the   original   paper \cite{Rauzy}  of  Rauzy. In particular it was shown that its cardinality equals $2^{\noi-1}-1$.
Numerous results on hyperelliptic flat surfaces and on their polygonal representations were obtained in the paper of W. A. Veech \cite{Veech:hyperelliptic:I}.
\end{NNRemark}
\section{Representatives of strata of quadratic differentials}
\label{s:quadratic}

In this section we consider meromorphic quadratic differentials with at
most simple poles. The strata of quadratic differentials which are
\emph{not} global squares of Abelian differentials are denoted by
$\cQ(d_1,\dots,d_\noz,\underbrace{-1,\dots,-1}_p)$ By convention,
throughout Section \ref{s:quadratic} we will denote the degrees of the
zeroes of a quadratic differential by $d_k$, \ie $d_k\ge 1$,
$k=1,\dots,\noz$. For brevity we shall often use the notation
$\cQ(d_1,\dots,d_\noz,-1^p)$ to indicate the number $p$ of simple poles.

\subsection*{Empty strata.}
Any collection of zeroes of a meromorphic quadratic differential with $p\ge 0$ simple poles on a surface of genus $g\ge 0$ satisfies the relation
$$
d_1+\dots+d_\noz-p=4g-4.
$$
In contrast with Abelian differentials where any collection of integer
numbers satisfying an analogous relation defines a nonempty stratum, there
are four exceptions for quadratic differentials. They are given by the
following theorem, see \cite{MS}.

\begin{NNTheorem}[H. Masur and J. Smillie]
Any collection $d_1,\dots d_\noz; p$ of positive integers $d_1,\dots d_\noz$ and of a nonnegative integer $p$ satisfying the relation
$$
d_1+\dots+d_\noz-p=4g-4,\quad\text{with }\ g\ge 0
$$
defines a nonempty stratum $\cQ(d_1,\dots,d_\noz,-1^p)$ of meromorphic quadratic differentials with exception for the following four empty strata in genera one and two:
$$
\cQ(\emptyset),\quad \cQ(1,-1),\quad \cQ(3,1),\quad  \cQ(4).
$$
\end{NNTheorem}

In particular, it follows from this theorem that any holomorphic quadratic differential without any zeroes on a surface of genus one, and any holomorphic quadratic differential with a single zero on a surface of genus two is a global square of a holomorphic 1-form.

\subsection*{Hyperelliptic connected components.}
Analogously to the case of Abelian differentials some strata of quadratic differentials contain \emph{hyperelliptic} connected components. They are defined as families of quadratic differentials on hyperelliptic Riemann surfaces invariant under hyperelliptic involution, having specified collections of zeroes and poles and specified action of the involution on the set of zeroes and poles.
They were described by E.~Lanneau in \cite{Lanneau:hyperelliptic}.

\begin{NNTheorem}[E. Lanneau]
Hyperelliptic connected components of strata of meromorphic quadratic differentials with at most simple poles are described by the following list, where $j, k$ are arbitrary nonnegative integer parameters.

\begin{itemize}
\item
Component $\cQ^{\mathit{hyp}}(2j-1,2j-1,2k-1,2k-1)$ of quadratic differentials on hyperelliptic Riemann surfaces, for which the quadratic differentials satisfy the additional requirement that the hyperelliptic involution interchanges the singularities corresponding to the pairs $2j-1,2j-1$ and $2k-1,2k-1$.
\item
Component $\cQ^{\mathit{hyp}}(2j-1,2j-1,4k+2)$ of quadratic differentials on hyperelliptic Riemann surfaces, for which the quadratic differentials satisfy the additional requirement that the hyperelliptic involution interchanges the singularities corresponding to the pair $2j-1,2j-1$.
\item
Component $\cQ^{\mathit{hyp}}(4j+2,4k+2)$ of quadratic differentials on hyperelliptic Riemann surfaces. When $j=k$, quadratic differentials in this component satisfy the additional requirement that both zeroes are fixed points of the hyperelliptic involution.
\end{itemize}
\end{NNTheorem}

\subsection{Classification of connected components
(after E. Lanneau)}
\label{ss:qd:Classification:of:connected:components}
The classification of connected components for quadratic differentials was
obtained by E.~Lanneau in \cite{L}.

\begin{NNTheorem}[E. Lanneau]
All connected components of any stratum  of meromorphic quadratic differentials with at most simple poles on a complex curve of genus $g\ge 3$ are described by the following list:

Each of the four exceptional strata
$$
\cQ(9,-1), \ \cQ(6,3,-1), \ \cQ(3,3,3,-1), \ \cQ(12)
$$
has exactly two connected components.

Each of the following strata
\begin{align*}
&\cQ(2j-1,2j-1,2k-1,2k-1) & j\ge 0,\ k\ge 0,\quad\ \ \, j+k=g\\
&\cQ(2j-1,2j-1,4k+2)      & j\ge 0,k\ge 0,\ j+k=g-1\\
&\cQ(4j+2,4k+2)           & j\ge 0,k\ge 0,\ j+k=g-2
\end{align*}
has exactly two connected components, precisely one   of  which is
hyperelliptic.

All  other strata of meromorphic quadratic differentials with at most simple poles on complex curves of genera $g\ge 3$ are nonempty and connected.
\end{NNTheorem}

Analogous to the case of Abelian differentials some components are
missing in small genera.

\begin{NNTheorem}[M. Kontsevich; H. Masur and J. Smillie; E. Lanneau]
In genus $0$ any stratum is nonempty and connected.

In genus $1$ strata $\cQ(\emptyset)$ and $\cQ(1,-1)$ are
empty; all other strata are nonempty and connected.

In genus $2$ strata $\cQ(3,1)$ and $\cQ(4)$ are empty.
Strata $\cQ(6,-1^2)$ and $\cQ(3,3,-1^2)$ contain exactly two
components, one component is hyperelliptic the other one is not. Any other stratum of meromorphic quadratic differentials with at most simple poles on a Riemann surface of genus $2$ is nonempty and connected.
\end{NNTheorem}

The part of the above theorem concerning genus $0$ is due to M. Kontsevich, see \cite{Kontsevich}. The results concerning empty strata is due to H. Masur and J. Smillie, see \cite{MS}. The remaining part of the theorem is due to E. Lanneau, \cite{L}.

\begin{NNRemark}
A simple invariant distinguishing components of the four exceptional strata
$$
\cQ(9,-1), \ \cQ(6,3,-1), \ \cQ(3,3,3,-1), \ \cQ(12)
$$
has not been found yet. The fact that each of these strata contains exactly two connected components is proved by an explicit computation of corresponding extended Rauzy classes, see Section \ref{ss:exceptional:strata} where we also discuss some geometric properties distinguishing the corresponding pairs of connected components.
\end{NNRemark}

\begin{Convention}
\label{conv::qd:zeroes:at:the:end}
Saying that \emph{a cylindrical generalized permutation $\pi$ represents a stratum $\cQ(d_1,\dots,d_\noz, -1^p)$} we always assume throughout this paper that the corresponding suspension
as in Figure \ref{fig:generalized:iet} has a singularity of degree $d_1$ at the left endpoint of\/ $X$. We do not control anymore the degree at the right endpoint of\/ $X$ (as it was done for Abelian differentials).
\end{Convention}

\subsection{Representatives of strata in genus $0$}

Let $d_1,\dots,d_\noz$ be an arbitrary (possibly empty) collection of strictly positive integers, and let $p= d_1+\dots+d_\noz+4$.

\begin{Proposition}
\label{pr:Q:0}
A generalized permutation obtained by erasing symbols
\begin{multline*}
\underbrace{3,5,\dots,2d_1-1}_{d_1-1},\quad \underbrace{2d_1+3,\dots,2(d_1+d_2)-1}_{d_2-1},\quad  \underbrace{\dots \dots}_{\dots}\ ,
\\
\underbrace{\dots \dots}_{\dots}\ ,\quad
\underbrace{2(d_1+\dots+d_{\noz-1})+3,\dots,2(d_1+\dots+d_\noz)-1}_{d_\noz-1}
\end{multline*}
in the generalized permutation
$$
\begin{pmatrix}
0,\ 2p-6,\ 2p-6\\
2,2,3,\quad 4,4,5,\quad \dots,\quad 2p-8,\ 2p-8,\ 2p-7,\qquad 2p-7,\ 2p-9,\ \dots,\ 3,1,\quad 1, 0\\
\end{pmatrix}
$$
represents the stratum $\cQ(d_1,\dots,d_\noz,-1^p)$.
\end{Proposition}
For an empty collection $\{d_1,\dots,d_\noz\}=\emptyset$ one has $p=4$ and the permutation above should be read as
$$
\begin{pmatrix}
0,\ 2,\ 2\\
1,\ 1,\ 0
\end{pmatrix}.
$$
\begin{proof}
The proof is left to the reader as an exercise.
\end{proof}

\subsection{Representatives of strata in genus 1}

\begin{Proposition}
\label{pr:Q:1}
Let $d_1,\dots,d_\noz$ be a collection of strictly positive integers, such that $d_1+\dots+d_\noz=p\ge 2$. A generalized permutation obtained by erasing symbols
\begin{multline*}
\underbrace{2,4,\dots,2(d_1-1)}_{d_1-1},\quad \underbrace{2(d_1+1),\dots,2(d_1+d_2-1)}_{d_2-1},\quad  \underbrace{\dots \dots}_{\dots}\ ,
\\
\underbrace{\dots \dots}_{\dots}\ ,\quad
\underbrace{2(d_1+\dots+d_{\noz-1}+1),\dots,2(d_1+\dots+d_\noz-1)}_{d_\noz-1}
\end{multline*}
in the generalized permutation
$$
\begin{pmatrix}
0,1,\quad 2,3,3,\quad 4,5,5,\quad \dots,\quad 2p-2,\; 2p-1,\; 2p-1\\
2,4,\dots,2p-2,\quad 1,\; 2p,\; 2p,\quad 0
\end{pmatrix}
$$
represents the stratum $\cQ(d_1,\dots,d_\noz,-1^p)$.
\end{Proposition}
\begin{figure}[htb]
\includegraphics{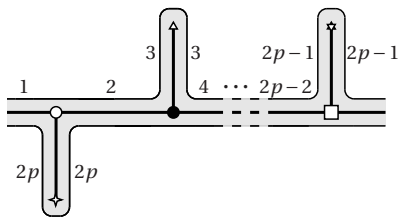}
\begin{picture}(0,0)(-15,-107)
\put(-82,-140){\scriptsize $1$}
\put(-49,-140){\scriptsize $2$}
\put(-34,-127){\scriptsize $3$}
\put(-17,-127){\scriptsize $3$}
\put(-14,-140){\scriptsize $4$}
\put(-5,-138){$\dots$}
\put(9,-140){\scriptsize $2p-2$}
\put(10,-127){\scriptsize $2p-1$}
\put(42,-127){\scriptsize $2p-1$}

\put(-83,-172){\scriptsize $2p$}
\put(-61,-172){\scriptsize $2p$}
\end{picture}
\vspace{80bp}
\caption{
\label{fig:Jenkins:Strebel:quadratic:g1}
A Jenkins--Strebel differential with a single cylinder in the stratum $\cQ(1^{p},-1^p)$,  $p\ge 2$, in genus $1$
}
\end{figure}
\begin{proof}
Note that by the theorem of H. Masur and J. Smillie cited in the
beginning of Section \ref{s:quadratic} the strata $\cQ(\emptyset)$ and $\cQ(1,-1)$ are empty. Thus, a meromorphic quadratic differential with simple poles in genus $g=1$ has at least $p=2$ poles. This implies that the relation
$$
|X_3|+\dots+|X_{2p-1}|=|X_{2p}|
$$
for the lengths of horizontal saddle connections of a Jenkins--Strebel differential such as in Figure \ref{fig:Jenkins:Strebel:quadratic:g1} always has a strictly positive solution. The resulting quadratic differential belongs to the stratum $\cQ(1^p,-1^p)$.

Note that none of the distinguished saddle connections $X_2, X_4, \dots, X_{2p-2}$ is involved in the relation above. It means that contracting any subcollection of these distinguished saddle connections does not affect the lengths of any other saddle connection.

It remains to note that an embedded chain of\/ $p-1$ distinguished saddle connections
$$
\begin{CD}
\quad@>X_2>> \quad @>X_4>> \quad @>X_{2p-2}>> \quad
\end{CD}
\begin{picture}(0,0)(150,0)

\thicklines
\put(0,0){\circle{6}}
\put(46,0){\circle*{7}}
\put(93,-3){$\dots$\quad}
\put(140,-4){\framebox(5,5){\quad}}
\end{picture}
$$
joins all $p$ zeroes of our quadratic differential.
\end{proof}

\subsection{Representatives of connected strata in genus $2$}

\begin{Proposition}
\label{pr:qd:g2}
The generalized permutation
$$
\begin{pmatrix}
0,6,1,5,6,4,3\\
1,2,3,4,2,5,0\\
\end{pmatrix}
$$
represents the stratum $\cQ(1,1,1,1)$.
The generalized permutations obtained from this one by erasing symbols $\{1\}$ and $\{1,3\}$ represent the strata $\cQ(2,1,1)$ and $\cQ(2,2)$, respectively.

Let $d_1,\dots,d_\noz,p$ be a collection of strictly positive integers satisfying the relation $d_1+\dots+d_\noz=p+4$. A generalized permutation obtained by erasing symbols
\begin{multline*}
\underbrace{1,2,\dots,d_1-1}_{d_1-1},\quad \underbrace{d_1+1,\dots,d_1+d_2-1}_{d_2-1},\quad  \underbrace{\dots \dots}_{\dots}\ ,
\\
\underbrace{\dots \dots}_{\dots}\ ,\quad
\underbrace{d_1+\dots+d_{\noz-1}+1,\dots,d_1+\dots+d_\noz-1}_{d_\noz-1}
\end{multline*}
in the generalized permutation
$$
\begin{pmatrix}
0,\ 2p+6,\ 1,\ 2p+5,\ 2p+6,\ p+4,\ p+3\\
1,\dots,p+4,\ \ p+2,p+5,p+5,\ \ p+1,p+6,p+6,\dots,3,2p+4,2p+4,\ \ 2,2p+5, 0
\end{pmatrix}
$$
represents the stratum $\cQ(d_1,\dots,d_\noz,-1^p)$.
\end{Proposition}
\begin{figure}[htb]
\includegraphics{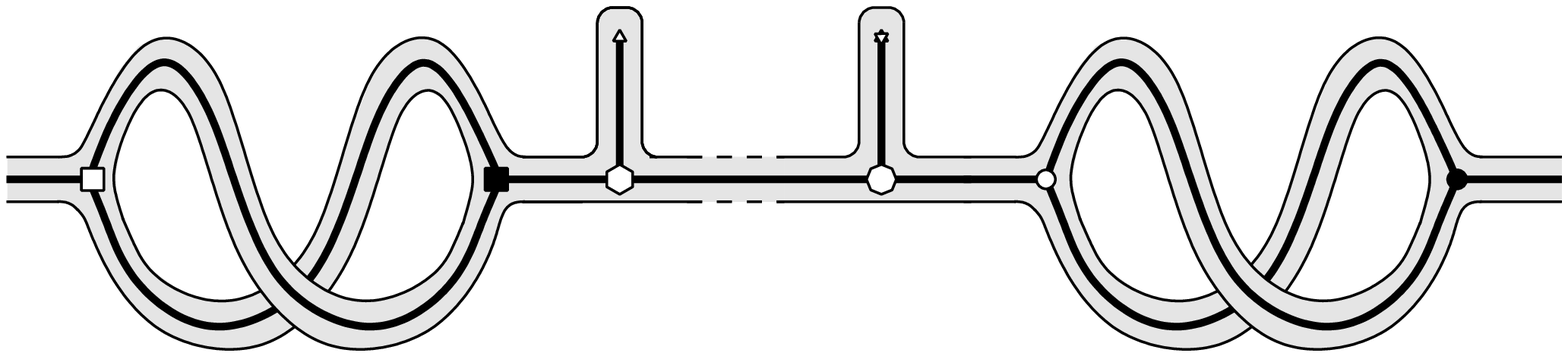}
\includegraphics{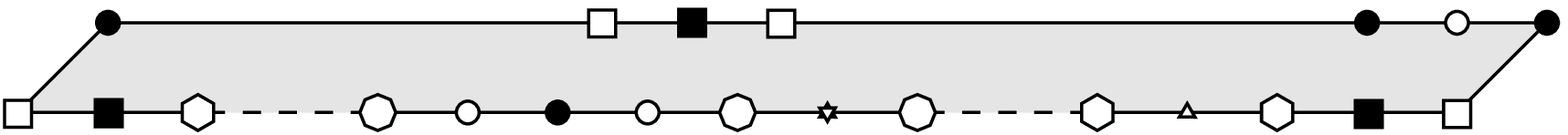}
\begin{picture}(0,0)(94,-102)
\put(-73,-139){\scriptsize $2\!p\!\!+\!\!6$}
\put(-37,-148){\scriptsize $1$}
\put(-35,-146){\circle{10}}
\put(2,-113){\scriptsize $2p+5$}
\put(154,-113){\scriptsize $p+4$}
\put(136,-178){\scriptsize $p+3$}
\put(144,-176){\oval(20,9)}
\put(236,-139){\scriptsize $2\!p\!\!+\!\!6$}
\put(40,-159){\scriptsize $2$}
\put(42,-157){\circle{10}}
\put(64,-159){\scriptsize $3$}
\put(66,-157){\circle{10}}
\put(75,-158){$\dots$}
\put(93,-159){\scriptsize $p\!+\!1$}
\put(101,-157){\oval(20,9)}
\put(120,-159){\scriptsize $p\!+\!2$}
\put(128,-157){\oval(20,9)}
\put(117,-126){\scriptsize $p\!+\!5$}
\put(89,-125){$\dots$}
\put(62,-126){\scriptsize $2p\!+\!4$}
\end{picture}

\begin{picture}(0,0)(0,-13)
\begin{picture}(0,0)(-40,-13)
\put(-139,-115){\tiny $2p\!+\!6$}
\put(-70,-115){\scriptsize $1$}
\put(-54.5,-115){\tiny $2\!p\!\!+\!\!5$}
\put(20,-115){\tiny $2p\!+\!6$}
\put(95,-115){\tiny $p\!+\!4$}
\put(115,-115){\tiny $p\!+\!3$}
\end{picture}
\begin{picture}(0,0)(44,-20)
\put(-117,-153){\scriptsize $1$}
\put(-98,-153){\scriptsize $2$}
\put(-81,-153){\scriptsize $3$}
\put(-72,-152){\scriptsize $...$}
\put(-63,-153){\tiny $p\!+\!1$}
\put(-42,-153){\tiny $p\!+\!2$}
\put(-22,-153){\tiny $p\!+\!3$}
\put(-2,-153){\tiny $p\!+\!4$}
\put(16,-153){\tiny $p\!+\!2$}
\put(37,-153){\tiny $p\!+\!5$}
\put(56,-153){\tiny $p\!+\!5$}
\put(76,-153){\tiny $p\!+\!1$}
\put(93,-152){\scriptsize $...$}
\put(105,-153){\scriptsize $3$}
\put(116.5,-153){\tiny $2\!p\!\!+\!\!4$}
\put(135,-153){\tiny $2\!p\!\!+\!\!4$}
\put(162,-153){\scriptsize $2$}
\put(176,-153){\tiny $2\!p\!\!+\!\!5$}
\put(-121,-136){\scriptsize $0$}
\put(205,-140){\scriptsize $0$}
\end{picture}
\end{picture}
\vspace{140bp}
\caption{
\label{fig:Jenkins:Strebel:Q:genus:2}
A Jenkins--Strebel differential with a single cylinder in the stratum $\cQ(1^{p+4},-1^p)$, $p\ge 0$, and its cylinder representation. Distinguished symbols $1,\dots,p+3$ enumerate saddle connections chosen for contraction.
}
\end{figure}
\begin{proof}
It follows from Figure \ref{fig:Jenkins:Strebel:Q:genus:2} that the generalized permutation
$$
\begin{pmatrix}
0,6,1,5,6,4,3\\
1,2,3,4,2,5,0\\
\end{pmatrix}
$$
represents the stratum $\cQ(1,1,1,1)$.
The basic relation \eqref{eq:relation:top:bottom} between the lengths of horizontal saddle connections has the form $|X_6|=|X_2|$. Thus, we may contract any of\/ $X_1, X_3$ or both of them without changing the lengths of other saddle connections. From Figure \ref{fig:Jenkins:Strebel:Q:genus:2} we conclude that in this way we get generalized permutations representing strata  $\cQ(2,1,1)$ and $\cQ(2,2)$ correspondingly. Note that by the theorem of H. Masur and J. Smillie strata $\cQ(3,1)$ and $\cQ(4)$ are empty.
Hence we have represented all strata of \emph{holomorphic} quadratic differentials in genus $g=2$.

Consider now a Jenkins--Strebel quadratic differential with a single cylinder as in Figure \ref{fig:Jenkins:Strebel:Q:genus:2} having at least one simple pole. Relation \eqref{eq:relation:top:bottom} between the lengths of horizontal saddle connections now has the form
$$
|X_{2p+6}|=\left(|X_2|+|X_3|+\dots+|X_{p+2}|\right)\ +\ \left(|X_{p+5}|+\dots+|X_{2p+4}|\right)
$$
and always admits strictly positive solutions. It is clear from Figure \ref{fig:Jenkins:Strebel:Q:genus:2} that the corresponding Jenkins--Strebel quadratic differential has a single cylinder and belongs to the stratum $\cQ(1^{p+4},-1^p)$. It remains to note that the $p+3$ saddle connections $X_1, \dots, X_{p+3}$ join all $p+4$ simple zeroes and that contracting any subcollection of these distinguished saddle connections yields a relation which still admits a strictly positive solution for the lengths of remaining ones.
\end{proof}

\subsection{Representatives of connected strata in genus $g\ge 3$}

Let $d_1,\dots,d_\noz$ be a collection of strictly positive integers,  $p$ a nonnegative integer, and let $g\ge 3$ be an integer. Assume that these integer data satisfy the relation $d_1+\dots+d_\noz-p=4g-4$.

Consider a generalized permutation represented by the following two strings of symbols (see also Figure \ref{fig:Jenkins:Strebel:quadratic:g:ge3}). The top string has the form
$$
\pi^{top}=
\big(0,1,2,3,V(0),\dots,V(g-3),W(0),\dots,W(p-1),4g+p-4\big).
$$
Here the word $V(k)$ is composed of the following six symbols:
$$
V(k)=
4+4k,4g\!-\!1\!+\!p\!+\!2k,7\!+\!4k,4g\!+\!p\!+\!2k,6\!+\!4k,4g\!-\!1\!+\!p\!+\!2k,5\!+\!4k,4g\!+\!p\!+\!2k
$$
and the word $W(l)$ is composed of the following three symbols:
$$
W(l)=4g-4+l,\ 6g+p-5+l,\ 6g+p-5+l\ .
$$
By convention, when $p=0$ the words $W(0),\dots,W(p-1)$ are omitted in $\pi^{top}$.

The bottom string has the form
\begin{equation*}
\begin{split}
\pi^{\mathit{bot}}&=\big(
4g-3+p,3,4g-2+p,2,4g-3+p,1,4g-2+p,\quad
4,5,\dots,4g-4+p,\ 0
\big)
\end{split}
\end{equation*}

\begin{Proposition}
\label{pr:qd:g:ge3}
For any genus $g\ge 3$, and any collection of integers $d_1, \dots, d_\noz,p$ as above the generalized permutation obtained by erasing symbols
\begin{multline*}
\underbrace{1,2,\dots,d_1-1}_{d_1-1},\quad \underbrace{d_1+1,\dots,d_1+d_2-1}_{d_2-1},\quad  \underbrace{\dots \dots}_{\dots}\ ,
\\
\underbrace{\dots \dots}_{\dots}\ ,\quad
\underbrace{d_1+\dots+d_{\noz-1}+1,\dots,d_1+\dots+d_\noz-1}_{d_\noz-1}
\end{multline*}
in the generalized permutation
$\begin{pmatrix}
\pi^{\mathit{top}}\\ \pi^{\mathit{bot}}
\end{pmatrix}$
represents the stratum $\cQ(d_1,\dots,d_\noz,-1^p)$.
\end{Proposition}
\begin{figure}[htb]
\includegraphics{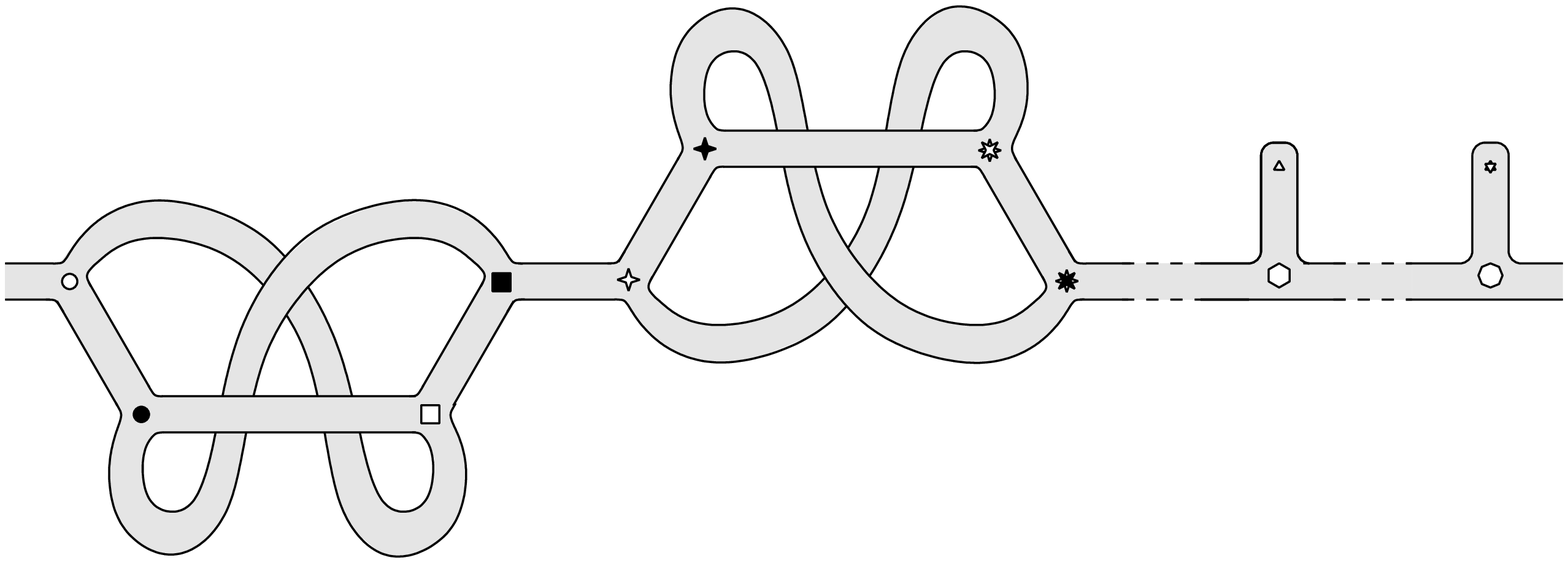}
\begin{picture}(0,0)(73,-85)
\put(166,-116){\scriptsize $\overbrace{\hspace*{55pt}}^{p\ \text{\emph{patterns}}}$}
\begin{picture}(0,0)(0,0)
\put(35,-165){\scriptsize $\underbrace{\hspace*{110pt}}_{g-2\ \text{\emph{patterns}}}$}
\put(-69,-128){\scriptsize $1$}
\put(-41,-186){\scriptsize $2$}
\put(-16,-128){\scriptsize $3$}
\put(19,-142){\scriptsize $4$}
\put(15,-127){\tiny $4g\!-1\!+\!p$}
\put(54,-88){\scriptsize $7$}
\put(119,-127){\tiny $4g\!+\!p$}
\put(75,-113){\scriptsize $6$}
\put(100,-88){\scriptsize $5$}
\put(130,-142){\scriptsize $8$}
\put(140,-141){$\dots$}
\end{picture}
\begin{picture}(0,0)(-132,0)
\put(20,-157){\tiny $4g\!-\!4$}
\put(40,-156){$\dots$}
\put(42,-134){\tiny $6g\!+\!p\!-\!5$}
\put(55,-157){\tiny $4g\!+\!p\!-\!5$}
\put(85,-157){\tiny $4g\!+\!p\!-\!4$}
\put(87,-134){\tiny $6g\!+\!2p\!-\!6$}
\end{picture}
\begin{picture}(0,0)(0,0)
\put(-108,-172){\tiny $4g\!-\!3\!+\!p$}
\put(-5,-172){\tiny $4g\!-\!2\!+\!p$}
\end{picture}
\end{picture}
\vspace{120bp} \caption{
\label{fig:Jenkins:Strebel:quadratic:g:ge3}
A Jenkins--Strebel differential with a single cylinder in the stratum $\cQ(1^{4g+p-4},-1^p)$, $p\ge 0$,  in genus $g\ge 3$
}
\end{figure}
\begin{proof}
Consider a Jenkins--Strebel quadratic differential with a single cylinder as in Figure \ref{fig:Jenkins:Strebel:quadratic:g:ge3}.  Relation \eqref{eq:relation:top:bottom} between the lengths of the horizontal saddle connections has the form
\begin{equation*}
\begin{split}
|X_{4g-3+p}|+|X_{4g-2+p}|=\Big(
\left(|X_{4g-1+p}|+|X_{4g+p}|\right)+&\dots+
\left(|X_{6g-7+p}|+|X_{6g-6+p}|\right)\Big)+\\
&+\left(|X_{6g+p-5}|+\dots
+|X_{6g+2p-6}|\right)
\end{split}
\end{equation*}
Note that by convention the genus $g$ is at least $3$. Thus, even for $p=0$, when the second sum in the right part of the equation is missing, the equation admits strictly positive solutions.

It is clear from Figure \ref{fig:Jenkins:Strebel:quadratic:g:ge3} that the corresponding Jenkins--Strebel quadratic differential has a single cylinder and belongs to the stratum $\cQ(1^{4g+p-4},-1^p)$. It remains to note that the $4g-5+p$ saddle connections $X_1, \dots, X_{4g-5+p}$ join all $4g-4+p$ simple zeroes and that none of these saddle connections is involved in the relation. This implies that we can contract any subcollection of these distinguished saddle connections without affecting the remaining ones.
\end{proof}

\subsection{Representatives of hyperelliptic components}
Representatives of hyperelliptic connected components
of the strata of quadratic differentials were constructed by E. Lanneau in \cite{L}, Section 4.1. In the theorem below we slightly modify the original notations.
\begin{NNTheorem}[E. Lanneau]
For any pair of nonnegative integer parameters $r,s$ the generalized permutation
$$
\begin{pmatrix}
0,A,\quad 1,2, \dots,s,\quad A,\quad s+1,s+2\dots,s+r\\
s+r,\dots,s+2,s+1,\quad B,\quad s,\dots,2,1,\quad B,0
\end{pmatrix}
$$
represents a Jenkins--Strebel quadratic differential with a single cylinder in a hyperelliptic connected component. The table below specifies the corresponding stratum.
\end{NNTheorem}
\noindent
Here ``$A, B$'' are just symbols of our alphabet.
By convention when $r=0$ (or when $s=0$) the  sequences $1,2,\dots,r$ (correspondingly $r+1,r+2,\dots,r+s$) are empty.
$$
\begin{array}{|c|c|l|}
\hline && \\ [-\halfbls]
r & s & \qquad\quad\text{Embodying stratum} \\
[-\halfbls] &&\\ \hline && \\ [-\halfbls]
2j+1 & 2k+1 & \cQ( \ 4j+2 \ , \ 4k+2 \ ) \\
[-\halfbls] &&\\ \hline && \\ [-\halfbls]
2j   & 2k+1 & \cQ( \ 2j-1 \ , \ 2j-1 \ , \ 4k+2 \ ) \\
[-\halfbls] &&\\ \hline && \\ [-\halfbls]
2j+1 & 2k   & \cQ( \ 4j+2 \ , \ 2k-1 \ , \ 2k-1 \ ) \\
[-\halfbls] &&\\ \hline && \\ [-\halfbls]
2j   & 2k   & \cQ( \ 2j-1 \ , \ 2j-1 \ ,\  2k-1 \ , \ 2k-1 \ ) \\
[-\halfbls] &&\\ \hline
\end{array}
$$

We complete this section with a criterion of E. Lanneau \cite{L} characterizing all cylindrical generalized permutations representing hyperelliptic connected components. This result will be used in the next section; it is analogous to Proposition \ref{pr:54321} in Section \ref{ss:H:hyp}.

\begin{Proposition}[E. Lanneau]
\label{pr:hyperelliptic:JS:implies:symmetry}
Let $q$ be a Jenkins--Strebel quadratic differential with a single cylinder. Suppose that it is not a global square of an Abelian differential. If it belongs to one of hyperelliptic connected components, then a natural cyclic structure \eqref{eq:js:prepermutation} on the set of horizontal saddle connections of\/ $q$ has one of the following two forms:
either it has the form

\begin{equation}
\label{eq:Lanneau:form}
\begin{picture}(0,0)(3,0)
\put(2,10){\vector(1,0){0}}
\put(0,15){\oval(10,10)[bl]}
\put(105,15){\oval(220,10)[t]}
\put(210,15){\oval(10,10)[br]}
\put(2,-4){\vector(1,0){0}}
\put(0,-9){\oval(10,10)[tl]}
\put(105,-9){\oval(220,10)[b]}
\put(210,-9){\oval(10,10)[tr]}
\end{picture}
\smallskip
\begin{matrix}
A\to 1\to 2\to \dots \to s \to A \to s+1 \to \dots \to s+r\\
s+r\to \dots \to s+1 \to B \to s \to \dots \to 2\to 1 \to B
\end{matrix}
\end{equation}
as in the theorem above, or it has the form
\begin{equation}
\label{eq:123123}
\begin{picture}(0,0)(-42,0)
\put(2,10){\vector(1,0){0}}
\put(0,15){\oval(10,10)[bl]}
\put(95,15){\oval(200,10)[t]} 
\put(190,15){\oval(10,10)[br]} 
\end{picture}
\smallskip
\begin{picture}(0,0)(3,0)
\put(2,-4){\vector(1,0){0}}
\put(0,-9){\oval(10,10)[tl]}
\put(140,-9){\oval(290,10)[b]}
\put(280,-9){\oval(10,10)[tr]}
\end{picture}
\begin{matrix}
1\to 2\to \dots \to r+1\to 1\to 2\to \dots \to r+1\\
r+2\to r+3\to  \dots \to r+s+2\to r+2\to r+3\to  \dots \to r+s+2
\end{matrix}
\end{equation}
and the corresponding stratum is specified by the table above. (As before, ``$A$'' and ``$B$'' are symbols of the alphabet.)
\end{Proposition}

\subsection{Representatives of nonhyperelliptic components}

We need to construct representatives of nonhyperelliptic components
of the strata where such a component exists. Note that such components appear only in genus 2 and higher. We can apply Propositions \ref{pr:qd:g2} and \ref{pr:qd:g:ge3} to obtain a representative of any stratum in genus 2 and higher, in particular of the stratum which contains a nonhyperelliptic component. It remains to prove that our candidate does not get to the hyperelliptic component of this stratum.

\begin{Proposition}
\label{pr:nonhyperelliptic}
For any nonconnected stratum of quadratic differentials containing a hyperelliptic component, the complementary nonhyperelliptic component can be represented by a Jenkins--Strebel differential with a single cylinder as in Propositions \ref{pr:qd:g2} and \ref{pr:qd:g:ge3}.
\end{Proposition}
\begin{proof}
By Theorems of M. Kontsevich and E. Lanneau stated in Section \ref{ss:qd:Classification:of:connected:components} any stratum in genera $0$ and $1$ is connected.

There are two disconnected strata $\cQ(3,3,-1^2)$ and $\cQ(6,-1^2)$ in genus $g=2$. The construction from Proposition \ref{pr:qd:g2} produces the following generalized permutations (and Jenkins--Strebel differentials with a single cylinder) representing these strata
correspondingly:
$$
\begin{pmatrix}
0,10,9,10,6\\
3,6,7,7,3,8,8,9,0
\end{pmatrix}
\quad\text{and}\quad
\begin{pmatrix}
0,10,9,10,6\\
6,7,7,8,8,9,0
\end{pmatrix}.
$$
By Proposition \ref{pr:hyperelliptic:JS:implies:symmetry} from the previous section the corresponding Jenkins--Strebel differentials belong to nonhyperelliptic components.

In genus $g\ge 3$ any stratum containing a hyperelliptic component also contains a nonhyperelliptic one. Consider such a stratum.
We can reach it by
applying the construction suggested by Proposition \ref{pr:qd:g:ge3}.  Recall that relation \eqref{eq:relation:top:bottom} between the lengths of horizontal saddle connections has the form
\begin{equation*}
\begin{split}
|X_{4g-3+p}|+|X_{4g-2+p}|=\Big(
\left(|X_{4g-1+p}|+|X_{4g+p}|\right)+&\dots+
\left(|X_{6g-7+p}|+|X_{6g-6+p}|\right)\Big)+\\
&+\left(|X_{6g+p-5}|+\dots
+|X_{6g+2p-6}|\right)
\end{split}
\end{equation*}
Thus, Proposition \ref{pr:hyperelliptic:JS:implies:symmetry} implies that our Jenkins--Strebel differential belongs to a nonhyperelliptic component.
\end{proof}

\subsection{Representatives of connected components of the four exceptional strata $\cQ(9,-1), \cQ(6,3,-1), \cQ(3,3,3,-1), \cQ(12)$}
\label{ss:exceptional:strata}

\begin{table}[t]
$$
\begin{array}{|c|c|c|}
\multicolumn{3}{c}{\text{Genus }g=3}\\
\hline && \\ [-\halfbls]
\text{Component}&
\text{Cardinality}&
\text{Representative}\\
\text{of a}&
\text{of extended}&
\text{of extended}\\
\text{stratum}&
\text{Rauzy class}&
\text{Rauzy class}\\
[-\halfbls] &&\\
\hline && \\ [-\halfbls]
\cQ^{\mathit{irr}}(3,3,3,-1) & 88\,374 &
\begin{pmatrix}0,1,2,3,4,5,1,6,2,3,4,5,6,7\\7,8,8,0\end{pmatrix} \\
[-\halfbls] &&\\ \hline && \\ [-\halfbls]
\cQ^{\mathit{irr}}(6,3,-1) & 72\,172 &
\begin{pmatrix}0,1,2,3,4,5,1,2,3,4,5,6\\6,7,7,0\end{pmatrix} \\
[-\halfbls] &&\\ \hline && \\ [-\halfbls]
\cQ^{\mathit{irr}}(9,-1) & 12\,366 &
\begin{pmatrix}0,1,2,3,4,1,2,3,4,5\\5,6,6,0\end{pmatrix} \\
[-\halfbls] &&\\ \hline\hline && \\ [-\halfbls]
\cQ^{\mathit{reg}}(3,3,3,-1) & 612\,838 &
\begin{pmatrix}0,1,2,3,4,2,3,5,5,6\\7,1,8,7,8,4,6,0\end{pmatrix} \\
[-\halfbls] &&\\ \hline && \\ [-\halfbls]
\cQ^{\mathit{reg}}(6,3,-1) & 531\,674 &
\begin{pmatrix}0,1,2,3,1,2,4,4,5\\6,7,6,7,3,5,0\end{pmatrix} \\
[-\halfbls] &&\\ \hline && \\ [-\halfbls]
\cQ^{\mathit{reg}}(9,-1) & 95\,944 &
\begin{pmatrix}0,1,2,1,2,3,3,4\\5,6,5,6,4,0\end{pmatrix} \\
[-\halfbls] &&\\
\hline
\multicolumn{3}{c}{}\\
[-\halfbls]
\multicolumn{3}{c}{\text{Genus }g=4}\\
\hline && \\ [-\halfbls]
\cQ^{\mathit{irr}}(12) & 146\,049 &
\begin{pmatrix}0,1,2,3,4,5,6,5\\7,6,4,7,3,2,1,0\end{pmatrix} \\
[-\halfbls] &&\\
\hline&&\\
[-\halfbls]
\cQ^{\mathit{reg}}(12) & 881\,599 &
\begin{pmatrix}0,1,2,1,2,3,4,3,4,5\\5,6,7,6,7,0\end{pmatrix} \\
[-\halfbls] &&\\
\hline
\end{array}
$$
\caption{
\label{tab:exceptional:strata}
Representatives of all connected components of the four exceptional strata $\cQ(9,-1)$, $\cQ(6,3,-1)$, $\cQ(3,3,3,-1)$, $\cQ(12)$}
\end{table}

\begin{Proposition}
\label{pr:exceptional:strata}
\hspace*{-5truept}
Each of the strata $\cQ(9,-1)$, $\cQ(6,3,-1)$, $\cQ(3,3,3,-1)$, and $\cQ(12)$ contains exactly two connected components.

Cylindrical generalized permutations representing these components and cardinalities of the corresponding extended Rauzy classes are given in Table \ref{tab:exceptional:strata}.
\end{Proposition}
\begin{proof}
This proposition can be obtained by a naive direct computation, which allows one to find all connected components of low-dimensional strata.

One can construct the sets of all irreducible  generalized permutations (see \cite{BL} for a combinatorial definition) of up to $9$ elements.
These permutations can be sorted by which strata they represent.
Having a set of irreducible permutations one can apply the combinatorial construction of the extended Rauzy class (see \cite{BL} and appendix \ref{a:generalized:rauzy:operations}) to decompose this set into a disjoint union of extended Rauzy classes. This direct calculation shows that each of the four exceptional strata $\cQ(9,-1)$, $\cQ(6,3,-1)$, $\cQ(3,3,3,-1)$, and $\cQ(12)$ has exactly two distinct connected components.
\end{proof}

\begin{NNRemark}
Actually, the initial computations were performed differently. Instead of working with extended Rauzy classes, which are really huge, it is more advantageous to find generalized permutations of some particular form, which are still present in every extended Rauzy class (say, those ones, for which both lines have the same length, see Lemma \ref{lm:ballanced} in Appendix \ref{a:generalized:rauzy:operations}). To prove that a stratum is connected it is sufficient to find all generalized permutations of this particular form corresponding to  the given stratum, and then verify that all these generalized permutations are connected by some chains of operations $a,b,c$.
\end{NNRemark}

\appendix
\section{Adjacency of special strata}

Currently there are no known simple invariants distinguishing pairs of connected components of the four exceptional strata. Of course having a flat surface in one of these strata one can consider an appropriate segment, consider the ``first-return map'' defined by a foliation in a transverse direction, figure out the resulting generalized permutation and then use a computer to check to which of the two corresponding extended Rauzy classes it belongs. However, this approach is not very exciting.

A much more geometric approach uses \emph{configurations of homologous saddle connections} (see \cite{MZ} for a definition and for a geometric description). Some configurations are specific for flat surfaces in specific connected components.
The classification of configurations of homologous saddle connections
admissible for each individual connected component of the four exceptional strata is research in progress by E. Lanneau and C. Boissy.

We formulate here just one result to give a flavor of this approach. Basically, it says that for a quadratic differential in, say, $\cQ^{\mathit{reg}}(9,-1)$ one can merge the simple pole and the zero and obtain a nondegenerate flat surface in $\cQ(8)$, while merging the simple pole with the zero for a quadratic differential in $\cQ^{\mathit{irr}}(9,-1)$ we necessarily degenerate the Riemann surface. This statement was conjectured by the author and proved by E. Lanneau in \cite{L}.

\begin{NNTheorem}[E. Lanneau]
Almost every quadratic differential in each component $\cQ^{\mathit{reg}}(3,3,3,-1)$, $\cQ^{\mathit{reg}}(6,3,-1)$, $\cQ^{\mathit{reg}}(9,-1)$ has infinitely many saddle connections of multiplicity one joining the simple pole with one of the zeroes.

No quadratic differential in
$\cQ^{\mathit{irr}}(3,3,3,-1)$, $\cQ^{\mathit{irr}}(6,3,-1)$, $\cQ^{\mathit{irr}}(9,-1)$
admits a saddle connection of multiplicity one joining the simple pole with one of the zeroes.
\end{NNTheorem}

To prove the second part of the Theorem E. Lanneau suggests
the following argument. The stratum $\cQ(8)$ is connected. Hence,
following \cite{KZ03} and applying classical deformation theory one concludes that any stratum $\cQ(d_1,\dots,d_\noz,-1^p)$ in genus three has a unique connected component adjacent to the minimal stratum $\cQ(8)$. By Proposition \ref{pr:exceptional:strata} each exceptional stratum $\cQ(3,3,3,-1)$, $\cQ(6,3,-1)$, $\cQ(9,-1)$ contains two connected components. Hence, one of the components in each pair is not adjacent to $\cQ(8)$. We shall see that that these components are $\cQ^{\mathit{irr}}(3,3,3,-1)$, $\cQ^{\mathit{irr}}(6,3,-1)$, $\cQ^{\mathit{irr}}(9,-1)$ respectively.

To prove the remaining part of the theorem it is basically sufficient to present at least one surface with at least one saddle connection of multiplicity one. This is done using
appropriate horizontal saddle connections for explicit examples of Jenkins--Strebel differentials with a single cylinder.

In the remaining part of this section we illustrate the technique of
\cite{KZ03} and of \cite{L} which uses Jenkins--Strebel differentials with a single cylinder to study adjacency of the strata. Following Lanneau \cite{L} we describe adjacency of three regular and of three irregular components in genus 3.

\subsection*{Regular components.}
Our representatives of three ``regular'' connected components $\cQ^{\mathit{reg}}(3,3,3,-1)$, $\cQ^{\mathit{reg}}(6,3,-1)$, $\cQ^{\mathit{reg}}(9,-1)$ in genus three are obtained from the canonical generalized permutation representing one-cylinder Jenkins--Strebel differentials in $\cQ(1^9,-1)$ by  the construction described in Proposition \ref{pr:qd:g:ge3}. We present the corresponding generalized permutations in Table \ref{tab:adjacency:regular}; we do not reenumerate the symbols to make modifications more traceable.

By construction, the associated Jenkins--Strebel differential in
$\cQ^{\mathit{reg}}(9,-1)$ can be obtained from the one in $\cQ^{\mathit{reg}}(6,3,-1)$
by contraction of the horizontal saddle connection $X_6$, which results in merging the zeroes of degrees $6$ and $3$ into a single zero of degree $9$. Similarly the Jenkins--Strebel differential in
$\cQ^{\mathit{reg}}(6,3,-1)$ can be obtained from the one in $\cQ^{\mathit{reg}}(3,3,3,-1)$
by contraction of the horizontal saddle connection $X_3$, which results in merging a pair of zeroes of degree $3$ into a single zero of degree $6$. Thus, $\cQ^{\mathit{reg}}(3,3,3,-1)$ is adjacent to $\cQ^{\mathit{reg}}(6,3,-1)$ which is in turn adjacent to $\cQ^{\mathit{reg}}(9,-1)$.

\begin{table}[ht]
$$
\begin{array}{|c|c|}
\hline & \\ [-\halfbls]
\cQ(1^{9},-1) &
\begin{pmatrix}
0,\ 1\ ,\ 2\ ,3,\quad 4\ ,12,\ 7\ ,13,6,12,\ 5\ ,13,\quad 8\ ,14,14,\quad 9\\
10,3,11,\ 2\ ,10,\ 1\ ,11,\quad 4\ ,\ 5\ ,6,\ 7\ ,\ 8\ ,9,\quad 0
\end{pmatrix}
\begin{picture}(0,0)(203,-10)
\put(-14,0){\circle{10}}
\put(0,0){\circle{10}}
\put(30,0){\circle{10}}
\put(58,0){\circle{10}}
\put(110,0){\circle{10}}
\put(145,0){\circle{10}}
\end{picture}
\begin{picture}(0,0)(205,3)
\put(35,0){\circle{10}}
\put(63,0){\circle{10}}
\put(99,0){\circle{10}}
\put(112,0){\circle{10}}
\put(135,0){\circle{10}}
\put(149,0){\circle{10}}
\end{picture}
\\
[-\halfbls] &\\ \hline & \\ [-\halfbls]
\cQ^{\mathit{reg}}(3,3,3,-1) &
\begin{pmatrix}
0,\ 3\ ,\ 12,13,6,12,13,\ 14,14,\ 9\\
10,\ 3\ ,11,10,11,\ 6,9,\ 0
\end{pmatrix}
\begin{picture}(0,0)(122,-10)
\put(0,0){\circle{10}}
\put(21,-13){\circle{10}}
\end{picture}
\\
[-\halfbls] &\\ \hline & \\ [-\halfbls]
\cQ^{\mathit{reg}}(6,3,-1) &
\begin{pmatrix}
0,\ 12,13,\ 6\ ,12,13,\ 14,14,\ 9\\
10,11,10,11,\ 6\ ,9,\ 0
\end{pmatrix}
\begin{picture}(0,0)(122,-10)
\put(40,0){\circle{10}}
\put(75,-13){\circle{10}}
\end{picture}
\\
[-\halfbls] &\\ \hline & \\ [-\halfbls]
\cQ^{\mathit{reg}}(9,-1) &
\begin{pmatrix}
0,\ 12,13,12,13,\ 14,14,\ 9\\
10,11,10,11,\ 9,\ 0
\end{pmatrix}\\
[-\halfbls] &\\
\hline
\end{array}
$$
\caption{
\label{tab:adjacency:regular}
Adjacency of regular components of exceptional strata in genus 3.}
\end{table}

Note that the basic relation \eqref{eq:relation:top:bottom} is the same for all cylindrical generalized permutations in Table \ref{tab:adjacency:regular}. It has the following form:
$$
|X_{10}|+|X_{11}|=
|X_{12}|+|X_{13}|+|X_{14}|.
$$
Thus, we can continuously contract the saddle connection $X_{14}$
compensating this by increasing $|X_{13}|$ or  $|X_{12}|$ or both; such deformation respects the relation above.
Combinatorially this results in erasing the symbol ``$14$'' in the corresponding permutations, see Table \ref{tab:merging:regular}.

\begin{table}[ht]
$$
\begin{array}{|c|c|}
\hline & \\ [-\halfbls]
\cQ(3,3,2) &
\begin{pmatrix}
0,3,\ 12,13,6,12,13,\  9\\
10,3,11,10,11,\ 6,9,\ 0
\end{pmatrix}\\
[-\halfbls] &\\ \hline & \\ [-\halfbls]
\cQ(6,2) &
\begin{pmatrix}
0,\ 12,13,6,12,13,\  9\\
10,11,10,11,\ 6,9,\ 0
\end{pmatrix}\\
[-\halfbls] &\\ \hline & \\ [-\halfbls]
\cQ(8) &
\begin{pmatrix}
0,\ 12,13,12,13,\  9\\
10,11,10,11,\ 9,\ 0
\end{pmatrix}\\
[-\halfbls] &\\
\hline
\end{array}
$$
\caption{
\label{tab:merging:regular}
Merging simple pole to a zero for regular components of exceptional strata in genus 3}.
\end{table}

Geometrically this means that the saddle connection $X_{14}$ joining the simple pole to a corresponding zero has multiplicity one. Merging the simple pole with the corresponding zero we get nondegenerate Jenkins--Strebel differentials in the strata $\cQ(3,3,2)$, $\cQ(6,2)$, $\cQ(8)$
respectively. In our ribbon graph representation of Jenkins--Strebel differentials as in Figure \ref{fig:Jenkins:Strebel:quadratic:g:ge3} we continuously contract the appendix with a single simple pole compensating the missing length on the complementary component by making the saddle connection $X_{13}$ longer. These considerations show, in particular, that regular connected components are adjacent to the minimal stratum $\cQ(8)$, and hence, irregular components are not.

\subsection*{Irregular (irreducible) components.}
The connected component $\cQ^{\mathit{irr}}(9,-1)$ is adjacent to
$\cQ^{\mathit{irr}}(6,3,-1)$ which is in turn adjacent to $\cQ^{\mathit{irr}}(3,3,3,-1)$. To verify this, it is sufficient to note that
the generalized permutation representing the component $\cQ^{\mathit{irr}}(6,3,-1)$ in Table \ref{tab:exceptional:strata} can be obtained from the one representing $\cQ^{\mathit{irr}}(3,3,3,-1)$
by erasing the symbol ``$6$'' followed by reenumeration. Note that the saddle connection $X_6$ is not involved in the corresponding relation between lengths of horizontal saddle connections. Similarly the generalized permutation representing $\cQ^{\mathit{irr}}(9,-1)$ in our table can be obtained from the one representing $\cQ^{\mathit{irr}}(6,3,-1)$
by erasing the symbol ``$5$'' followed by reenumeration. The saddle connection $X_5$ is not involved in the corresponding relation between lengths of horizontal saddle connections.

\section{Rauzy operations for generalized permutations}
\label{a:generalized:rauzy:operations}

The Rauzy operations $a$ and $b$ on permutations were introduced by G. Rauzy in \cite{Rauzy}. The related dynamical system (a discrete analog of the Teichm\"uller geodesic flow) was extensively studied in the breakthrough paper \cite{V82} of W. A. Veech. This technique was developed and applied in \cite{AGY,AF,AV,Bf,MMY,Z96}, and in other papers; see also surveys \cite{Vi,Y03,Y05,Z:Houches}.

In this section we extend the combinatorial definitions of
Rauzy operations $a$ and $b$ to generalized permutations; see
paper \cite{BL} of C. Boissy and E. Lanneau for a comprehensive study of the geometry, dynamics and combinatorics of generalized permutations and of Rauzy--Veech induction on generalized permutations.

\subsection*{Rauzy operations.}
Operations are not defined for a generalized permutation $\pi$ when the rightmost entries of both lines of\/ $\pi$ are represented by the same symbol.

Rauzy operation $a$ acts as follows.
Denote by ``$0$'' the symbol representing the rightmost entry in the top line. Recall that every symbol appears in a generalized permutation exactly twice. To distinguish two appearances of the symbol ``$0$'' we denote by $0_1$ the rightmost entry in the top line and by $0_2$ its twin.

If\/ $0_2$ belongs to the bottom line, we erase the rightmost entry $\beta_s$ of the bottom line and insert $\beta_s$ to the right of\/ $0_2$, as in the example below:
$$
\begin{pmatrix}
\alpha_1, \dots, \alpha_r, 0_1  \\
\beta_1, \dots, \beta_{k-1}, 0_2,\ \beta_k,\dots, \beta_{s-1},\  \beta_s\
\end{pmatrix}
\begin{picture}(0,0)(16,4)
\put(2,0){\circle{14}}
\put(-62,-4){\vector(0,1){0}}
\put(-30,-7){\oval(64,10)[b]}
\end{picture}
\stackrel{a}{\longmapsto}
\begin{pmatrix}
\alpha_1, \dots, \alpha_r, 0_1  \\
\beta_1, \dots, \beta_{k-1}, 0_2, \beta_s, \beta_k, \dots,\beta_{s-1}
\end{pmatrix}
\vspace{5pt}
$$

If\/ $0_2$ belongs to the same line as $0_1$ (\ie to the top one), we have to apply an additional test. If the rightmost symbol $\beta_s$ in the bottom line appears in this line twice and all other symbols of the bottom line appear in this line only once, the corresponding operation is not defined. Otherwise, we remove the rightmost symbol $\beta_s$ in the bottom line and insert it to the left of\/ $0_2$. Thus, $\beta_s$ moves to the top line, as in the example below:
$$
\begin{pmatrix}
\dots, \alpha_k,\ 0_2, \alpha_{k+1}, \dots, \alpha_r, 0_1  \\
\beta_1, \dots, \beta_{s-1},\ \beta_s
\end{pmatrix}
\begin{picture}(0,0)(37,4)
\put(2,0){\circle{14}}
\put(-50,14){\vector(0,-1){0}}
\put(-45,-7){\oval(94,10)[b]}
\put(-92,-7){\line(0,1){25}}
\put(-71,17){\oval(42,10)[t]}
\end{picture}
\stackrel{a}{\longmapsto}
\begin{pmatrix}
\dots, \alpha_k,  \beta_s, 0_2, \alpha_{k+1}, \dots, \alpha_r, 0_1  \\
\beta_1   \dots, \beta_{s-1}
\end{pmatrix}
\vspace{8pt}
$$

Operation $b$ is defined in complete analogy with operation $a$ by interchanging the words ``top'' and ``bottom'' in the definition. Namely, it acts as:

$$
\begin{pmatrix}
\beta_1, \dots, \beta_{k-1}, 0_2,\ \beta_k,\dots, \beta_{s-1},\  \beta_s\ \\
\alpha_1, \dots, \alpha_r, 0_1
\end{pmatrix}
\begin{picture}(0,0)(16,-10)
\put(2,0){\circle{14}}
\put(-62,3){\vector(0,-1){0}}
\put(-30,7){\oval(64,10)[t]}
\end{picture}
\stackrel{b}{\longmapsto}
\begin{pmatrix}
\beta_1, \dots, \beta_{k-1}, 0_2, \beta_s, \beta_k, \dots,\beta_{s-1}\\
\alpha_1, \dots, \alpha_r, 0_1  \\
\end{pmatrix}
\vspace{5pt}
$$
when $0_1$ and $0_2$ are in different lines, and it acts as:

$$
\begin{pmatrix}
\beta_1, \dots, \beta_{s-1},\ \beta_s\\
\dots, \alpha_k,\ 0_2, \alpha_{k+1}, \dots, \alpha_r, 0_1
\end{pmatrix}
\begin{picture}(0,0)(37,-10)
\put(2,0){\circle{14}}
\put(-50,-15){\vector(0,1){0}}
\put(-45,7){\oval(94,10)[t]}
\put(-92,-18){\line(0,1){25}}
\put(-71,-18){\oval(42,10)[b]}
\end{picture}
\stackrel{b}{\longmapsto}
\begin{pmatrix}
\beta_1   \dots, \beta_{s-1}\\
\dots, \alpha_k,  \beta_s, 0_2, \alpha_{k+1}, \dots, \alpha_r, 0_1
\end{pmatrix}
\vspace{8pt}
$$
when they are both in the bottom line. In the second case the operation $b$ is not defined when the rightmost symbol $\beta_s$ in the top line appears in the top line twice and all other symbols of the top line appear in this line only once.

In addition to Rauzy operations $a$, $b$ we define one more operation denoted by $c$. It reverses the elements in each line and then interchanges the lines:
$$
\begin{pmatrix}
\alpha_1,\alpha_2,\dots,\alpha_p  \\
\beta_1,\beta_2, \dots, \beta_q
\end{pmatrix}\
\stackrel{c}{\longmapsto}\
\begin{pmatrix}
\beta_p, \dots, \beta_2,\beta_1\\
\alpha_p,\dots,\alpha_2,\alpha_1
\end{pmatrix}
$$

Rauzy operations have the following geometric interpretation.
Consider a closed flat surface $S$ represented by a meromorphic quadratic differential with at most simple poles. Consider a horizontal segment $X$, the ``first-return map'' of the vertical foliation to $X$ in the sense of Section \ref{ss:analogs:of:iet},
and the corresponding generalized permutation $\pi$. Compare the lengths of the rightmost intervals in the top and in the bottom lines and denote the longer of two intervals by $X_0$ and the shorter one by $X_{\beta_s}$. Consider now a horizontal segment $X'\subset X$ obtained by shortening $X$ on the right by chopping out of\/ $X$ a piece of length $|X_{\beta_s}|$. A generalized interval exchanged map induced by the vertical foliation on $X'$ will be associated to a permutation $a(\pi)$ or $b(\pi)$ depending whether $X_0$ was in the top or in the bottom line.

Recall that the notion of ``top'' or ``bottom'' shores of a slit along $X$ is a matter of convention: it depends on a choice of orientation of\/ $X$ which is not canonical for quadratic differentials. Choosing the opposite orientation of\/ $X$ we get a generalized interval-exchange transformation with generalized permutation $c(\pi)$.

\subsection*{Rauzy classes and extended Rauzy classes.}

\begin{Definition}
A generalized permutation is called \emph{irreducible} if it can be
realized by a generalized interval-exchange transformation $T$ satisfying
the following conditions:
\begin{enumerate}
\item
there exists a Riemann surface $S$ and a meromorphic quadratic differential
$q$ with at most simple poles on it, such that the vertical foliation of\/
$q$ is minimal;
\item
there exists an oriented horizontal segment $X$ adjacent
at its left endpoint to a singularity of\/ $q$ such that the first-return map of the vertical foliation to $X$ induces the generalized interval-exchange transformation $T$.
\end{enumerate}
\end{Definition}

This natural definition is however extremely inefficient.
A purely combinatorial criterion for irreducibility of a generalized permutation was not known for quite a long time; it was recently elaborated by C. Boissy and E. Lanneau in \cite{BL}.

It follows from the definition above that if\/ $\pi$ is irreducible, $a(\pi)$ and $b(\pi)$ are also irreducible (when the corresponding operation is applicable).

\begin{Definition}
A \emph{Rauzy class} $\mathfrak{R}(\pi)$ of an irreducible generalized
permutation $\pi$ is a minimal set containing $\pi$ and invariant under the
Rauzy operations $a$, $b$.
\end{Definition}

Here invariance under operations $a$, $b$ is understood in the sense ``when
applicable'': a Rauzy class may contain a generalized permutation for which
one of the two operations is not applicable.

Up to this point there is not so much difference between combinatorial definitions of Rauzy operations and Rauzy classes
of ``true'' permutations and ``generalized'' permutations (with the exception of the fact that for some generalized permutations one of the Rauzy operations might be not defined). There is, however, a radical difference with operation $c$. Namely, C. Boissy and E. Lanneau have observed that for an irreducible generalized permutation $\pi$
the image $c(\pi)$ is not necessarily irreducible, while for true permutations the operation $c$ preserves this property. Thus, in their definition of extended Rauzy classes C. Boissy and E. Lanneau make the corresponding correction (compared to ``true'' permutations):

\begin{Definition}[C. Boissy and E. Lanneau]
An \emph{extended Rauzy class} $\mathfrak{R}_{ex}(\pi)$ of an irreducible
generalized permutation $\pi$ is the intersection of a minimal set
containing $\pi$ and invariant under operations $a$, $b$, $c$ with the set of
irreducible generalized permutations.
\end{Definition}

Conjecturally an extended Rauzy class can be defined in the following alternative way. Let us modify the definition of the operation
$c$ by saying that $c(\pi)$ is not defined when $c(\pi)$ is not irreducible. Having an irreducible permutation $\pi$ we define a minimal set $\mathfrak{R}'_{ex}(\pi)$
containing $\pi$ invariant under Rauzy operations $a$, $b$, $c$ (where for some generalized permutations in $\mathfrak{R}'_{ex}(\pi)$ some operations might not be applicable).

\begin{Conjecture}
The sets $\mathfrak{R}_{ex}(\pi)$ and $\mathfrak{R}'_{ex}(\pi)$ coincide.\footnote{When the paper was already in press C.~Boissy has
announced a proof of this conjecture.}
\end{Conjecture}

C. Boissy and E. Lanneau prove in \cite{BL} that for any generalized permutation $\pi_1$ in a Rauzy class $\mathfrak{R}(\pi_0)$ of an irreducible generalized permutation $\pi_0$ one has $\mathfrak{R}(\pi_0)=\mathfrak{R}(\pi_1)$. This implies that an extended Rauzy class (whichever of the two definitions above we choose) is a disjoint union of Rauzy classes.

In the statement of the conjecture below we allow a degree $d_i$ in $\cQ(d_1,\dots,d_\noz)$ to have value ``$-1$''; in this case it corresponds to a simple pole.

\begin{Conjecture}
\label{conj:extended:equals:k:Rauzy:calsses}
Let an extended Rauzy class $\mathfrak{R}_{ex}$ represent a connected component of a stratum $\cQ(d_1, \dots, d_\noz)$ or of a stratum $\cH(d_1, \dots, d_\noz)$.  Let
$k$ be the number of pairwise distinct entries in $\{d_1,\dots,d_\noz\}$.
Then $\mathfrak{R}_{ex}$ is a disjoint union of exactly $k$ nonempty Rauzy classes.
\end{Conjecture}

For example, the extended Rauzy class representing the stratum $\cH(2,1,1)$ is a disjoint union of two Rauzy classes. Permutations
of the first one correspond to horizontal intervals adjacent on the left to a simple zero while permutations of the second Rauzy class correspond to horizontal intervals adjacent on the left to a zero
of degree two. This conjecture is confirmed for all low-dimensional strata, see Appendix \ref{a:tables}.

\begin{NNRemark}
Note that following Convention \ref{conv:zeroes:at:the:end} in Section
\ref{ss:Interval:exchange:transformations:and:Rauzy:classes} and Convention
\ref{conv::qd:zeroes:at:the:end} in Section
\ref{ss:qd:Classification:of:connected:components} the first entry in a set
$(d_1,\dots,d_\noz)$ of singularity data defining a connected component of
a stratum $\cH(d_1,\dots,d_\noz)$ or of a stratum
$\cQ(d_1,\dots,d_\noz,-1^p)$ was distinguished: a cylindrical permutation
representing the corresponding component was constructed in such way that
the degree of a zero associated to the left endpoint was equal to $d_1$.
For example, Proposition \ref{pr:abelian:general} constructs
representatives of distinct Rauzy classes for $\cH(2,1,1)$ and for
$\cH(1,2,1)$.

For the strata of quadratic differentials
we did not construct generalized permutations having simple poles associated to their left endpoints. This can easily be done by an appropriate cyclic move of an appropriate line in a cylindrical representation \eqref{eq:js:prepermutation} of the corresponding permutation. Otherwise, modulo Conjecture \ref{conj:extended:equals:k:Rauzy:calsses} we have constructed representatives of all Rauzy classes (and not only extended Rauzy classes) of irreducible nondegenerate generalized permutations.
\end{NNRemark}

\subsection*{Inverse generalized permutations and balanced generalized permutations.}
\strut\break We complete this section with a short discussion of two notions related to generalized permutations.

\begin{Definition}
We say that a generalized permutation $\pi^{-1}$ is \emph{inverse} to a generalized permutation $\pi$ if it can be obtained from $\pi$
by interchanging the lines (up to a standard convention on equivalence of alphabets).
$$
\pi=\begin{pmatrix}
\alpha_1,\alpha_2,\dots,\alpha_p  \\
\beta_1,\beta_2, \dots, \beta_q
\end{pmatrix}
\qquad
\pi^{-1}=
\begin{pmatrix}
\beta_1,\beta_2, \dots, \beta_q\\
\alpha_1,\alpha_2,\dots,\alpha_p
\end{pmatrix}
$$
\end{Definition}

The operation of taking inverse interacts with the Rauzy operations in the following way:
\begin{equation}
\label{eq:inverse:and:Rauzy}
a(\pi^{-1})=\big(b(\pi)\big)^{-1}\qquad
b(\pi^{-1})=\big(a(\pi)\big)^{-1}\qquad
c(\pi^{-1})=\big(c(\pi)\big)^{-1}
\end{equation}

\begin{Lemma}
The generalized permutations $\pi$ and $\pi^{-1}$
are simultaneously irreduci\-b\-le or not. When they are irreducible,
they belong to the same extended Rauzy class.
\end{Lemma}
\begin{proof}
Consider a closed flat surface $S$ represented by a meromorphic quadratic
differential $q$ with at most simple poles and some polygonal pattern $\Pi$
for $S$. We can associate to $S$ a conjugate flat surface $\bar S$, which
is obtained by identifying the corresponding pairs of sides of the polygon
$\overline\Pi$ obtained from $\Pi$ by a symmetry with respect to the
horizontal axes. The lemma above is equivalent to its geometric version
below.
\end{proof}

\begin{Lemma}
The map $S\to \bar S$ preserves connected components of the strata.
\end{Lemma}
\begin{proof}
Clearly $S$ and $\bar S$ belong to the same stratum.
By the result of W. Veech \cite{Veech:hyperelliptic:I} which immediately generalizes to quadratic differentials, a
surface in a hyperelliptic component can be represented by a centrally-symmetric polygonal pattern $\Pi$, where the central symmetry acts as a hyperelliptic involution. It is easy to see that the central symmetry of\/ $\overline\Pi$ induces on $\bar S$ a hyperelliptic involution which acts on singularities in the same way as the one associated to $S$. Hence if\/ $S$ belongs to a hyperelliptic component, so does $\bar S$.

Suppose that $S$ belongs to a stratum $\cH(2d_1,\dots,2d_\noz)$.
Consider a canonical basis of cycles on $S$ realized by smooth simple closed curves avoiding singularities. The images of these curves under the pointwise map $S\to \bar S$ represent a canonical basis of cycles on $\bar S$. The indices of the corresponding curves (see Appendix \ref{a:spin:structure}) counted modulo $2$ are the same as the indices of the original curves. Thus, the surfaces $S$ and $\bar S$ share the same parity of the spin structure, see \eqref{eq:parity:of:spin}.

Suppose that $S$ and $\bar S$ belong to one of the exceptional strata
$\cQ(3,3,3,-1)$, $\cQ(6,3,-1)$, $\cQ(9,-1)$. It follows from
\eqref{eq:inverse:and:Rauzy} that the operation of taking inverse
bijectively maps an extended Rauzy class to an extended Rauzy class. We
know from a direct computation that there are exactly two extended Rauzy
classes associated to each of the four exceptional strata, and that the
cardinalities in each pair are different, see Table
\ref{tab:exceptional:strata}. Hence, the operation of taking inverse maps
all of these extended Rauzy classes to themselves and so the geometric
realization $S\mapsto\bar S$ of this operation maps the connected
components to themselves.
\end{proof}

\begin{NNRemark}
The Lemma above cannot be extended to all closed $GL(2,\R{})$-invariant suborbifold of the moduli spaces. Counterexamples can be found among orbits of square-tiled surfaces (arithmetic Veech surfaces). Namely, the $GL(2,\R{})$-orbits of the two rightmost surfaces in Figure \ref{fig:6tiled:surfaces} are distinct and the map $S\mapsto \bar{S}$ interchanges the two orbits.
\end{NNRemark}

\begin{Definition}
A generalized permutation $\pi$ is called \emph{balanced} if the lines of\/ $\pi$ have the same length:
$$
\pi=\begin{pmatrix}
\alpha_1,\alpha_2,\dots,\alpha_p  \\
\beta_1,\beta_2, \dots, \beta_p
\end{pmatrix}.
$$
\end{Definition}

\begin{Lemma}
\label{lm:ballanced}
The Rauzy class of any irreducible generalized permutation $\pi$ contains a balanced generalized permutation.
\end{Lemma}
\begin{proof}
In the proof below we do not reenumerate the symbols of a generalized
permutation after applying operations $a$ and $b$.

Suppose that the bottom line of\/ $\pi$ is shorter than the top one. Suppose that the twin $0_2$ of the rightmost element $0_1$ in the bottom line is also located in the bottom line. Since the bottom line is shorter than the top one, the top line contains more than one pair of identical symbols, and hence operation $b$ is applicable to $\pi$. Note that by applying the operation $b$ we made
the bottom line longer and the top one shorter while the twin $0_2$ of the rightmost element $0_1$ in the bottom line stayed in the bottom line. Recursively applying this argument we eventually get a balanced permutation.

Suppose now that the twin $0_2$ of the rightmost element $0_1$ in the
bottom line is located in the top line. If our generalized permutation is a
``true'' permutation it is already balanced. If not, the bottom line
contains at least one symbol $\beta_k$ which has his twin in the bottom
line. It follows from results \cite{BL} of C. Boissy and E. Lanneau on the
dynamics of Rauzy--Veech induction that applying operations $a$ and $b$ in
all possible ways we can eventually make every symbol appear in the
rightmost position. Consider a chain of operations $a$ and $b$ which place
$\beta_k$ in the rightmost position and suppose that our chain does not
contain a shorter one placing $\beta_k$ in the rightmost position. By
construction, the corresponding generalized permutation $\pi'$ has
$\beta_k$ in the rightmost position in the bottom line and the twin of\/
$\beta_k$ also belongs to the bottom line. Hence, either our chain of
operations $a$ and $b$ already contains a balanced generalized permutation,
or we can apply our previous argument to $\pi'$ to obtain one.
\end{proof}

\section{Parity  of a spin  structure  in terms of a permutation}
\label{a:spin:structure}

Consider a permutation $\pi$ representing a stratum of Abelian differentials of the form $\cH(2d_1,\dots,2d_\noz)$.  In this section we describe how to compute the parity of the spin structure of a surface associated to the permutation $\pi$. For strata different from $\cH(2g-2)$ or $\cH(2k,2k)$, the parity of the spin structure determines a connected component $\cH^{\mathit{even}}(2d_1,\dots,2d_\noz)$ or $\cH^{\mathit{odd}}(2d_1,\dots,2d_\noz)$ represented by $\pi$.

There remains ambiguity with the strata $\cH(2g-2)$ or $\cH(2k,2k)$
since they contain three connected components when $g\ge 4$.
The parity $\phi(S)$ of the spin structure for
surfaces from hyperelliptic connected components of these strata is expressed as follows, see \cite{KZ03}:
$$
\phi(S)=\begin{cases}
\left[\frac{g+1}{2}\right]\ (mod\ 2)&
\text{for }S\in \cH_1^{\mathit{hyp}}(2g-2)\\ \\
k+1\ (mod\ 2)&
\text{for }S\in\cH_1^{\mathit{hyp}}(2k,2k)
\end{cases}
$$
where $[x]$ denotes the integer part of a number $x$.

\subsection*{Parity of the spin structure. Definition.} Consider a flat surface $S$ in a stratum $\cH(2d_1, \dots, 2d_\noz)$.
Consider a smooth simple closed curve $\gamma$ on $S$ such that $\gamma$
does not pass through singularities of the flat metric. Our flat structure
defines a trivialization of the tangent bundle to $S$ that is punctured at
the singularities. Thus, we can consider the Gauss map from $\gamma$ to a
unit circle; this map associates to a point $x$ of\/ $\gamma$ the unit
tangent vector $T_x\gamma$ at this point.

We define $\ind(\gamma)\in  \Z{}$ as the  degree of the Gauss
map. When we follow the curve tracing how the tangent vector turns with respect to the vertical direction, we observe a total change of the angle along the curve of\/ $2\pi\cdot \ind(\gamma)$.

Now  take   a   symplectic   homology   basis   $\{  a_1,  b_1, a_2, b_2, \dots, a_g, b_g\}$ in which  the  intersection  matrix  has  the
canonical form: $ a_i\circ a_j = b_i\circ b_j =0$, $ a_i\circ b_j
=  \delta_{ij}$,  $1\le i,j  \le  g$. Though  such  basis is  not
unique,  traditionally  it  is  called a \emph{canonical  basis}.
Consider  a  collection of smooth simple closed curves representing  the chosen basis. Denote them by $\alpha_i, \beta_i$ respectively.
Perturbing the curves $ \alpha_i, \beta_i$ if necessary, we can make them avoid singularities of the metric.

When  all   zeroes   of    Abelian   differential   $\omega$
representing the flat structure have only even degrees we can
define a \emph{parity of the spin structure}
\begin{equation}
\label{eq:parity:of:spin}
\phi(S)\dfn \sum_{i=1}^g
(\ind( \alpha_i)+1)(\ind( \beta_i)+1)\quad (mod\ 2).
\end{equation}

It follows from the results of  D. Johnson \cite{J} that the
parity of the  spin structure $\phi(S)$  depends neither on  the
choice of representatives nor on the  choice  of  the  canonical
homology bases.

On the other hand, by results of M. Atiyah \cite{At}
the quantity $\phi(S)$ expressed in different terms is invariant under continuous  deformations of the
flat surface inside the stratum,
which implies that it is an invariant of a connected component of the stratum (see \cite{KZ03} for details).

\subsection*{Generating family of cycles associated to an interval exchange map.} Having an irreducible permutation $\pi$ one can construct an interval exchange map $T\colon X\to X$, a flat surface $S$
(suspension over $T$) and an embedding of\/ $X$ into a horizontal leaf on $S$ such that
the first-return map of the vertical flow to the horizontal segment $X$ induces the interval-exchange transformation $T\colon X\to X$ with permutation $\pi$ (see \cite{M82,V82} and Section \ref{ss:Interval:exchange:transformations:and:Rauzy:classes}).

For every interval $X_i$ being exchanged consider a leaf of the vertical foliation launched at some interior point of\/ $X$ and follow it till the first return to $X$. Join the endpoints of the resulting vertical curve along the interval $X$. It is easy to smoothen the resulting closed path to make it everywhere transverse to the horizontal direction, see Figure \ref{fig:iet:cycles}. Denote the resulting smooth simple closed curve by $\gamma_i$. Since $\gamma_i$ is everywhere transverse to the horizontal direction, we get
\begin{equation}
\label{eq:ind:gamma:i:0}
\ind(\gamma_i)=0.
\end{equation}
We denote by $c_i$ the cycle represented by the oriented curve $\gamma_i$.

\begin{figure}[htb]
\includegraphics{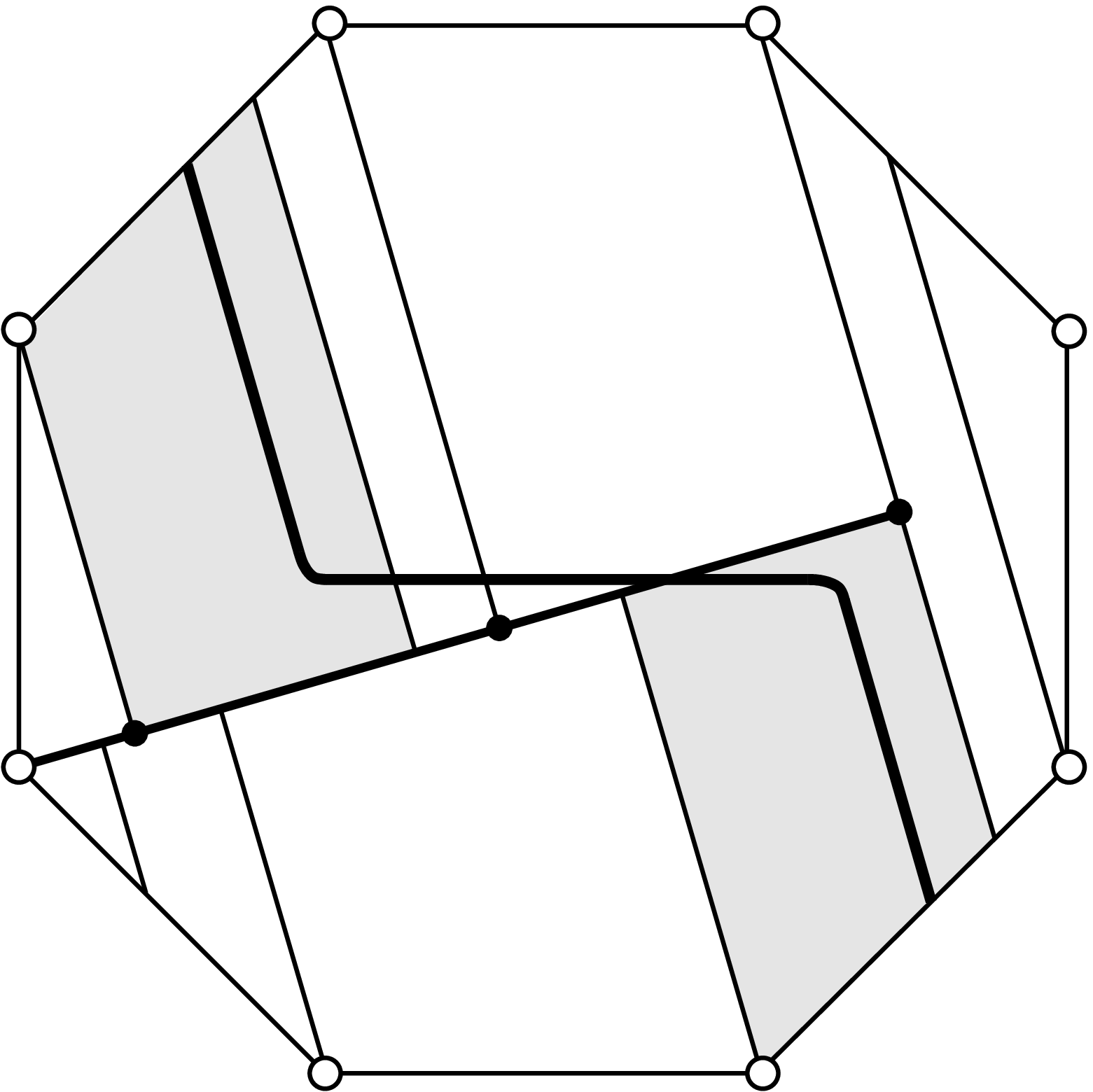}
\begin{picture}(0,0)(0,0)
\put(-47,-41){$\gamma_2$}
\end{picture}
\vspace{130bp}
\caption{
\label{fig:iet:cycles}
Construction of a generating family of cycles associated to an interval-exchange transformation.
}
\end{figure}

\begin{Lemma}
\label{lm:generating:family}
If the vertical foliation on $S$ is minimal, then the cycles $\{c_1, \dots, c_n\}$ form a generating family in $H_1(S,\Z{})$. None of them is trivial.
\end{Lemma}

\begin{proof}We leave the proof as an exercise to the reader.\end{proof}

\subsection*{Algebraic definition.}
Given a permutation $\pi$ we define the following skew-sym\-met\-ric
matrix  $\Omega(\pi)$:
\begin{equation}
\Omega_{ij}(\pi)=\left\{
\begin{array}{rl}
-1 & \text{if\/ $i<j$ and $\pi^{-1}(i)>\pi^{-1}(j)$}\\
1 & \text{if\/ $i>j$ and $\pi^{-1}(i)<\pi^{-1}(j)$}\\
0 & \text{otherwise}
\end{array}
\right.
\end{equation}

\begin{Lemma} Let $T\colon X\to X$ be the interval-exchange transformation induced by the vertical flow on a surface $S$. Let $\pi$ be the associated permutation, and let $c_1, \dots, c_n$ be the ``first return cycles'' as in Figure \ref{fig:iet:cycles}.
The matrix $\Omega(\pi)$ defines the intersection numbers of the cycles $\{c_1, \dots, c_n\}$:
$$
c_i\cdot c_j =\Omega_{ij}(\pi)
$$
\end{Lemma}
\begin{proof}
The proof immediately follows from Figure \ref{fig:iet:cycles}.
\end{proof}

From now on it is convenient to pass to the field $\Zz$.
Note that the intersection form $\Omega$ becomes symmetric over $\Zz$.

Define a function $\Phi\colon H_1(S,\Zz)\to \Zz$ as follows. Fixing a cycle $c$, represent it by a simple closed curve $\gamma$. Deforming $\gamma$ if necessary, we may assume it avoids singularities of the metric. Let
\begin{equation}
\label{eq:definition:Phi:c}
\Phi(c)\dfn \ind(\gamma)+1 \mod 2
\end{equation}

\noindent
The following theorem is a corollary of a corresponding theorem of D. Johnson, see \cite{J}.

\begin{NNTheorem}
The function $\Phi$ is well-defined on $H_1(S,\Zz)$. It is a quadratic form associated to the bilinear form $\Omega$ in the following sense:
\begin{equation}
\label{eq:Phi}
\Phi(c+c')=\Phi(c)+\Phi(c')+\Omega(c,c')
\end{equation}
\end{NNTheorem}

Recall that $\Omega$ is the intersection form, $\Omega(c,c')=c\cdot c'$. Note also that the parity $\phi(S)$ of the spin structure defined by \eqref{eq:parity:of:spin} is expressed in terms of the quadratic form $\Phi$:
$$
\phi(S)=\sum_{i=1}^g \Phi(a_i)\,\Phi(b_i)
$$
where  $\{a_1, b_1, a_2, b_2, \dots, a_g, b_g\}$ is a canonical basis of cycles in $H_1(S,\Zz)$.
In  other  words, the parity $\phi(S)$ of the spin structure of a flat surface $S$ associated to the permutation $\pi$ is  equal to the
Arf-invariant of the quadratic form $\Phi$ on
$H_1(S,\Zz)$ (see \cite{J,Arf}).

\subsection*{Orthogonalization.}
The problem of evaluation of\/ $\phi(S)$ can be reformulated now in terms
of linear algebra over the field $\Zz$. We have a generating family of
vectors $\{c_1, \dots, c_n\}$ in $H_1(S,\Zz)\simeq \left(\Zz\right)^{2g}$.
By the construction suggested by Figure \ref{fig:iet:cycles}, all these
cycles can be represented by curves $\gamma_i$ that are everywhere
transverse to the horizontal direction. Thus, applying
\eqref{eq:ind:gamma:i:0} to definition \eqref{eq:definition:Phi:c} of\/
$\Phi(c_i)$ we obtain the values of\/ $\Phi$ on the vectors from this
family:
$$
\Phi(c_i)=1   \qquad\text{for } i=1, \dots, n
$$

We also know the values of the symplectic bilinear form $\Omega$ on every
pair of vectors in this family. It remains to apply an orthogonalization
procedure to construct a canonical basis $\{a_1, b_1, a_2, b_2, \dots, a_g,
b_g\}$. As we construct new vectors, we can keep track of the values of the
quadratic form $\Phi$ on them by repeatedly applying formula
\eqref{eq:Phi}. Since we work over the field $\Zz$ the formulae are
especially simple.

We let $a_1\dfn c_1$. It follows from Lemma \ref{lm:generating:family} that we can find an index $j$, where $1\le j\le n$, such that $\Omega_{1j}\neq 0$. We let $b_1\dfn c_j$. Now we modify the remaining vectors to make them orthogonal to $a_1, b_1$. Recall that we work over $\Zz$.
$$
c'_i\dfn c_i + \Omega_{ij} c_1 + \Omega_{i1} c_j \quad\text{for }i=2,\dots,n,\quad\text{ where } i\neq j.
$$
Using \eqref{eq:Phi} and simplifying the resulting expression we compute the values of\/ $\Phi$ on these new vectors:
$$
\Phi(c'_i)\dfn \Phi(c_i) + \Omega_{ij}\Phi(c_1)+ \Omega_{i1}\Phi(c_j) + \Omega_{i1} \Omega_{ij}.
$$
Finally we compute the intersection matrix for the resulting family
of vectors $\{c'_2, \dots, c'_{j-1},c'_{j+1}, \dots, c'_n\}$. After simplification, we get:
$$
c'_k\cdot c'_l = \Omega_{kl} + \Omega_{k1} \Omega_{lj} +
\Omega_{kj} \Omega_{l1}.
$$

Now let us remove from our family those vectors $c'_i$ which are orthogonal to all other vectors of the family. If the resulting family is nonempty, reenumerate the vectors from this new family and proceed by induction.

\section{Rauzy  classes  for  small genera}
\label{a:tables}

For genus $g=3$ the stratum of Abelian differentials of maximal dimension,
\ie the principal stratum, has dimension 9. Thus, using extended Rauzy
classes to describe connected components of the strata we deal with
permutations of at most 9 elements, provided we study genera $g=2$ and
$g=3$. The number of such permutations is small enough to construct the
Rauzy classes explicitly (using a computer, of course).

Note that for the stratum  $\mathcal{H}(4)$  in genus $3$, the presence  of  two
different  extended  Rauzy  classes  was  proved  by W. A. Veech  in \cite{V90};
P. Arnoux proved that there are  three  different  extended  Rauzy
classes  corresponding  to  the   stratum   $\mathcal{H}(6)$.

In the tables below we  present the list of  all  Rauzy  classes
determined  by   nondegenerate ``true''  permutations  of  at  most  $9$ elements and by ``generalized'' permutations of at most $6$ elements, see \cite{KZ97}. We indicate hyperellipticity and parity of the spin
structure of  corresponding  components  when  they  are defined.
Horizontal lines separate extended Rauzy classes.
For ``true'' permutations $\pi$ we present only the second line in the canonical enumeration
$$
\begin{pmatrix}
\medskip
1&2&\dots&n\\
\pi^{-1}(1)&\pi^{-1}(2)&\dots&\pi^{-1}(n)
\end{pmatrix}.
$$

Recall that a (generalized) permutation representing a stratum
$\cH(d_1,\dots,d_\noz)$ or $\cQ(d_1,\dots,d_\noz)$ has a singularity of degree $d_1$ at the left endpoint of\/ $X$. This is the reason why, say, the stratum $\cQ(2,1,-1^3)$ appears in our table three times
(see also Conjecture \ref{conj:extended:equals:k:Rauzy:calsses}).

The reader can find a \emph{Mathematica} script generating a Rauzy class
and an extended Rauzy class from a given generalized permutation on the
author's web site. The same web page contains scripts realizing most of the
algorithms described in the present paper, including the ones which
construct cylindrical generalized permutations defining Jenkins--Strebel
differentials with a single cylinder representing a given connected
component of a given stratum of Abelian or quadratic differentials.

\begin{NNRemark}
By convention we identify generalized permutations which can be obtained
one from another by reenumerations. Studying the dynamics of Rauzy
induction it is sometimes more natural to avoid reenumeration, see
\cite{MMY,Y03,Y05}.

However, under the second convention the Rauzy classes become larger and
require more space in computer experiments. The cardinalities of Rauzy
classes in the tables below are given under the first convention.
\end{NNRemark}

\pagebreak

\vspace*{-0.9truecm}

{\small
$$
\begin{array}{|c|c|c|c|}
\multicolumn{4}{c}{\text{Strata of Abelian differentials}}\\
[-\halfbls]
\multicolumn{4}{c}{}\\
\multicolumn{4}{c}{\text{Genus }g=2}\\
\hline &&& \\ [-\halfbls]
\multicolumn{1}{|c|}{\text{Representative}}&
\text{Cardinality}&
\text{Degrees}&
\text{Hyperelliptic}\\
\multicolumn{1}{|c|}{\text{of Rauzy class}}&
\text{of Rauzy class}&
\text{of zeros}&
\text{or spin structure}\\
[-\halfbls] &&&\\
\hline &&& \\ [-\halfbls]
(4,3,2,1)          & 7        & (2)      &\textit{hyperelliptic} \\
[-\halfbls] &&&\\ \hline &&& \\ [-\halfbls]
(5,4,3,2,1)        & 15       & (1,1)    &\textit{hyperelliptic} \\
[-\halfbls] &&&\\
\hline
\multicolumn{4}{c}{}\\
\multicolumn{4}{c}{\text{Genus }g=3}\\
[-\halfbls]
\multicolumn{4}{c}{}\\
\hline &&& \\ [-\halfbls]
(6,5,4,3,2,1)      & 31      & (4)      &\textit{hyperelliptic} \\
[-\halfbls] &&&\\ \hline &&& \\ [-\halfbls]
(4, 3, 6, 5, 2, 1)     & 134      & (4)      &\textit{odd} \\
[-\halfbls] &&&\\ \hline &&& \\ [-\halfbls]
(4, 3, 7, 6, 5, 2, 1)   & 509     & (3, 1)   &\\
[-\halfbls]       &         & & -\\[-\halfbls]
(5, 4, 3, 7, 6, 2, 1)   & 261     & (1, 3)   &\\
[-\halfbls] &&&\\ \hline &&& \\ [-\halfbls]
(7,6,5,4,3,2,1)    & 63     &  (2, 2)   &\textit{hyperelliptic} \\
[-\halfbls] &&&\\ \hline &&& \\ [-\halfbls]
(4, 3, 5, 7, 6, 2, 1)   & 294     & (2, 2)   &\textit{odd}  \\
[-\halfbls] &&&\\ \hline &&& \\ [-\halfbls]
(5, 4, 3, 8, 7, 6, 2, 1) & 1258    & (1, 2, 1)&\\
[-\halfbls]       &         & & - \\[-\halfbls]
(4, 3, 5, 8, 7, 6, 2, 1) & 919     & (2, 1, 1)&\\
[-\halfbls] &&&\\ \hline &&& \\ [-\halfbls]
(5, 4, 3, 6, 9, 8, 7, 2, 1)    & 1255     & (1,1,1,1)& - \\
[-\halfbls] &&&\\
\hline
\multicolumn{4}{c}{}\\
\multicolumn{4}{c}{\text{Genus }g=4}\\
[-\halfbls]
\multicolumn{4}{c}{}\\
\hline &&& \\ [-\halfbls]
(8,7,6,5,4,3,2,1)     & 127   & (6)      &\textit{hyperelliptic} \\
[-\halfbls] &&&\\
\hline &&& \\ [-\halfbls]
(6, 5, 4, 3, 8, 7, 2, 1)     & 2327  & (6)      &\textit{even} \\
[-\halfbls] &&&\\
\hline &&& \\ [-\halfbls]
(4, 3, 6, 5, 8, 7, 2, 1)     & 5209  & (6)      &\textit{odd} \\
[-\halfbls] &&&\\
\hline &&& \\ [-\halfbls]
(5, 4, 3, 7, 6, 9, 8, 2, 1)   & 10543 & (1,5)    &     \\
[-\halfbls]           &       &          & -   \\[-\halfbls]
(4, 3, 6, 5, 9, 8, 7, 2, 1)   & 31031 & (5,1)    &     \\
[-\halfbls] &&&\\
\hline &&& \\ [-\halfbls]
(7, 6, 5, 4, 3, 9, 8, 2, 1)   & 3954  & (2,4)    &     \\
[-\halfbls]           &       &          & \textit{even} \\[-\halfbls]
(6, 5, 4, 3, 7, 9, 8, 2, 1)   & 6614  & (4,2)    &     \\
[-\halfbls] &&&\\
\hline &&& \\ [-\halfbls]
(4, 3, 5, 7, 6, 9, 8, 2, 1)   & 8797  & (2,4)    &     \\
[-\halfbls]           &       &          & \textit{odd}  \\[-\halfbls]
(4, 3, 6, 5, 7, 9, 8, 2, 1)   & 14709 & (4,2)    &     \\
[-\halfbls] &&&\\
\hline &&& \\ [-\halfbls]
(9,8,7,6,5,4,3,2,1)   & 255   & (3,3)    &\textit{hyperelliptic} \\
[-\halfbls] &&&\\
\hline &&& \\ [-\halfbls]
(4, 3, 7, 6, 5, 9, 8, 2, 1)   & 15568 & (3,3)    & - \\
[-\halfbls] &&&\\
\hline
\dots & \dots & \dots & \dots \\
\hline\end{array}
$$
}

\vspace*{-.5truecm}
\centerline{\small Strata of quadratic differentials}
\vspace*{-.5truecm}
{\small
$$
\begin{array}{|c|c|c|}
\multicolumn{3}{c}{\text{Genus }g=0}\\
[-\halfbls]
\multicolumn{3}{c}{}\\
\hline && \\ [-\halfbls]
\begin{pmatrix}1, 2, 2\\3, 3, 1\end{pmatrix} & 4 & (-1^4)\\
[-\halfbls] &&\\
\hline && \\ [-\halfbls]
\begin{pmatrix}1, 2, 2\\3, 3, 4, 4, 5, 5, 1\end{pmatrix} & 10 & (1, -1^5)\\[-\halfbls] &&\\
\begin{pmatrix}1, 2, 3, 3, 4, 4, 2\\5, 5, 1\end{pmatrix} & 22 & (-1^5, 1)\\
[-\halfbls] &&\\
\hline && \\ [-\halfbls]
\begin{pmatrix}1, 2, 2\\3, 3, 4, 4, 5, 5, 6, 6, 1\end{pmatrix} & 13 & (2, -1^6)\\[-\halfbls] &&\\
\begin{pmatrix}1, 2, 3, 3, 4, 4, 5, 5, 2\\6, 6, 1\end{pmatrix} & 28 & (-1^6, 2)\\
[-\halfbls] &&\\
\hline
\multicolumn{3}{c}{}\\
\multicolumn{3}{c}{\text{Genus }g=1}\\
[-\halfbls]
\multicolumn{3}{c}{}\\
\hline && \\ [-\halfbls]
\begin{pmatrix}1, 2, 3, 3\\2, 4, 4, 1\end{pmatrix} & 43 & (2, -1^2)\\[-\halfbls] &&\\
\begin{pmatrix}1, 2, 3, 3\\4, 2, 4, 1\end{pmatrix} & 20 & (-1^2, 2)\\
[-\halfbls] &&\\
\hline && \\ [-\halfbls]
\begin{pmatrix}1, 2, 3, 3, 4, 4\\2, 5, 5, 1\end{pmatrix} & 198 & (3, -1^3)\\[-\halfbls] &&\\
\begin{pmatrix}1, 2, 3, 3, 4, 4\\5, 2, 5, 1\end{pmatrix} & 120 & (-1^3, 3)\\
[-\halfbls] &&\\
\hline && \\ [-\halfbls]
\begin{pmatrix}1, 2, 3, 3, 4, 4, 5, 5\\2, 6, 6, 1\end{pmatrix} & 596 & (4, -1^4)\\[-\halfbls] &&\\
\begin{pmatrix}1, 2, 3, 3, 4, 4, 5, 5\\6, 2, 6, 1\end{pmatrix} & 440 & (-1^4, 4)\\
[-\halfbls] &&\\
\hline && \\ [-\halfbls]
\begin{pmatrix}1, 2, 3, 4, 4\\3, 2, 5, 5, 1\end{pmatrix} & 128 & (1^2, -1^2)\\[-\halfbls] &&\\
\begin{pmatrix}1, 2, 3, 4, 4\\5, 3, 2, 5, 1\end{pmatrix} & 34 & (-1^2, 1^2)\\
[-\halfbls] &&\\
\hline && \\ [-\halfbls]
\begin{pmatrix}1, 2, 3, 3, 4, 5, 5\\4, 2, 6, 6, 1\end{pmatrix} & 714 & (2, 1, -1^3)\\[-\halfbls] &&\\
\begin{pmatrix}1, 2, 3, 4, 4, 5, 5\\3, 2, 6, 6, 1\end{pmatrix} & 514 & (1, 2, -1^3)\\[-\halfbls] &&\\
\begin{pmatrix}1, 2, 3, 3, 4, 5, 5\\6, 4, 2, 6, 1\end{pmatrix} & 510 & (-1^3, 2, 1)\\
[-\halfbls] &&\\
\hline
\end{array}
\qquad
\begin{array}{|c|c|c|c|}
\multicolumn{4}{c}{\text{Genus }g=2}\\
[-\halfbls]
\multicolumn{4}{c}{}\\
\hline &&& \\ [-\halfbls]
\begin{pmatrix}1, 2, 3, 2, 4\\4, 5, 5, 3, 1\end{pmatrix} & 440 & (5, -1)&\\[-\halfbls] &&&\\
\begin{pmatrix}1, 2, 3, 2, 4\\5, 3, 4, 5, 1\end{pmatrix} & 54 & (-1, 5)&\\
[-\halfbls] &&&\\
\hline &&& \\ [-\halfbls]
\begin{pmatrix}1, 2, 3, 2, 4\\4, 5, 5, 6, 6, 3, 1\end{pmatrix} & 4832 & (6, -1^2)&\textit{non}\\[-\halfbls] &&&\\
\begin{pmatrix}1, 2, 3, 2, 4\\5, 3, 4, 6, 6, 5, 1\end{pmatrix} & 1118 & (-1^2, 6)&\textit{non}\\
[-\halfbls] &&&\\
\hline &&& \\ [-\halfbls]
\begin{pmatrix}1, 2, 2, 3, 4, 5\\5, 4, 3, 6, 6, 1\end{pmatrix} & 347 & (6, -1^2)&\textit{hyp}\\[-\halfbls] &&&\\
\begin{pmatrix}1, 2, 3, 4, 5, 2\\6, 5, 4, 3, 6, 1\end{pmatrix} &  60 & (-1^2, 6)&\textit{hyp}\\
[-\halfbls] &&&\\
\hline &&& \\ [-\halfbls]
\begin{pmatrix}1, 2, 3, 2, 4\\5, 4, 5, 3, 1\end{pmatrix} & 73 & (2^2)&\\
[-\halfbls] &&&\\
\hline &&& \\ [-\halfbls]
\begin{pmatrix}1, 2, 3, 2, 4\\5, 4, 5, 6, 6, 3, 1\end{pmatrix} & 1666 & (3, 2, -1)&\\[-\halfbls] &&&\\
\begin{pmatrix}1, 2, 3, 2, 4\\5, 4, 6, 6, 5, 3, 1\end{pmatrix} & 1348 & (2, 3, -1)&\\[-\halfbls] &&&\\
\begin{pmatrix}1, 2, 3, 2, 4\\5, 3, 6, 4, 6, 5, 1\end{pmatrix} & 294 & (-1, 3, 2)&\\
[-\halfbls] &&&\\
\hline &&& \\ [-\halfbls]
\begin{pmatrix}1, 2, 3, 2, 4, 5\\5, 4, 6, 6, 3, 1\end{pmatrix} & 2062 & (4, 1, -1)&\\[-\halfbls] &&&\\
\begin{pmatrix}1, 2, 3, 4, 2, 5\\3, 5, 6, 6, 4, 1\end{pmatrix} & 1076 & (1, 4, -1)&\\[-\halfbls] &&&\\
\begin{pmatrix}1, 2, 3, 2, 4, 5\\6, 3, 5, 4, 6, 1\end{pmatrix} & 260 & (-1, 4, 1)&\\
[-\halfbls] &&&\\
\hline &&& \\ [-\halfbls]
\begin{pmatrix}1, 2, 3, 2, 4, 5\\6, 5, 4, 6, 3, 1\end{pmatrix} & 125 & (2, 1^2)&\\[-\halfbls] &&&\\
\begin{pmatrix}1, 2, 3, 4, 2, 5\\3, 6, 5, 6, 4, 1\end{pmatrix} & 220 & (1^2, 2)&\\
[-\halfbls] &&&\\
\hline
\multicolumn{4}{c}{}\\
\multicolumn{4}{c}{\text{Genus }g=3}\\
[-\halfbls]
\multicolumn{4}{c}{}\\
\hline &&& \\ [-\halfbls]
\begin{pmatrix}1, 2, 3, 2, 3, 4\\5, 6, 5, 6, 4, 1\end{pmatrix} & 2590 & (8)&\\
[-\halfbls] &&&\\
\hline
\end{array}
$$
}

\acknowledgment

This paper has grown from experiments performed at IHES in 1995--1996 by {M. Kontsevich} and the author, see \cite{KZ97}. In particular, the initial version of this paper was planned as an appendix to the paper \cite{KZ03}.

I would like to thank Maxim Kontsevich for numerous valuable discussions of this subject; his ideas were at the origin of the study of the
combinatorics, geometry and dynamics of generalized interval-exchange transformations.

I am very much indebted to Corentin Boissy and to Erwan Lanneau who have constructed a rigorous theory of generalized permutations and of their Rauzy classes. Without their work the present paper would make no sense.

I would like to thank Giovanni Forni, Boris Hasselblatt, Erwan Lanneau and the referee for an extremely careful reading of the manuscript and for numerous
helpful comments.

I highly appreciate hospitality of Institut des Hautes \'Etudes Scientifiques and of Max--Planck--Institut f\"ur Mathematik at Bonn while preparing this paper.


\begin{thebibliography}{99}

\bibitem{Arf} (0008069)
C. Arf, \emph{Untersuchungen  \"uber  quadratische  Formen   in
K\"opern  der Charakteristik 2. I.}, J. Reine Angew. Math.,
\textbf{183} (1941), 148--167.

\bibitem{At} (0286136)
M. Atiyah, \emph{Riemann surfaces and spin structures}, Ann. scient.
\'Ecole Norm. Sup. (4), \textbf{4} (1971), 47--62.

\bibitem{AGY} (2264836)
A. Avila, S. Gouezel and J.-C. Yoccoz, \emph{Exponential mixing for
the Teichmüller flow}, Publications Mathèmatiques de l'IHES,
\textbf{104} (2006), 143--211.

\bibitem{AF} (2299743)
A. Avila and G. Forni, \emph{Weak mixing for interval-exchange transformations and translation flows}, Annals of Math.,
\textbf{165} (2007), 637--664.

\bibitem{AV} (2316268)
A. Avila and M. Viana, \emph{Simplicity of Lyapunov spectra: proof
of the Zorich--Kontsevich conjecture}, Acta Math., \textbf{198}
(2007), 1--56.

\bibitem{BL}
C. Boissy and E. Lanneau, \emph{On Generalized interval exchange
maps: Dynamics and geometry of the Rauzy--Veech induction}, E-print
at \arXiv:math.GT/0710.5614.

\bibitem{Bf} (2220100)
A. Bufetov, \emph{Decay of Correlations for the Rauzy--Veech--Zorich
Induction Map on the Space of Interval Exchange Transformations and
the Central Limit Theorem for the Teichmueller Flow on the Moduli
Space of Abelian Differentials}, Journal of AMS, \textbf{19} (2006),
579--623.

\bibitem{DN} (1055445)
C. Danthony and A. Nogueira, \emph{Measured foliations on
nonorientable surfaces}, Ann. Scient. \'ENS, \textbf{23} (1990),
469--494.

\bibitem{DH} (0396936)
A. Douady and J. Hubbard, \emph{On the density of strebel
differentials}, Inventiones Mathematicae, \textbf{30} (1975),
175--179.

\bibitem{Eskin:Masur:Zorich} (2010740)
A. Eskin, H. Masur and A. Zorich, \emph{Moduli spaces  of  Abelian
differentials: the principal boundary, counting problems and the
Siegel--Veech  constants}, Publications Mathèmatiques de l'IHES,
\textbf{97} (2003), 61--179.

\bibitem{HM} (0523212)
J. Hubbard and H. Masur, \emph{Quadratic differentials and
foliations}, Acta Mathematica, \textbf{142} (1979), 221--274.

\bibitem{J} (0588283)
D. Johnson, \emph{Spin structures and quadratic forms on surfaces},
J. London Math. Soc. (2), \textbf{22} (1980), 365--373.

\bibitem{Kontsevich:Airy:Function} (1171758)
M. Kontsevich, \emph{Intersection theory on the moduli space of
curves and the matrix Airy function}, Communications in Math. Phys.,
\textbf{147} (1992), 1--23.

\bibitem{Kontsevich} (1490861)
M. Kontsevich, \emph{Lyapunov exponents and Hodge theory}, ``The
mathematical beauty of physics'' (Saclay, 1996), (in Honor of C.
Itzykson), Adv. Ser. Math. Phys., \textbf{24}, World Sci.
Publishing, River Edge, NJ, (1997), 318--332.

\bibitem{KZ97}
M. Kontsevich and  A. Zorich,  \emph{ Lyapunov exponents and Hodge theory},
Preprint IHES M/97/13, pp.\ 1--16; \arXiv:hep-th/9701164.

\bibitem{KZ03} (2000471)
M. Kontsevich and  A. Zorich,  \emph{Connected components of the
moduli spaces of Abelian differentials with prescribed
singularities}, Inventiones mathematicae, \textbf{153} (2003),
631--678.

\bibitem{Lanneau:hyperelliptic} (2081723)
E. Lanneau, \emph{Hyperelliptic components of the moduli spaces of
quadratic differentials with prescribed singularities}, Comment.
Math. Helv., {\bf 79} (2004), 471--501.

\bibitem{L}
E. Lanneau, \emph{Connected components of the strata of the moduli
spaces of quadratic differentials}, to appear in Annales de l'ENS; E-print at \arXiv:math.GT/0506136.

\bibitem{MMY} (2163864)
S. Marmi, P. Moussa and J.-C. Yoccoz \emph{The cohomological
equation for Roth-type interval exchange maps}, Journal of AMS,
\textbf{18} (2005), 823--872.

\bibitem{M79} (0535053)
H. Masur, \emph{The Jenkins--Strebel differentials with one cylinder
are dense}, Commentarii Mathematici Helvetici, \textbf{54} (1979),
179--184.

\bibitem{M82} (0644018)
H. Masur, \emph{Interval-exchange transformations  and   measured
foliations}, Annals of Math. (2), \textbf{115} (1982), 169--200.

\bibitem{MS} (1214233)
H. Masur and J. Smillie, \emph{Quadratic differentials with
prescribed singularities and pseudo-Anosov diffeomorphisms},
Comment. Math. Helvetici, \textbf{68} (1993), 289--307.

\bibitem{MZ}
H. Masur and A. Zorich,
\emph{Multiple saddle  connections on  flat  surfaces and
principal boundary of the moduli spaces of quadratic differentials},
to appear in GAFA (2008); Eprint in \arXiv:math.GT/0402197

\bibitem{Rauzy} (0543205)
G. Rauzy, \emph{Echanges d'intervalles et transformations induites},
Acta Arith., \textbf{34} (1979), 315--328.

\bibitem{V82} (0644019)
W. A. Veech, \emph{Gauss  measures for transformations  on the space
of interval exchange maps}, Annals of Mathematics, \textbf{115}
(1982), 201--242.

\bibitem{V86} (0866707)
W. A. Veech, \emph{The Teichm\"uller geodesic flow}, Annals of
Mathematics (2), \textbf{124} (1986), 441--530.

\bibitem{V90} (1094714)
W. A. Veech, \emph{Moduli spaces of quadratic differentials},
Journal Analyse Math., \textbf{55} (1990), 117--171.

\bibitem{Veech:hyperelliptic:I} (1402493)
W. A. Veech, \emph{Geometric realizations of hyperelliptic curves},
Algorithms, fractals and dynamics (Okayama/Kyoto 1992), 217--226,
Plenum, New-York, 1995.

\bibitem{Vi}
M. Viana,
\emph{Dynamics of interval exchange maps and Teichmüller flows},
Lecture notes of graduate courses taught at IMPA in 2005 and 2007,
Preprint.

\bibitem{Y03} (2261103)
J.-C. Yoccoz, \emph{Continued fraction algorithms for interval
exchange maps: an introduction}, in collection ``Frontiers in Number
Theory, Physics and Geometry. Vol. 1: On random matrices, zeta
functions and dynamical systems''; Ecole de physique des Houches,
France, March 9--21 2003, P. Cartier; B. Julia; P. Moussa; P.
Vanhove (Editors), Springer-Verlag, Berlin, 2006, 401--435.

\bibitem{Y05}
J.-C. Yoccoz,
\emph{Echanges d'intervalles},
Notes de cours donn\'e au Coll\`ege de France en
2005, Preprint.

\bibitem{Z96} (1393518)
A. Zorich \emph{Finite Gauss measure on the space of interval-exchange transformations. Lyapunov exponents}, Annals Inst. Fourier
(Grenoble), \textbf{46} (1996), 325--370.

\bibitem{Z:Houches}
A. Zorich, \emph{Flat surfaces}, in collection ``Frontiers in Number
Theory, Physics and Geometry. Vol. 1: On random matrices, zeta
functions and dynamical systems''; Ecole de physique des Houches,
France, March 9--21 2003, P. Cartier; B. Julia; P. Moussa; P.
Vanhove (Editors), Springer-Verlag, Berlin, 2006, 439--586.

\end{thebibliography}
\end{document}